\documentclass[onefignum,onetabnum]{siamonline250211}

\usepackage{lipsum}
\usepackage{amsfonts}
\usepackage{graphicx}
\usepackage{epstopdf}
\usepackage{algorithmic}
\ifpdf
  \DeclareGraphicsExtensions{.eps,.pdf,.png,.jpg}
\else
  \DeclareGraphicsExtensions{.eps}
\fi

\usepackage{braket}

\usepackage{cite}

\usepackage{pifont}

\newcommand{\bm}{\mathbf}
\newcommand{\der}{\text{d}}
\newcommand{\R}{\mathbb{R}}
\newcommand{\Exp}{\text{Exp}}
\newcommand{\E}{\mathbb{E}}
\newcommand{\Var}{\text{Var}}

\newcommand{\ltot}{\lambda_{\mathrm{tot}}}

\newcommand{\gi}{g^{(i)}}

\newcommand{\Ei}{E^{(i)}}
\newcommand{\Ai}{A^{(i)}}
\newcommand{\Ji}{J^{(i)}}
\newcommand{\Jiapp}{\widetilde{J}^{(i)}}
\newcommand{\etai}[2]{\eta^{(i),\mathrm{#1}}_{#2}}

\newcommand{\hsup}[1]{H^{(#1)}}
\newcommand{\hsupa}[1]{$H^{(#1)}$}
\newcommand{\nsup}[1]{\bm{N}^{(#1)}}

\newcommand{\nsupapp}[1]{\widetilde{\bm{N}}^{(#1)}}

\newcommand{\nsupsub}[2]{N^{(#1)}_{#2}}

\newcommand{\psisup}[1]{\boldsymbol{\psi}^{(#1)}}
\newcommand{\psisupa}[1]{$\boldsymbol{\psi}^{(#1)}$}
\newcommand{\psisupsub}[2]{\psi^{(#1)}_{#2}}

\newcommand{\psisupapp}[1]{\widetilde{\boldsymbol{\psi}}^{(#1)}}
\newcommand{\psisupsubapp}[2]{\widetilde\psi^{(#1)}_{#2}}
\newcommand{\norminf}[1]{\left\lVert#1\right\rVert_\infty}

\newcommand{\normtwo}[1]{\left\lVert#1\right\rVert_2}

\newcommand{\normLinf}[2]{\left\lVert#1\right\rVert_{L^\infty #2}}

\newcommand{\normLtwoa}[2]{$\left\lVert#1\right\rVert_{L^2 #2}$}

\newcommand{\normod}[1]{\left\lVert#1\right\rVert}

\usepackage{etoolbox}
\definecolor{linkgreen}{rgb}{0,0.5,0} 
\hypersetup{colorlinks=true, linkcolor=linkgreen}

\DeclareRobustCommand{\crefcomma}[1]{%
  \begingroup
    \def\sep{}%
    \textcolor{linkgreen}{\normalfont\textup(}%
    \renewcommand{\do}[1]{%
      \textcolor{linkgreen}{\normalfont\textup\sep\hyperref[##1]{\ref*{##1}}}%
      \def\sep{, }%
    }%
    \docsvlist{#1}%
    \textcolor{linkgreen}{\normalfont\textup)}%
  \endgroup
}

\usepackage{enumitem}
\setlist[enumerate]{leftmargin=.5in}
\setlist[itemize]{leftmargin=.5in}

\newsiamremark{remark}{Remark}
\newsiamremark{hypothesis}{Hypothesis}
\crefname{hypothesis}{Hypothesis}{Hypotheses}
\newsiamthm{claim}{Claim}

\headers{Interacting Hosts with Microbiome Exchange}{M. Johnson and M. A. Porter}

\title{Interacting Hosts with Microbiome Exchange: An Extension of Metacommunity Theory for Discrete Interactions between Hosts\thanks{Submitted to the editors DATE.
\funding{This work was funded by the National Science Foundation (through award 2124903).}}}

\author{Michael Johnson\thanks{Department of Mathematics, University of California, Los Angeles, CA 90095, USA (\email{mcjcard@math.ucla.edu}).} 
\and Mason A. Porter\thanks{Department of Mathematics and Department of Sociology, University of California, Los Angeles, CA 90095, USA; Santa Fe Institute, Santa Fe, NM 87501, USA (\email{mason@math.ucla.edu}).}
}

\usepackage{amsopn}

\ifpdf
\hypersetup{
  pdftitle={Interacting Hosts with Microbiome Exchange},
  pdfauthor={M. Johnson, M. A. Porter}
}
\fi

\begin{document}

\maketitle

\begin{abstract}
Microbiomes, which are collections of interacting microbes in an environment, often substantially impact the environmental patches or living hosts that they occupy. In microbiome models, it is important to consider both the local dynamics within an environment and exchanges of microbiomes between environments. One way to incorporate these and other interactions across multiple scales is to employ metacommunity theory. Metacommunity models commonly assume continuous microbiome dispersal between the environments in which local microbiome dynamics occur. Under this assumption, a single parameter between each pair of environments controls the dispersal rate between those environments. This metacommunity framework is well-suited to abiotic environmental patches, but it fails to capture an essential aspect of the microbiomes of living hosts. Living hosts generally do not interact continuously with each other. Instead, living hosts interact with each other in discrete time intervals. In this paper, we develop a modeling framework that encodes such discrete interactions and uses two parameters to separately control the interaction frequencies between hosts and the amount of microbiome exchange during each interaction. We derive analytical approximations of models in our framework in three parameter regions and prove that they are accurate in these regions. We compare these approximations to numerical simulations for an illustrative model, and we demonstrate that both parameters in our modeling framework are necessary to determine microbiome dynamics. Key features of the dynamics, such as microbiome convergence across hosts, depend sensitively on the interplay between interaction frequency and strength.
\end{abstract}






\section{Introduction} \label{sec:Intro}

The microbiomes of humans and other animals play a critical role in their functioning and health \cite{Hou22,Valdes18,Wang17}, and there is strong evidence that a host's social interactions significantly impact their microbiome composition \cite{mice,AnimalSurvey,baboon}. Therefore, it is important to study ecological modeling frameworks that account simultaneously for microbe-scale dynamics and the effects of host interactions \cite{ADAIR17,Miller2018}.
For example, socially-influenced microbiome signatures are significant indicators of childhood airway development \cite{airway}, communicable-disease resistance \cite{disease_resistance}, and mental health \cite{mental_health}.


\subsection{Models of Local Ecological Dynamics} \label{sec:EMLD}

The dynamics of microbe populations are affected by environmental factors and the abundances of microbe species. A classical approach to study such populations is to analyze phenomenological dynamical-systems models \cite{Keshet}, such as the generalized Lotka--Volterra model \cite{Cui24} 
\begin{equation}
    \frac{\der N_k}{\der t} = r_k N_k + \sum_{l=1}^m \alpha_{kl} N_k N_l \,,
\label{eq:GLV}
\end{equation}
which describes the dynamics of the abundances $N_1(t), \ldots, N_m(t)$ of $m$ 
microbe species. Each species $k$ has an intrinsic birth rate and death rate, which are combined into a single parameter $r_k$. A positive $r_k$ signifies that the birth rate exceeds the death rate 
{and that the population of species $k$ increases in the absence of other microbe species.} 
A negative $r_k$ signifies that the death rate exceeds the birth rate 
{and that the population of species $k$ decreases in the absence of other microbe species.} The cross parameter $\alpha_{kl}$ quantifies the effect of species $l$ on the population of species $k$. A positive $\alpha_{kl}$ indicates that species $l$ is beneficial to species $k$, and a negative $\alpha_{kl}$ indicates that species $l$ is harmful to species $k$.

Researchers also employ mechanistic models, such as consumer--resource models \cite{Cui24}, to describe microbial population dynamics. One class of such models is niche models
\begin{align}\label{eq:CR}
    \frac{\der N_k}{\der t} &= N_k\, A_k(\bm{R}) \,, \\
    \frac{\der R_l}{\der t} &= B_l(\bm{R}) - \sum_{k=1}^m N_k\, C_{kl}(\bm{R}) \nonumber \,.
\end{align}
In a niche model, one tracks both the microbe species abundances $\bm{N} = (N_1(t), N_2(t), \ldots, N_m(t))$ and the resource abundances $\bm{R} = (R_1(t), R_2(t), \ldots, R_n(t))$. The abundance $N_k(t)$ of microbe species $k$ grows or decays according to a growth rate that is a function $A_k(\bm{R})$ of the resource abundances. The abundance $R_l(t)$ of resource $l$ is affected both by the resource abundances and by the consumption of it by microbes. The function $B_l(\bm{R})$ encodes the intrinsic dynamics of the resource abundances. The function $C_{kl}(\bm{R})$ encodes the amount of resource $l$ that species $k$ consumes per unit abundance of species $k$. Niche models capture the fact that microbes affect one another indirectly through resource competition, rather than interacting directly with each other. Niche models also allow researchers to examine environment-dependent cross-species effects.


\subsection{Metacommunity Theory}

The phenomenological and mechanistic models in \Cref{sec:EMLD} assume that all microbe species are in a single 
well-mixed system, which we call a \emph{locale}. 
In a locale, all microbes and all resources are well-mixed, so all dynamics depend only on
the microbe species abundances and resource abundances.
{We refer to models of dynamics within a single locale} as models of \emph{local dynamics}. The specification of local ecological dynamics provides a necessary starting point, but it is also necessary to account for interactions across different {locales}, as such interactions are essential to understand microbiome composition in many settings \cite{mice,AnimalSurvey,baboon}. For example, a coral reef has distinct 
{spatial} {\emph{patches}, which are {locales that represent regions of physical space (rather than, e.g., mobile hosts).}}
Stony coral, sponges, algae, and other biotopes provide different conditions to their respective microbiomes, but
the microbiomes of these patches are not isolated and thus impact each other via microbe dispersal \cite{Sponge}. 

Researchers employ \emph{metacommunity theory}~\cite{MetaSDEC,Leibold04} 
to investigate the effects of multiple-scale interactions that include both local ecological dynamics and interactions across distinct 
{locales}. One relevant framework in metacommunity theory is the \emph{mass-effects paradigm}~\cite{KlausLV,Loreau03,Moquet02,Mouquet03,Thompson20}, in which researchers study systems with local ecological dynamics and dispersal between patches. Models in this framework are often coupled ordinary differential equations (ODEs) of the form
\begin{equation}
    \frac{\der \nsup{i}}{\der t} = g^{(i)}(\nsup{i}) + \sum_j \sigma_{ij} \left(\nsup{j} - \nsup{i}\right) \,,
    \label{eq:meta}
\end{equation}
where $\nsup{i}(t)$ is a vector that encodes the microbe species abundances in patch $i$ at time $t$. For consumer--resource models, one can include resources as separate entries of each vector $\nsup{i}(t)$. The autonomous function $g^{(i)}$ encodes the local dynamics in patch $i$. The parameter $\sigma_{ij}$ governs the dispersal between patches $i$ and $j$.


\subsection{Our Contributions}

Mass-effects models (e.g., see \cref{eq:meta}) fail to capture an essential aspect of the microbiomes of many living hosts.
Many living hosts (such as humans) do not interact continuously \cite{Interaction} and thus do not sustain a continuous dispersal of microbes. Instead, they interact in discrete time intervals. In the present paper, we develop a framework that considers the discrete nature of host interactions. In this framework, when two hosts interact with each other, they instantaneously exchange some of their microbiomes.


\subsection{Organization of our Paper}

Our paper proceeds as follows. In \Cref{sec:Model}, we describe our modeling framework for interacting microbiome hosts. In \Cref{sec:LFA}, we describe the behavior of models in our framework in a parameter region in which hosts interact with low frequency. In \Cref{sec:HFA}, we describe two distinct parameter regions in which hosts interact with high frequency. In \Cref{sec:Num}, we present numerical experiments for our framework. In \Cref{sec:Conc}, we discuss conclusions, limitations, and potential future directions of our work. In the appendices, we prove the accuracy of the approximations in \Cref{sec:LFA,sec:HFA}. Our code{, including scripts and notebooks to run the numerical experiments and generate the figures in the present paper,} is available at \url{https://github.com/mcjcard/Interacting-Hosts-with-Microbe-Exchange.git}.



\section{Our Modeling Framework} \label{sec:Model}

Our modeling framework has three ingredients: an interaction network (see \Cref{sec:IN}), exchange dynamics (see \Cref{sec:exchange}), and local dynamics (see \Cref{sec:LD}). In \cref{tab:glossary}, we summarize our key notation.

\def\arraystretch{1.5}
\begin{table}[htbp]
\footnotesize
\caption{Glossary of our Key Notation}\label{tab:sym}
\begin{center}
  \begin{tabular}{|c|l|} \hline
   Symbols & Definition \\ \hline
    $H$ & Node set, which is the set of microbiome hosts (i.e., nodes) \\
    \hsupa{i} & Microbiome host in $H$ \\
    $E$ & Edge set, which is the set of connections (i.e., edges)
    between hosts \\
    $\left(\hsup{i},\hsup{j}\right)$ & Edge between \hsupa{i} and \hsupa{j} \\
    $\lambda_{ij}$ & Interaction-frequency parameter between hosts \hsupa{i} and \hsupa{j} \\
    $\ltot$ & Total-interaction-frequency parameter $\left(\ltot = \sum_{i=1}^{|H|} \sum_{j=i+1}^{|H|} \lambda_{ij}\right)$ \\
    $l_{ij}$ & Relative interaction-frequency parameter $\left(l_{ij} = \frac{\lambda_{ij}}{\ltot}\right)$ \\
    $\gamma$ & Interaction strength \\
    $\nsup{i}(t)$ & Microbiome abundance vector of host \hsupa{i} \\
    $n$ & Dimension of each microbiome abundance vector $\nsup{i}(t)$ \\
    $\overline{\bm{N}}$ & Mean microbiome abundance vector $\left(\overline{\bm{N}} = \frac{1}{|H|} \sum_{j=1}^{|H|} \nsup{j}\right)$\\
    $\gi$ & Local-dynamics function of host \hsupa{i} \\
    \psisupa{i} & Basin probability vector of host \hsupa{i} \\
    $m_i$ & Dimension of the basin probability vector \psisupa{i} \\
    $\Psi$ & Basin probability tensor \\
    $\Phi^{(ij)}$ & Pairwise interaction operator between hosts \hsupa{i} and \hsupa{j}\\
    $\Phi$ & total-interaction operator \\
    $t^*$ & Frequency-scaled time $\left(t^* = \ltot t\right) $ \\
    \hline
  \end{tabular}
\end{center}
 \label{tab:glossary}
\end{table}


\subsection{Interaction Network} \label{sec:IN}

To study the microbiomes of living hosts, we consider networks that encode the interactions between these hosts. The simplest type of network is a graph $G = (H,E)$, which consists of a set $H$ of nodes and a set $E \subseteq H \times H$ of edges between nodes. See \cite{fblec,Newman} for introductions to networks. In the present paper, we use the terms ``graph" and ``network" interchangeably. Each node of a network represents an entity, and each edge encodes a tie between two entities. In the context of our modeling framework, we refer to each graph as an \emph{interaction network}. The nodes in the node set $H$ are microbiome hosts. The edges in the edge set $E$ encode whether or not two hosts can interact with each other. We represent an undirected edge between two hosts $\hsup{i},\hsup{j}\in H$ by writing either $\left(\hsup{i},\hsup{j}\right)$ or $\left(\hsup{j},\hsup{i}\right)$, which are equivalent. If an edge exists between two hosts, we say that they are \emph{adjacent}.

Each edge $e\in E$ also has an associated weight $\lambda_{e} \in \R^+$. For an edge $e = \left(\hsup{i},\hsup{j}\right)$, we equivalently write $\lambda_e$ or $\lambda_{ij}$. If there is no edge between hosts \hsupa{i} and \hsupa{j}, we set $\lambda_{ij} = 0$. We refer to this weight as the \emph{interaction-frequency parameter} between hosts \hsupa{i} and \hsupa{j}, as it determines the frequency of the interactions between those two hosts. The order of the hosts in indexing is arbitrary, so $\lambda_{ji} = \lambda_{ij}$. Throughout this paper, any symbol with the index order $ij$ is equivalent to that symbol with the reverse index order $ji$. In \cref{fig:Interaction Network}, we show an example of an interaction network with ten hosts. We use this network 
{topology} for our {main} numerical experiments (see \Cref{sec:Num}).

\begin{figure}[htbp]
  \centering
  \includegraphics[scale=0.6]{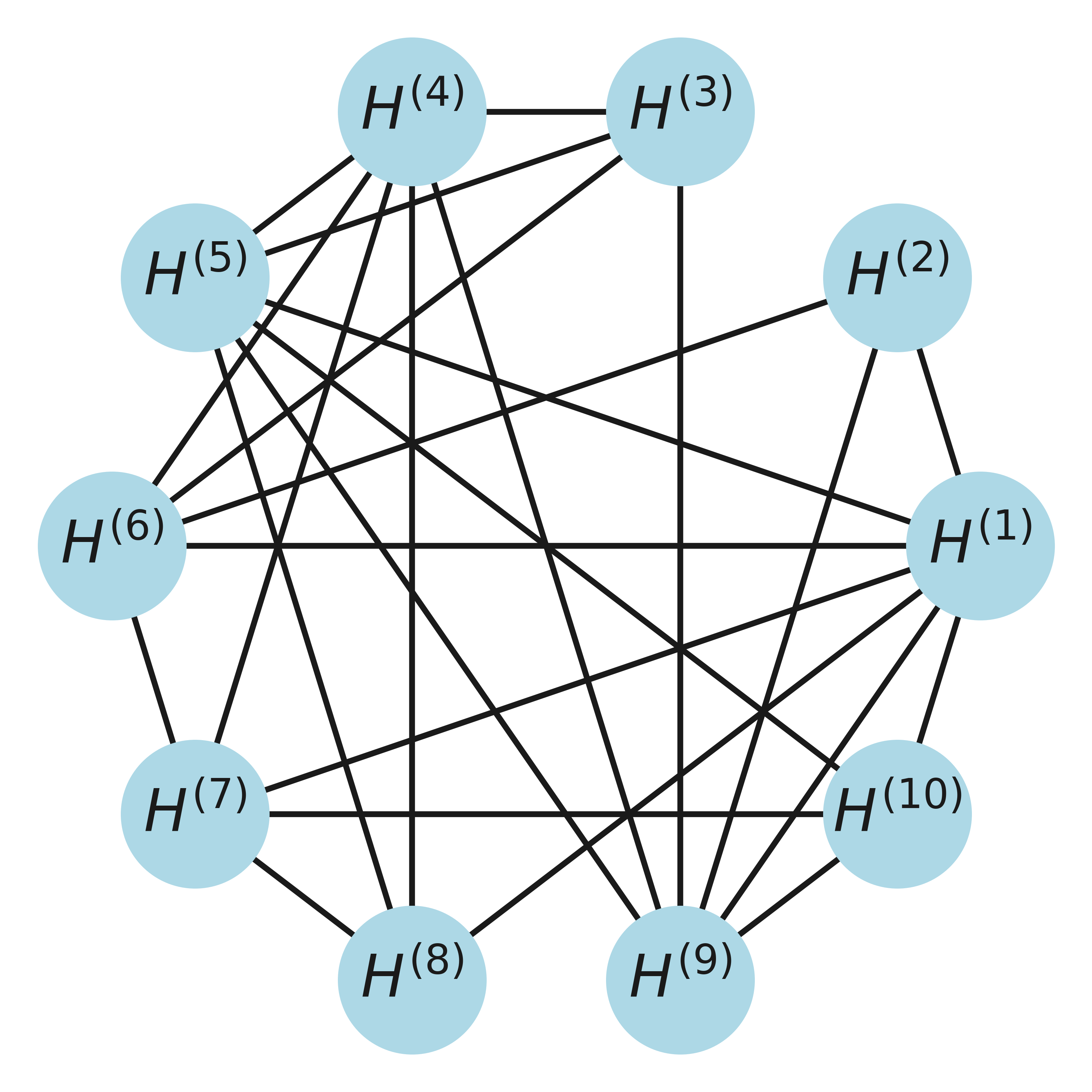}
  \caption{An example of an interaction network with 10 hosts. An edge between two hosts indicates that those two hosts can interact with each other. 
  In this example, all nonzero interaction-frequency parameters $\lambda_{ij}$ have identical values.
  }
  \label{fig:Interaction Network}
\end{figure}


\subsection{Exchange Dynamics}\label{sec:exchange}

Each host $\hsup{i}\in H$ supports a microbiome system. We encode the state of this system with a vector $\nsup{i}(t)$ of microbe species abundances. We refer to $\nsup{i}(t)$ as the \emph{microbiome abundance vector} of host \hsupa{i}. The $k$th entry $\nsupsub{i}{k}(t)$ of the microbiome abundance vector encodes the abundance of microbe species $k$ in host \hsupa{i} at time $t$. We order the microbiome abundance vector of each host so that its $k$th entry describes the same microbe species for each host. The dimension $n$ of all microbiome abundance vectors is the same. If two hosts \hsupa{i} and \hsupa{j} interact at time $t_I$, each host instantaneously exchanges a proportion $\gamma$ of its microbiome with the other host. Let $t_I^-$ denote the limit $t \rightarrow t_I$ from below, and let $t_I^+$ denote the limit $t \rightarrow t_I$ from above. We then write instantaneous microbiome exchange as
\begin{align} \label{eq:exch}
    \nsup{i}\left(t_I^+\right) &= (1 - \gamma) \nsup{i}\left(t_I^-\right) + \gamma \nsup{j}\left(t_I^-\right) \,, \\
    \nsup{j}\left(t_I^+\right) &= (1 - \gamma) \nsup{j}\left(t_I^-\right) + \gamma \nsup{i}\left(t_I^-\right) \,, \nonumber
\end{align}
where the parameter $\gamma$ governs the strength of the interaction between host \hsupa{i} and \hsupa{j}.
{One can define the microbiome abundance vector at the interaction time $t_I$ as either $\nsup{i}\left(t_I\right) = \nsup{i}\left(t_I^-\right)$ or $\nsup{i}\left(t_I\right) = \nsup{i}\left(t_I^+\right)$. This choice is arbitrary and does not affect our modeling framework. In the present paper, we use convention $\nsup{i}\left(t_I\right) = \nsup{i}\left(t_I^+\right)$.} 
For simplicity, we use the same value of the \emph{interaction strength} $\gamma$ for each pair of hosts. We model the time between consecutive interactions for a pair of adjacent hosts \hsupa{i} and \hsupa{j} as an exponentially distributed random variable $X_{ij} \sim \Exp(\lambda_{ij})$. Our interaction process is thus a Markov chain~\cite{Feller1950-fh}.

For convenience, we review relevant background information about exponential distributions. For further details, see \cite{Feller1950-fh}. 
The probability density function $f_{ij}$ of an exponential distribution with parameter $\lambda_{ij}$ is
\begin{equation} \label{eq:int dist}
    f_{ij}(t) = \lambda_{ij} e^{-\lambda_{ij} t} \,.
\end{equation}

An exponentially distributed random variable $X_{ij}$ is memoryless. No matter how much time passes after the most recent interaction between hosts \hsupa{i} and \hsupa{j}, the time that remains until the next interaction is distributed as $X_{ij}$. That is,
\begin{equation} \label{eq:mem}
    \Pr(X_{ij} > t + s \mid X_{ij} > s) = \Pr(X_{ij} > t) \,.
\end{equation}
Because $X_{ij}$ is memoryless, it is easy to describe the random variable
\begin{equation}
    X = \min_{i,j > i} \{X_{ij}\} 
\end{equation}
for the time until the next interaction between any pair of hosts. This random variable is also exponentially distributed: $X \sim \Exp(\ltot)$, where
\begin{equation} \label{eq:ltot}
    \ltot = \sum_{i = 1}^{|H|} \sum_{j = i + 1}^{|H|} \lambda_{ij}
\end{equation}
is the \emph{total-interaction-frequency parameter}. 
The probability that a given interaction is between a specified pair, \hsupa{i} and \hsupa{j}, of hosts is equal to the \emph{relative interaction-frequency parameter}
\begin{equation} \label{eq:rifp}
    l_{ij} = \Pr(X_{ij} = X) = \frac{\lambda_{ij}}{\ltot} \,.
\end{equation}


\subsection{Local Dynamics} \label{sec:LD}

Between interactions, an autonomous local dynamical system 
\begin{equation}
    \frac{\der \nsup{i}}{\der t} = g^{(i)}\left(\nsup{i}\right)
    \label{eq:loc dyn} 
\end{equation}
governs the time evolution of the microbiome abundance vector of each host \hsupa{i}. This dynamical system encodes the local dynamics of host \hsupa{i}. We refer to the function $\gi$ as the \emph{local-dynamics function} of host \hsupa{i}. {It is possible to consider stochastic models of local dynamics, but we restrict ourselves to deterministic local dynamics in the present paper.
With such deterministic dynamics and a Markov-chain interaction process, our modeling framework becomes a piecewise-deterministic Markov process (PDMP) \cite{PDMP}. Such stochastic processes have been used to model many biological processes, including cell growth, gene expression, and neural activity \cite{PDMP_Bio}.

Let the flow $\bm{X}^{(i)}(t,\bm{x)}$ be the solution of
\begin{align}
    \frac{\partial \bm{X}^{(i)}}{\partial t}(t,\bm{x}) &= \gi\left(\bm{X}^{(i)}(t,\bm{x})\right) \,, \\
    \bm{X}^{(i)}(0,\bm{x}) &= \bm{x} \,. \nonumber
\end{align}
For each $\gi$, we require that each element of every valid flow is always finite and nonnegative. Specifically, there is a constant $M \in R^+$ such that $\bm{x} \in [0,M]^n$ implies that each flow $\bm{X}^{(i)}(t,\bm{x}) \in [0,M]^n$ for all times $t \geq 0$. When this condition holds, we say that each $g^{(i)}$ is \emph{bounded}. This is a reasonable assumption for microbiome systems because abundances cannot increase without bound or become negative. 

When all $\gi$ are bounded, local dynamics cannot cause any microbiome abundance vector $\nsup{i}(t)$ to leave the region $[0,M]^n$. Interactions also cannot cause any microbiome abundance vector to leave this region. Consider an interaction at time $t_I$ between hosts \hsupa{i} and \hsupa{j} with microbiome abundance vectors that satisfy $\nsup{i}\left(t_I^-\right) \in [0,M]^n$ and $\nsup{j}\left(t_I^-\right) \in [0,M]^n$. The microbiome exchange between these hosts is an averaging process. After the interaction, the microbiome abundance vectors satisfy $\nsup{i}\left(t_I^+\right) \in [0,M]^n$ and $\nsup{j}\left(t_I^+\right) \in [0,M]^n$. Therefore, if each $\gi$ is bounded and each $\nsup{i}(0) \in [0,M]^n$, it follows that each $\nsup{i}(t) \in [0,M]^n$ for all times $t \geq 0$.

We now present an illustrative model of local dynamics that we use repeatedly to illustrate our framework. For this illustrative model, we assume that each host sustains two microbe species. Therefore, each microbiome abundance vector has dimension two. We use the dynamical system
\begin{align} \label{eq:toy}
    \frac{\der \nsupsub{i}{1}}{\der t} &= -\frac{\nsupsub{i}{1}}{10} \left(\nsupsub{i}{1} - 2\right) \left(\nsupsub{i}{1} - 8\right) \left(\nsupsub{i}{1} - 12\right) \,, \\
    \frac{\der \nsupsub{i}{2}}{\der t} &= -\frac{\nsupsub{i}{2}}{10} \left(\nsupsub{i}{2} - 2\right) \left(\nsupsub{i}{2} - 11\right) \left(\nsupsub{i}{2} - 12\right) \nonumber
\end{align}
for the local dynamics of each host. In \cref{fig:Basins}, we show this dynamical system's four stable equilibrium points and their associated basins of attraction. We use the labels $1$, $2$, $3$, and $4$ for the basins of attraction of the attractors $(2,2)$, $(12,2)$, $(2,12)$, and $(12,12)$, respectively.

\begin{figure}[htbp]
  \centering
  \includegraphics[scale=0.6]{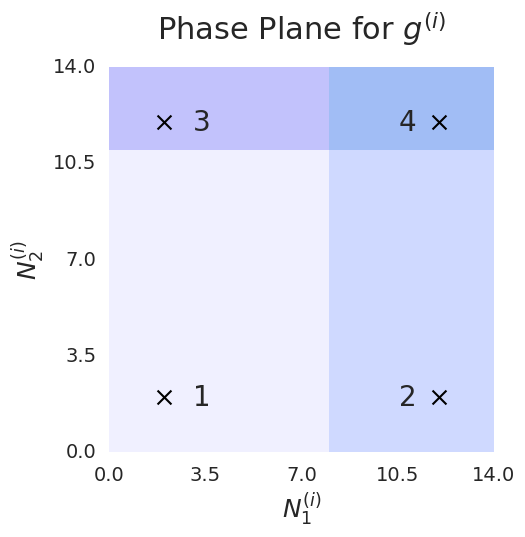}
  \caption{The four stable equilibrium points and associated basins of attraction for the illustrative model \cref{eq:toy} of local dynamics.
    We use the labels $1$, $2$, $3$, and $4$ for the basins of attraction of the attractors $(2,2)$, $(12,2)$, $(2,12)$, and $(12,12)$, respectively.} 
  \label{fig:Basins}
\end{figure}

Consider an interaction network with two hosts \hsupa{1} and \hsupa{2} that are connected by a single edge. Suppose that there is an interaction between the two hosts at time $t_I$ and that $\nsup{1}(t^-_I) = (2,2)$ and $\nsup{2}(t^-_I) = (12,12)$. In \cref{fig:Exchanges}, we show the states of the two hosts immediately after interacting for three values of the interaction strength $\gamma$. This figure demonstrates that a single interaction can change 
{which basin of attraction includes a host's microbiome abundance vector when $\gamma$ is sufficiently large.}

\begin{figure}[htbp]
  \centering
  \includegraphics[scale=0.47]{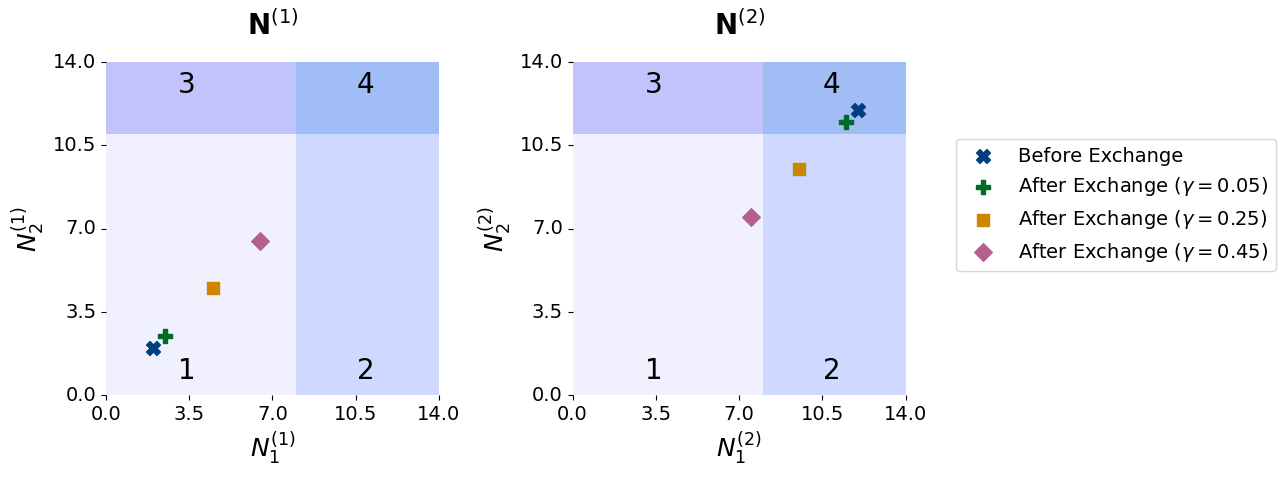}
  \caption{Two hosts with local dynamics \cref{eq:toy}. Immediately before interacting at time $t_I$, the hosts have microbiome abundance vectors $\nsup{1}\left(t_I^-\right) = (2,2)$ and $\nsup{2}\left(t_I^-\right) = (12,12)$. We show the microbiome abundance vectors $\nsup{1}\left(t_I^+\right)$ and $\nsup{2}\left(t_I^+\right)$ of the two hosts immediately after interacting for interaction strengths $\gamma = 0.05$, $\gamma = 0.25$, and $\gamma = 0.45$.
  }
  \label{fig:Exchanges}
\end{figure}

When interactions occur in sufficiently quick succession, a host's microbiome abundance vector can transition to a new basin of attraction even for small values of $\gamma$.
In \cref{fig:Rep Int}, we show an example of this situation. We begin with microbiome abundance vectors $\nsup{1}(0) = (2,2)$ and $\nsup{2}(0) = (12,12)$, and we track the microbiome abundance vectors through five interactions between the hosts. These interactions occur at times $0.1$, $0.3$, $0.4$, $0.7$, and $0.73$. In this example, the interaction strength is $\gamma = 0.32$. The first interaction is sufficient to move $\nsup{2}(t)$ from basin $4$ to basin $2$. However, for $\nsup{2}(t)$ to move from basin $2$ to basin $1$, two interactions must occur in sufficiently quick succession. The interactions at times 0.3 and 0.4 do not cause this transition. However, the interactions at times 0.7 and 0.73 are close enough in time to cause $\nsup{2}(t)$ to move from basin $2$ to basin $1$.

\begin{figure}[htbp]
  \centering
  \includegraphics[scale=0.45]{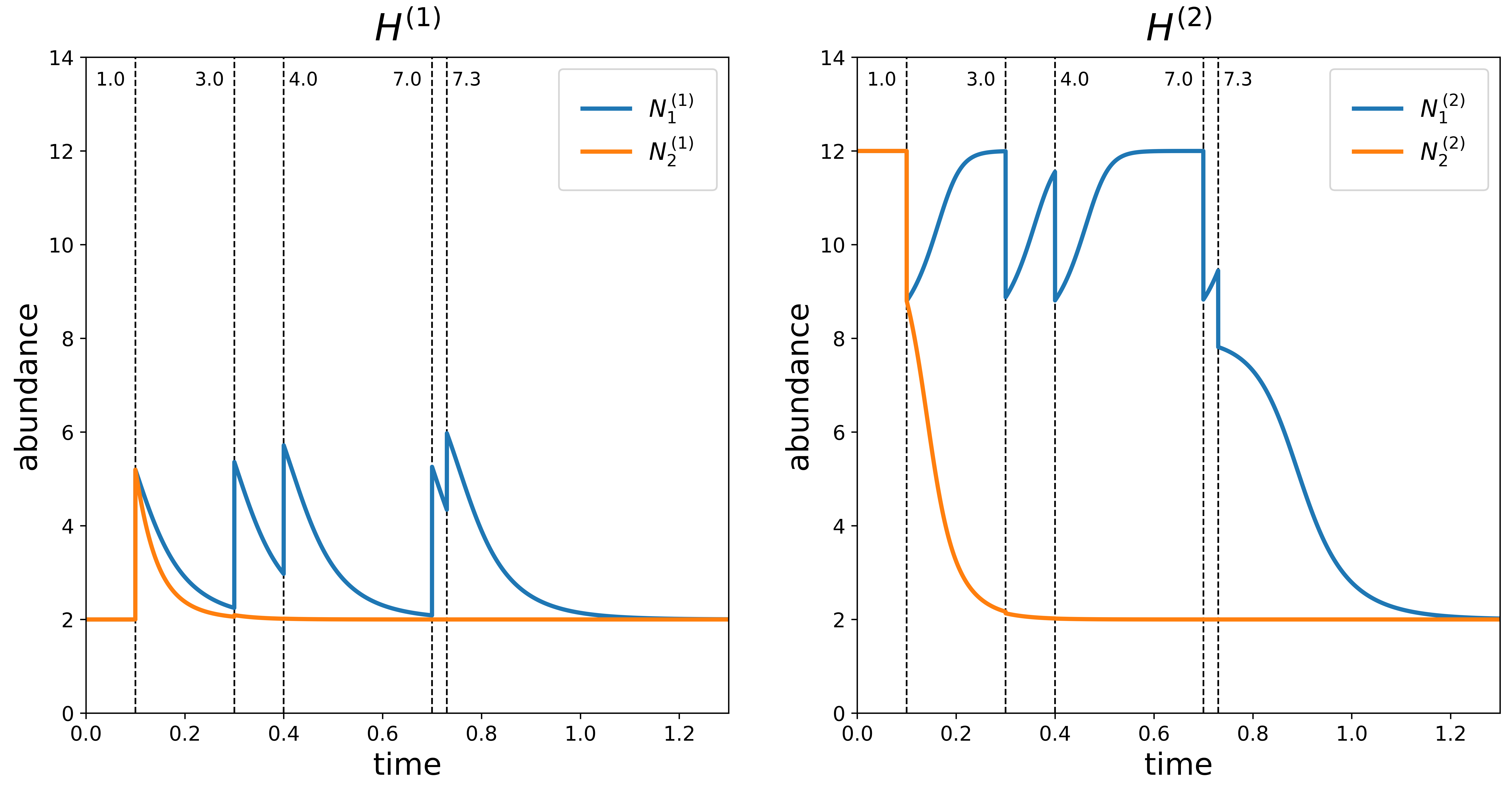}
  \caption{Two hosts with local dynamics \cref{eq:toy}. These hosts have initial 
  microbiome abundance vectors $\nsup{1}(0) = (2,2)$ and $\nsup{2}(0) = (12,12)$. We show the abundances of microbe species $1$ and $2$ in each host through the course of five 
  interactions, which occur at times $0.1$, $0.3$, $0.4$, $0.7$, and $0.73$.}
  \label{fig:Rep Int}
\end{figure}



\section{Low-Frequency Approximation (LFA)} \label{sec:LFA}

In this section, we discuss an approximation that is accurate when all interaction-frequency parameters $\lambda_{ij}$ are sufficiently small. This \emph{low-frequency approximation} (LFA) describes the evolution of the dynamical system \crefcomma{eq:exch,eq:loc dyn} when the local dynamics of each host is much faster than the exchange dynamics between hosts.
 
We develop the LFA for systems in which the set of stable attractors of each host's local dynamics,
{which is deterministic,} is a finite set of equilibrium points. We believe that it is possible to derive extensions of the LFA for systems with other types of attractors, and we discuss this possibility in \Cref{sec:FW}. We define relevant terminology in \Cref{sec:BST,sec:IO}, and we describe the LFA and outline the proof of its accuracy in \Cref{sec:LFA theo}. We give a complete proof in 
\Cref{sec:LFA Proof}.


\subsection{Basin State Tensor} \label{sec:BST}

We illustrated in \cref{fig:Exchanges} that interactions can result in transitions of 
{a host's microbiome abundance vector between different basins of attraction.}
Local dynamics cannot cause such a transition to occur, so interactions between hosts are necessary for such transitions.

{In the LFA,} we need to track which basins of attraction include each microbiome abundance vector.
To do this, we define a \emph{basin probability vector} $\psisup{i}(t)$ for each host \hsupa{i}. An entry $\psisupsub{i}{a}(t)$ of this vector gives the probability that 
the microbiome abundance vector of host \hsupa{i} is in basin of attraction $a$ at time $t$. If the local dynamics of host \hsupa{i} has $m_i$ basins of attraction, then $\psisup{i}(t)$ has dimension $m_i$. Different hosts can have different local dynamics, so the basin probability vectors of different hosts can have different dimensions. 

Each local-dynamics function $\gi$ is bounded, as we described in \Cref{sec:LD}. Therefore, for all times $t \geq 0$, each microbiome abundance vector $\nsup{i}(t) \in [0,M]^n$. Because the set of stable attractors of each host's local dynamics consists of a finite set of equilibrium points, the set of points $\mathcal{U}^{(i)} \subset [0,M]^n$ that are not in the basin of attraction of some equilibrium point has measure $0$. For the LFA to be accurate, we require specific conditions on the local-dynamics functions $\gi$ and the interaction strength $\gamma$. We describe these conditions in the Low-Frequency Approximation Theorem (see \cref{thm:LFA}). When these conditions on $\gi$ and $\gamma$ are satisfied and the total-interaction-frequency parameter $\ltot \to 0$, no microbiome abundance vector $\nsup{i}(t)$ lies on the border between basins of attraction at any time $t$ in a finite interval $[0,T]$ with arbitrarily high probability.

We represent the state of the entire set of hosts using a \emph{basin probability tensor} $\Psi(t)$. If each basin probability vector $\psisup{i}(t)$ has dimension $m_i$, then $\Psi(t)$ has dimension $m_1\times m_2 \times \cdots \times m_{|H|}$. An entry of the basin probability tensor $\Psi_{a_1,a_2,\ldots,a_{|H|}}(t)$ gives the probability that the microbiome abundance vector of each host \hsupa{i} is in basin of attraction $a_i$ at time $t$. If we assume that these probabilities are independent at time $t$, then
\begin{equation} \label{eq:tens prod}
    \Psi(t) = \bigotimes_i \psisup{i}(t) \,.
\end{equation}
Typically, the basin probability vectors $\psisup{i}(t)$ are not independent after any interaction, and then \cref{eq:tens prod} no longer holds.


\subsection{Interaction Operators} \label{sec:IO}

The basin probability tensor $\Psi(t)$ can change only due to an interaction. Unfortunately, knowing $\Psi(t)$ before an interaction and which pair of hosts interact does not give enough information to determine $\Psi(t)$ after the interaction. One also needs information about the microbiome abundance vectors of the interacting hosts.

Suppose that an interaction occurs between hosts \hsupa{1} and \hsupa{2} at time $t_I$ and that their microbiome abundance vectors $\nsup{1}\left(t_I^-\right)$ and $\nsup{2}\left(t_I^-\right)$ are at stable equilibrium points immediately before the interaction. We make this assumption throughout the rest of \Cref{sec:IO}. If we know that 
 {$\nsup{1}\left(t_I^-\right)$ and $\nsup{2}\left(t_I^-\right)$ are in basins of attraction $a_1$ and $a_2$, respectively, then we can determine which basins of attraction include $\nsup{1}\left(t_I^+\right)$ and $\nsup{2}\left(t_I^+\right)$.} 
 We illustrate this in \cref{fig:All Ex} for an example in which both hosts have local dynamics \cref{eq:toy} and the interaction strength is $\gamma = 0.25$. We show $\nsup{1}\left(t_I^+\right)$ for every possible combination of $\nsup{1}\left(t_I^-\right)$ and $\nsup{2}\left(t_I^-\right)$. Because $\nsup{1}\left(t_I^-\right)$ and $\nsup{2}\left(t_I^-\right)$ are at stable equilibrium points (by assumption), there are 16 such combinations. For example, for $\nsup{1}\left(t_I^-\right) = (12,12)$, there are four possible values of $\nsup{1}\left(t_I^+\right)$. These values are $(9.5,9.5)$, $(12,9.5)$, $(9.5,12)$, and $(12,12)$; the four corresponding values of $\nsup{2}\left(t_I^-\right)$ are $(2,2)$, $(12,2)$, $(2,12)$, and $(12,12)$. In \cref{fig:All Ex}, we mark these four possible values of $\nsup{1}\left(t_I^+\right)$ with the diamonds in the upper-right corner. Host \hsupa{2} has the same local dynamics as host \hsupa{1}, so the situation is identical for $\nsup{2}\left(t_I^+\right)$, except that we exchange the indices $1$ and $2$ everywhere.

\begin{figure}[htbp]
  \centering
  \includegraphics[scale=0.6]{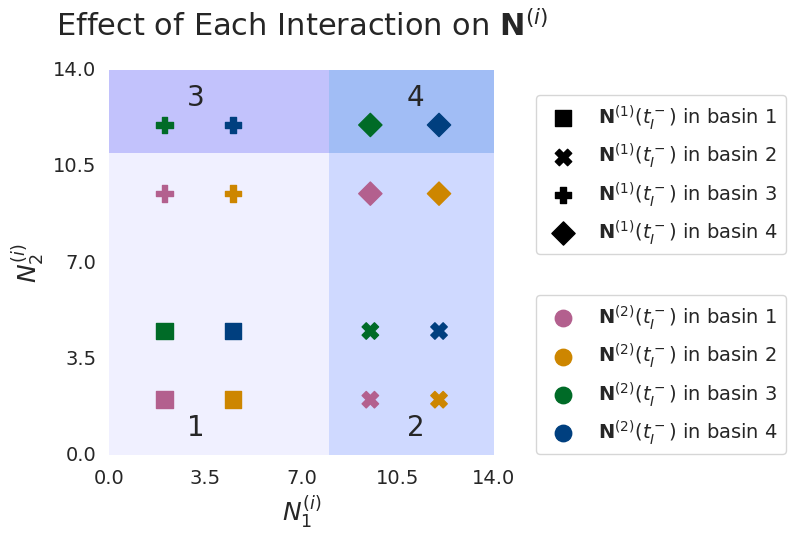}
  \caption{Each possible microbiome abundance vector {$\nsup{1}\left(t_I^+\right)$} after an interaction at time $t_I$ between hosts \hsupa{1} and \hsupa{2} with local dynamics \cref{eq:toy}, assuming that $\nsup{1}\left(t_I^-\right)$ and $\nsup{2}\left(t_I^-\right)$ are at stable equilibrium points before the interaction. 
  {The interaction strength is $\gamma = 0.25$}. The marker shapes {(see the top legend)} indicate 
  {which basin of attraction includes $\nsup{1}\left(t_I^-\right)$, and the marker colors {(see the bottom legend)} indicate which basin of attraction includes $\nsup{2}\left(t_I^-\right)$.} 
  For example, for $\nsup{1}\left(t_I^-\right) = (12,12)$, the four possible values of $\nsup{1}\left(t_I^+\right)$ are $(9.5,9.5)$, $(12,9.5)$, $(9.5,12)$, and $(12,12)$; the 
  four corresponding values of $\nsup{2}\left(t_I^-\right)$ are $(2,2)$, $(12,2)$, $(2,12)$, and $(12,12)$. We mark the four possible values of $\nsup{1}\left(t_I^+\right)$ with the diamonds in the upper-right corner. The colors of the diamonds indicate the corresponding values of $\nsup{2}\left(t_I^-\right)$.}
  \label{fig:All Ex}
\end{figure}

For some values of the interaction strength $\gamma$, an interaction between hosts \hsupa{i} and \hsupa{j} can result in either $\nsup{i}\left(t_I^+\right)$ or $\nsup{j}\left(t_I^+\right)$ lying on a boundary between different basins of attraction, rather than in a basin of attraction. That is, after such an interaction, $\nsup{i}\left(t_I^+\right) \in \mathcal{U}^{(i)}$ or $\nsup{j}\left(t_I^+\right) \in \mathcal{U}^{(j)}$. We refer to the set of such interaction strengths as the \emph{boundary set} $\mathcal{B}_{ij}$ for the hosts \hsupa{i} and \hsupa{j}. For the example in \cref{fig:All Ex}, the boundary set is $\mathcal{B}_{12} = \{0.1,0.4\}$. If hosts \hsupa{i} and \hsupa{j} cannot interact, then $\mathcal{B}_{ij}$ is the empty set. The \emph{total boundary set} is
\begin{equation}
    \mathcal{B} = \bigcup_{i,j} \mathcal{B}_{ij} \,.
\end{equation}
For the LFA, we require $\gamma \not\in \mathcal{B}$, and we assume that this is the case for the rest of \Cref{sec:LFA}. 

Under our assumptions, we need to know only 
{which basins of attraction include $\nsup{i}\left(t_I^-\right)$ and $\nsup{j}\left(t_I^-\right)$ (i.e., immediately before an interaction) to determine which basins of attraction include $\nsup{i}\left(t_I^+\right)$ and $\nsup{j}\left(t_I^+\right)$ (i.e., immediately after an interaction).} 
 Therefore, we need to know only the basin probability tensor $\Psi\left(t_I^-\right)$ prior to an interaction to determine the basin probability tensor $\Psi\left(t_I^+\right)$ after that interaction. To describe such an interaction-induced change, we define a \emph{pairwise interaction operator} $\Phi^{(ij)}$ with dimension $m_1\times \cdots \times m_{|H|}\times m_1 \times \cdots \times m_{|H|}$. Its entry $\Phi^{(ij)}_{b_1,\ldots,b_{|H|},a_1,\ldots,a_{|H|}} = 1$ if $b_i$ and $b_j$ are the basins of attraction 
{that include $\nsup{i}\left(t_I^+\right)$ and $\nsup{j}\left(t_I^+\right)$, respectively, when $a_i$ and $a_j$ are the basins of attraction that include $\nsup{i}\left(t_I^-\right)$ and $\nsup{j}\left(t_I^-\right)$, respectively, and $b_k = a_k$ for all $k \not\in \{i,j\}$. Otherwise, $\Phi^{(ij)}_{b_1,\ldots,b_{|H|},a_1,\ldots,a_{|H|}} = 0$.} 

Consider a two-host interaction network in which both hosts have local dynamics \cref{eq:toy} and the interaction strength is $\gamma = 0.25$. The basin probability tensor of this system has dimension $4\times 4$. Therefore, the pairwise interaction operator $\Phi^{(12)}$ has dimension $4 \times 4 \times 4 \times 4$. In \cref{fig:All Ex}, we show the results of all possible interactions between these two hosts. For example, if $\nsup{1}\left(t_I^-\right)$ and $\nsup{2}\left(t_I^-\right)$ are in basins of attraction $1$ and $4$, respectively, then the ensuing basins of attraction 
{that include $\nsup{1}\left(t_I^+\right)$ and $\nsup{2}\left(t_I^+\right)$ are $1$ and $2$, respectively.} Therefore, $\Phi^{(12)}_{1,2,1,4} = 1$. There are 16 possible combinations of $\nsup{1}\left(t_I^-\right)$ and $\nsup{2}\left(t_I^-\right)$, so 16 entries of $\Phi^{(12)}$ are equal to $1$. We list these 16 entries in \cref{fig:Int Op Entries}.

\begin{figure}[htbp]
  \centering
  \def\arraystretch{2.5}
  \begin{tabular} {c c c c}
       $\Phi^{(12)}_{1,1,1,1}$  &  $\Phi^{(12)}_{1,2,1,2}$ & $\Phi^{(12)}_{1,1,1,3}$ & $\Phi^{(12)}_{1,2,1,4}$ \\
       $\Phi^{(12)}_{2,1,2,1}$  &  $\Phi^{(12)}_{2,2,2,2}$ & $\Phi^{(12)}_{2,1,2,3}$ & $\Phi^{(12)}_{2,2,2,4}$ \\
       $\Phi^{(12)}_{1,1,3,1}$  &  $\Phi^{(12)}_{1,2,3,2}$ & $\Phi^{(12)}_{3,3,3,3}$ & $\Phi^{(12)}_{3,4,3,4}$ \\
       $\Phi^{(12)}_{2,1,4,1}$  &  $\Phi^{(12)}_{2,2,4,2}$ & $\Phi^{(12)}_{4,3,4,3}$ & $\Phi^{(12)}_{4,4,4,4}$
    \end{tabular}
  \caption{The 16 entries of the pairwise interaction operator {$\Phi^{(12)}$} that equal $1$ for a two-host interaction network in which both hosts have local dynamics \cref{eq:toy} and the interaction strength is $\gamma = 0.25$.}
  \label{fig:Int Op Entries}
\end{figure}

We use the Einstein summation convention~\cite{einstein} to describe how $\Phi^{(ij)}$ operates on the basin probability tensor $\Psi$. In this convention, one sums over any repeated index that occurs in a single term. For example, we describe the product 
\begin{equation}
    \bm{y} = A  \bm{x}
\end{equation}
of a matrix $A$ and a vector $\bm{x}$ by writing
\begin{equation}
    y_i = \sum_j A_{ij} x_j  \,.
\end{equation}
Using the Einstein summation convention, we write
\begin{equation}
    y_i = A_{ij} x_j \,.
\end{equation}
We write the effect of the pairwise interaction operator $\Phi^{(ij)}$ on the basin probability tensor $\Psi$ as
\begin{equation}
    \Psi_{b_1,\ldots,b_{|H|}}\left(t_I^+\right) = \Phi^{(ij)}_{b_1,\ldots,b_{|H|},a_1,\ldots,a_{|H|}} \Psi_{a_1,\ldots, a_{|H|}}\left(t_I^-\right) \,.
\end{equation}

Each $\Phi^{(ij)}$ acts linearly on the basin probability tensor. At any time, the next interaction is between hosts \hsupa{i} and \hsupa{j} with probability $l_{ij}$ (see \cref{eq:rifp}). This yields the \emph{total-interaction operator}
\begin{equation} \label{eq:Phi}
    \Phi = \sum_{i = 1}^{|H|} \sum_{j = i + 1}^{|H|} l_{ij} \Phi^{(ij)} \,.
\end{equation}

If we now assume that an interaction occurs between some pair of hosts at time $t_I$ (with probabilities $l_{ij}$ for each pair) and that all $\nsup{i}\left(t_I^-\right)$ are at a stable equilibrium point, then we obtain
\begin{equation}
    \Psi_{b_1,\ldots,b_{|H|}}\left(t_I^+\right) = \Phi_{b_1,\ldots,b_{|H|},a_1,\ldots,a_{|H|}} \Psi_{a_1,\ldots, a_{|H|}}\left(t_I^-\right) \,.
\end{equation}


\subsection{Low-Frequency-Approximation Theorem} \label{sec:LFA theo}

The LFA encodes the evolution of the system \crefcomma{eq:exch,eq:loc dyn} when the local dynamics of each host is much faster than the exchange dynamics between hosts. 
In this parameter region, each microbiome abundance vector $\nsup{i}(t)$ becomes close to a stable equilibrium point before the next interaction that involves host \hsupa{i}. Therefore, the total-interaction operator $\Phi$ accurately describes the dynamics of the basin probability tensor $\Psi(t)$.

Before we state the LFA Theorem (see \cref{thm:LFA}), we introduce some helpful terminology. The expected number of host interactions in a time interval of duration $\Delta t$ is $\ltot\Delta t$. We refer to $t^* = \ltot t$ as the \emph{frequency-scaled time}. Additionally, we say that the local-dynamics function $\gi$ is \emph{inward pointing} at a point $\bm{x}$ if there exists a constant $\delta > 0$ such that $\normtwo{\bm{y} - \bm{x}} \leq \delta$ implies that $\gi(\bm{y}) \cdot (\bm{x} - \bm{y}) > 0$.

\begin{theorem}[Low-Frequency-Approximation Theorem] \label{thm:LFA}
  Suppose that the attractors of each host's local dynamics consist of a finite set of stable equilibrium points at which the local-dynamics function $\gi$ is inward pointing, and let each $\gi$
  be continuous and bounded (see \Cref{sec:LD}). Fix $\gamma \not\in \mathcal{B}$, all $l_{ij}$, and a frequency-scaled time $T^*$. As the total-interaction-frequency parameter $\lambda_{\mathrm{tot}} \to 0$, the basin probability tensor $\Psi(t^*)$ converges uniformly to $\widetilde{\Psi}(t^*)$ on $[0,T^*]$, where
  \begin{align}\label{eq:LFA}
      \frac{\der}{\der t^*}\widetilde\Psi_{b_1,\ldots,b_{|H|}} (t^*) &= \Phi_{b_1,\ldots,b_{|H|},a_1,\ldots,a_{|H|}} \widetilde\Psi_{a_1,\ldots, a_{|H|}}(t^*) - \widetilde\Psi_{b_1,\ldots,b_{|H|}}(t^*) \,, \\
      \widetilde{\Psi}(0) &= \Psi(0) \,. \nonumber
  \end{align}
\end{theorem}

In this section, we present key steps of the proof of \cref{thm:LFA}.
 We give a complete proof in 
\Cref{sec:LFA Proof}.

As we described in \Cref{sec:IO}, when each $\nsup{i}\left(t_I^-\right)$ is at a stable equilibrium point, the interaction operator $\Phi$ describes the update of the basin probability tensor after an interaction at time $t_I$. We can construct neighborhoods around each stable equilibrium point of each host's local dynamics such that if each $\nsup{i}\left(t_I^-\right)$ is in one of these neighborhoods, then $\Phi$ perfectly describes the transition of $\Psi(t)$ due to an interaction at time $t_I$. In \cref{fig:Basins with Balls}, we show an example of what these neighborhoods can look like around the stable equilibrium points for our illustrative model \cref{eq:toy} of local dynamics.

\begin{figure}[htbp]
  \centering
  \includegraphics[scale=0.6]{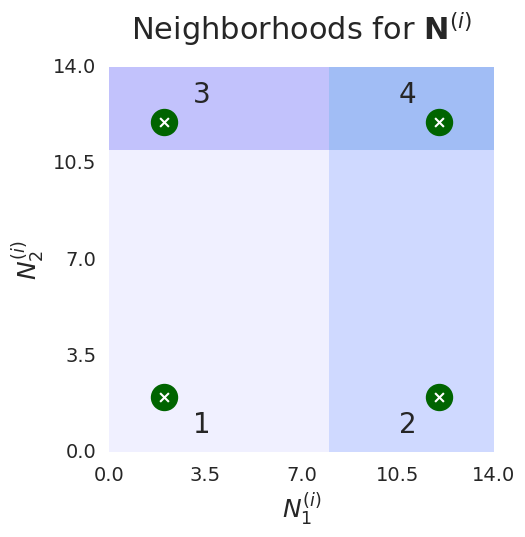}
  \caption{An illustration of potential neighborhoods around the four stable equilibrium points for a host with the local dynamics \cref{eq:toy}. These neighborhoods illustrate potential neighborhoods from \cref{thm:LFA}; they are not the neighborhoods for any particular value of the interaction strength $\gamma$.
  }
  \label{fig:Basins with Balls}
\end{figure}

After an interaction, there is an upper bound on the time that it takes for each microbiome abundance vector to re-enter a neighborhood around a stable equilibrium point. 
The system is finite, so there is an upper bound $\tau$ that guarantees that each microbiome abundance vector $\nsup{i}(t + \tau)$ is in one of these neighborhoods if no interactions occur in the time interval $[t, t  + \tau]$.
If no two interactions occur within time $\tau$ of each other, then successive applications of the total-interaction operator $\Phi$ perfectly describe the evolution of the basin probability tensor $\Psi(t)$. In the frequency-scaled time interval $[0,T^*]$, the expected number of interactions in a system is $T^*$. As $\ltot\to 0$, the frequency-scaled time $\tau^* \to 0$. Therefore, it becomes vanishingly unlikely that any pair of interactions occurs within a frequency-scaled time that is less than $\tau^*$. Consequently, as the total-interaction-frequency parameter $\ltot \to 0$, the effect on the basin probability tensor $\Psi(t^*)$ of all interactions in $[0,T^*]$ is described perfectly by the total-interaction operator $\Phi$ with arbitrarily high probability. The interactions are exponentially distributed, so the basin probability tensor $\Psi(t^*)$ converges uniformly to $\widetilde\Psi(t^*)$ 
(see \cref{eq:LFA}).



\section{High-Frequency Approximations} \label{sec:HFA}

In this section, we discuss two approximations that are accurate when the total-interaction-frequency parameter $\ltot$ is large.
In both approximations, which are valid in distinct parameter regions, we fix the relative interaction-frequency parameters {$l_{ij}$. (Recall that $l_{ij} = \lambda_{ij} / \ltot$ (see \cref{eq:rifp}).)}
The first of these approximations is the \emph{high-frequency, low-strength approximation} (HFLSA), which becomes increasingly accurate as 
$\ltot \to \infty$ and $\gamma \to 0$ for fixed relative interaction-frequency parameters $l_{ij}$ and fixed $\ltot \gamma$. The HFLSA yields a model that has the same form as the mass-effects model \cref{eq:meta}. The second approximation is the \emph{high-frequency, constant-strength approximation} (HFCSA), which becomes increasingly accurate as $\ltot\to\infty$ for fixed relative interaction-frequency parameters $l_{ij}$ and fixed interaction strength 
$\gamma$. {We include the description ``constant strength" in the name of the HFCSA to emphasize that we fix the interaction strength}.


\subsection{High-Frequency, Low-Strength Approximation (HFLSA)} \label{HFLS}

{The HFLSA describes the evolution of the dynamical system \crefcomma{eq:exch,eq:loc dyn}} when the interactions between hosts are very frequent but very weak. In this parameter region, the expectation of the exchange dynamics \cref{eq:exch} is constant, but the variance of the exchange dynamics is small. This results in an approximate model for the dynamics of each microbiome abundance vector $\nsup{i}(t)$ that is deterministic and has terms that encode the effects of the local dynamics and the exchange dynamics. This approximate model has the same form as the mass-effects model \cref{eq:meta}.

\begin{theorem}[High-Frequency, Low-Strength Approximation Theorem]\label{thm:HFLSA}
  Fix the relative interaction-frequency parameters $l_{ij}$, the product $\ltot \gamma$, and a time $T$. Let each local-dynamics function $g^{(i)}$ be continuously differentiable and bounded (see \Cref{sec:LD}), and fix arbitrary $\varepsilon\in(0,1]$ and $\delta > 0$. For sufficiently large $\ltot$, each host's microbiome abundance vector $\nsup{i}(t)$ satisfies
  \begin{equation} \label{eq:HFLSA bound}
      \normLinf{\nsup{i} - \nsupapp{i}}{[0,T]} < \delta
  \end{equation}
  with probability larger than $1 - \varepsilon$, where
\begin{align}\label{eq:HFLSA}
        \frac{\der \nsupapp{i}}{\der t} &= g^{(i)}(\nsupapp{i}) + \sum_j \lambda_{ij} \gamma \left(\nsupapp{j} - \nsupapp{i}\right) \,, \\
        \nsupapp{i}(0) &= \nsup{i}(0) \,. \nonumber
\end{align}
\end{theorem}

In this section, we present key steps of the proof of \cref{thm:HFLSA}. We give a complete proof in \Cref{sec:HFLSA Proof}.
See \cite{Pakdaman10, Pakdaman12} for similar approximations of other PDMPs.

Consider the evolution of a microbiome abundance vector $\nsup{i}(t)$ over a short time interval $[t', t' + dt]$. 
Without interactions, the effect of the local dynamics is
\begin{equation} \label{eq:effloc}
    \left(\nsup{i}(t' + dt) - \nsup{i}(t')\right)_\text{loc} = \gi\left(\nsup{i}(t')\right) dt + O(dt^2) \, .
\end{equation}

Let
\begin{equation}
	\Ji(t_I) = \nsup{i}\left(t_I^+\right) - \nsup{i}\left(t_I^-\right) 
\end{equation}
be the effect of an interaction at time $t_I$ on $\nsup{i}(t)$. If the interaction does not involve \hsupa{i}, then $\Ji(t_I) = \bm{0}$. Otherwise, if the interaction is between \hsupa{i} and \hsupa{j}, then 
\begin{equation}
    \Ji(t_I) = \gamma \left(\nsup{j}\left(t_I^-\right) - \nsup{i}\left(t_I^-\right)\right) \,.
\end{equation}

Suppose that no interactions occur precisely at times $t'$ or $t' + dt$ and that $L$ interactions occur during the time interval $(t', t' + dt)$. We denote this ordered set of {interaction times} by $\{t_{l}\}_{l = 1}^{L}$. 
Suppose for all $i$ that the microbiome abundance vector $\nsup{i}(t)$ changes very little during the interval $[t',t' + dt]$ and in particular that each $\Ji(t_l)$ is well-approximated by
\begin{equation}
        \Jiapp_{l} = \begin{cases}
            \bm{0} & \text{if the interaction at } t_I \text{ does not involve host \hsupa{i}} \\
            \gamma \left(\nsup{j}(t') - \nsup{i}(t')\right) & \text{if the interaction at } t_I \text{ is between hosts \hsupa{i} and \hsupa{j}} \,.
        \end{cases} 
\end{equation}
{We refer to $\Jiapp_{l}$ as the \emph{approximate interaction effect} of the interaction at time $t_l$ on host \hsupa{i}.} The effect of the interactions {on the microbiome abundance vector $\nsup{i}(t)$} during 
the interval $[t',t' + dt]$ is well-approximated by
\begin{equation} \label{eq:this}
    \left(\nsup{i}(t' + dt) - \nsup{i}(t')\right)_\text{exch} = \sum_{l = 1}^{L} \Jiapp_{l} \, {.}
\end{equation}
Each of the approximate interaction effects $\Jiapp_{l}$ is vector-valued. We denote entry $x$ of this vector by $\left(\Jiapp_{l}\right)_x$. The sum \cref{eq:this} is also vector-valued, and we denote entry $x$ of this sum by 
$\left(\sum_{l = 1}^{L} \Jiapp_{l}\right)_x$. 

We now calculate the expectation and variance of each entry of \cref{eq:this}. The stochasticity in \cref{eq:this} arises both from the number $L$ of interactions and from which pair of hosts interact at each time. The number of interactions follows a Poisson distribution with mean $\ltot  dt$, and the approximate interaction effects $\Jiapp_l$ are independent of one another. An interaction at time $t_l$ is between hosts \hsupa{i} and \hsupa{j} with probability $\lambda_{ij} / \ltot$. We thus obtain
\begin{align}
    \E\left[L\right] &= \ltot   dt \,, \\
    \Var\left[L\right] &= \ltot   dt \,, \nonumber \\
    \E\left[\left(\Jiapp_{l}\right)_x\right] &= \sum_j \frac{\lambda_{ij}}{\ltot} \gamma \left(\nsup{j}(t')-\nsup{i}(t')\right)_x \,, \nonumber \\
    \Var\left[\left(\Jiapp_{l}\right)_x\right] &= \E\left[\left(\Jiapp_{l}\right)_x^2\right] - \left(\E\left[\left(\Jiapp_{l}\right)_x\right]\right)^2 \,. \nonumber
\end{align}
The expectation of each entry of the sum in \cref{eq:this} is
\begin{align} \label{eq:exsum}
    \E\left[\left(\sum_{l = 1}^{L} \Jiapp_{l}\right)_x\right] &= \E\left[\;\E \left[\left(\sum_{l = 1}^{L} \Jiapp_{l}\right)_x \;\middle|\; L \ \right]\;\right]  \\
    &= \E \left[L  \sum_j \frac{\lambda_{ij}}{\ltot} \gamma \left(\nsup{j}(t') - \nsup{i}(t')\right)_x\right] \nonumber\\
    &= \ltot   dt \sum_j \frac{\lambda_{ij}}{\ltot} \gamma \left(\nsup{j}(t') - \nsup{i}(t')\right)_x \nonumber\\
    &= \sum_j \lambda_{ij} \gamma \left(\nsup{j}(t') - \nsup{i}(t')\right)_x \, dt \,. \nonumber
    \end{align}
Applying the law of total variance, the variance of each entry of the sum {in \cref{eq:this}} is
\begin{align} \label{eq:varsum}
        \Var\left[\left(\sum_{l=1}^{L} \Jiapp_{l}\right)_x\right] &= \E\left[\; \Var \left[\left(\sum_{l=1}^{L} \Jiapp_{l}\right)_x \;\middle|\; L \ \right]\;\right] + \Var\left[\;\E\left[\left(\sum_{l=1}^{L} \Jiapp_{l}\right)_x \;\middle|\; L \ \right]\;\right] \\
        &= \E\left[L \ \Var\left[\left(\Jiapp_{l}\right)_x\right]\;\right] + \Var\left[ \; L \ \E\left[\left(\Jiapp_{l}\right)_x\right]\;\right] \nonumber\\
        &= \ltot   dt \, \Var\left[\left(\Jiapp_{l}\right)_x\right] + \ltot   dt \left(\E\left[\left(\Jiapp_{l}\right)_x\right]\right)^2 \nonumber\\
        &= \ltot   dt \, \E\left[\left(\Jiapp_{l}\right)_x^2\right] \nonumber\\
        &= \ltot  dt \sum_j \frac{\lambda_{ij}}{\ltot} \gamma^2 \left(\nsup{j}(t') - \nsup{i}(t')\right)_x^2 \nonumber\\
        &= \sum_j \lambda_{ij} \gamma^2 \left(\nsup{j}(t') - \nsup{i}(t')\right)_x^2 \, dt \, . \nonumber
\end{align}

Because $\ltot \gamma$ is fixed, the interaction strength $\gamma \rightarrow 0$ as the total-interaction-frequency parameter $\ltot \rightarrow \infty$. As $\gamma \to 0$, the expectation \cref{eq:exsum} of each entry of \cref{eq:this} remains fixed, but the variance \cref{eq:varsum} of each entry decreases to $0$. For sufficiently small $\gamma$, the effect of interactions {on the microbiome abundance vector $\nsup{i}(t)$} during the interval $[t',t' + dt]$ is
\begin{equation} \label{eq:effex}
    \left(\nsup{i}(t' + dt) - \nsup{i}(t')\right)_\text{exch} = \sum_j \lambda_{ij} \gamma \left(\nsup{j}(t') - \nsup{i}(t')\right) dt + O(dt^2)
\end{equation}
with arbitrarily high probability.

For sufficiently small $\gamma$, the change in the microbiome abundance vector $\nsup{i}(t)$ during the interval $[t',t' + dt]$ is approximately equal to the sum of the effect \cref{eq:effloc} of local dynamics and the effect \cref{eq:effex} of interactions. In 
\Cref{sec:HFLSA Proof},
we show that
\begin{align}
    \nsup{i}(t' + dt) - \nsup{i}(t') &= \left(\nsup{i}(t' + dt) - \nsup{i}(t')\right)_\text{loc} \\
    &\qquad + \left(\nsup{i}(t' + dt) - \nsup{i}(t')\right)_\text{exch} + O(dt^2) \nonumber \\
    		&= \left[\gi\left(\nsup{i}(t')\right) + \sum_j \lambda_{ij} \gamma \left(\nsup{j}(t') - \nsup{i}(t')\right)\right] dt + O(dt^2)  \nonumber
\end{align}
with arbitrarily high probability. Therefore, $\nsup{i}(t)$ is well-approximated by $\nsupapp{i}(t)$.

As an example of the HFLSA, consider a two-host system in which each host has local dynamics \cref{eq:toy}. Let $\nsup{1}(0) = (2,2)$ and $\nsup{2}(0) = (12,12)$. In \cref{fig:Meta Limit Ex}, we show how the approximation improves as we increase the total-interaction-frequency parameter $\ltot = \lambda_{12}$ and decrease the interaction strength $\gamma$ for fixed $\ltot \gamma = 8$. The error that we obtain by using the approximate microbiome abundance vector $\nsupapp{i}(t)$ appears to be largest near times that the $i$th abundance vector $\nsupapp{i}(t)$ transitions between different basins of attraction. However, for sufficiently large $\ltot$, this error becomes arbitrarily small for all times $t \in [0,T]$ with arbitrarily high probability.

\begin{figure}[tbp]
  \centering
  \includegraphics[scale=0.45]{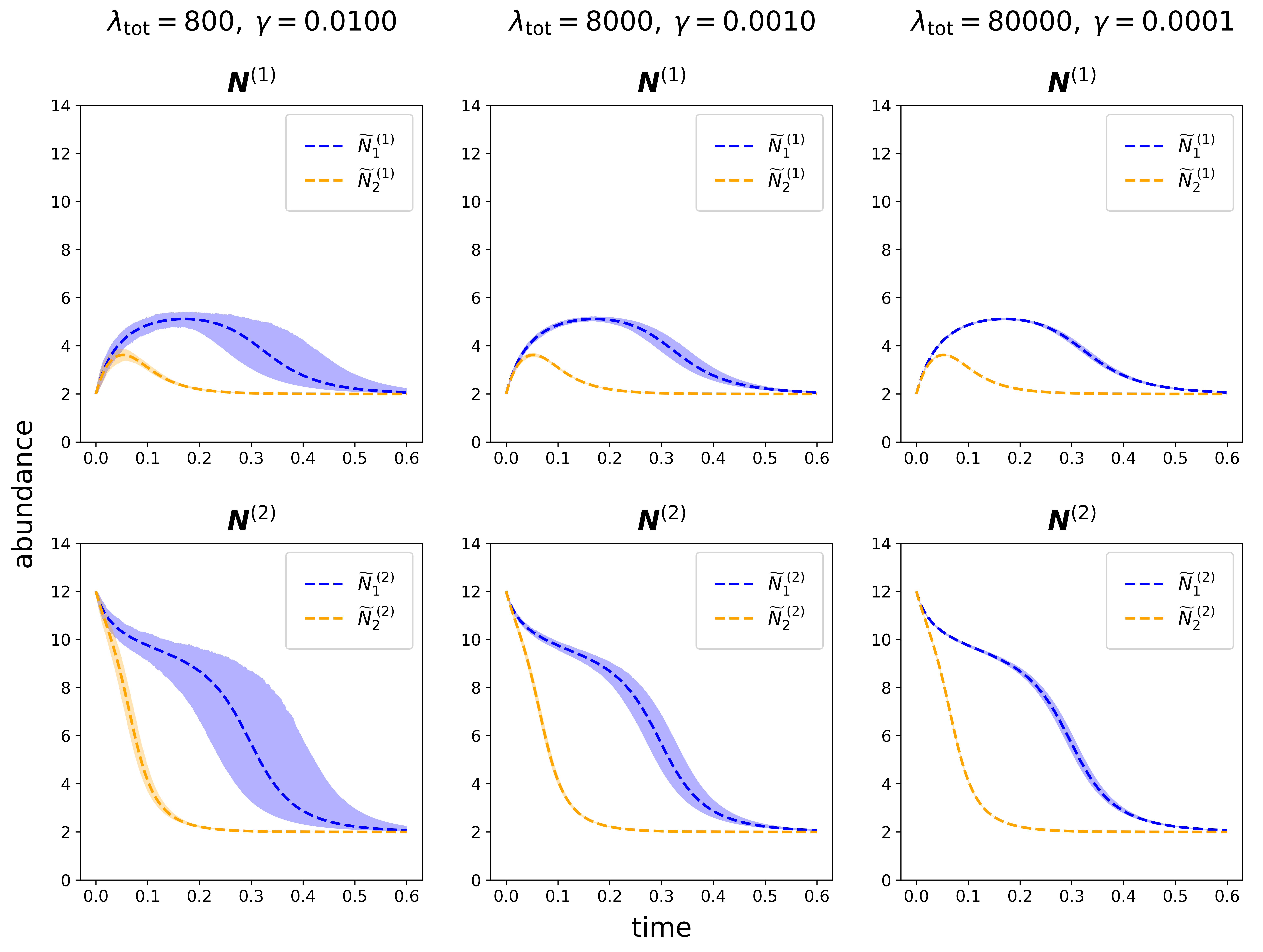}
  \caption{Numerical experiments for a two-host system in which each host has local dynamics \cref{eq:toy}. The three columns show experiments for different values of the total-interaction-frequency parameter $\lambda_{\mathrm{tot}}$ and the interaction strength $\gamma$ for fixed $\lambda_{\mathrm{tot}} \gamma = 8$. We show the microbiome abundances for \hsupa{1} in the first row and the microbiome abundances for \hsupa{2} in the second row. We run 500 simulations for each set of parameters. The highlighted region shows the range between the $5^\mathrm{{th}}$ and $95^\mathrm{{th}}$ percentiles of the simulated host abundances. The dashed curves show the HFLSAs for these experiments.}
  \label{fig:Meta Limit Ex}
\end{figure}


\subsection{High-Frequency, Constant-Strength Approximation (HFCSA)}

{The HFCSA describes the evolution of the dynamical system \crefcomma{eq:exch,eq:loc dyn} when the exchange dynamics between hosts is much faster than the local dynamics of the hosts.
With a fixed interaction strength and} very frequent interactions, all
microbiome abundance vectors converge rapidly to the mean microbiome abundance vector
\begin{equation}
    \overline{\bm{N}}(t) = \frac{1}{|H|} \sum_{j = 1}^{|H|} \nsup{j}(t) \, .
\end{equation} 
Subsequently, these ``synchronized" microbiome abundance vectors each follow the mean of their local dynamics (see \cref{eq:dermean} below). 

\begin{theorem}[High-Frequency, Constant-Strength Approximation Theorem]\label{thm:HFCSA}
  Fix the relative interaction-frequency parameters $l_{ij}$, the interaction strength $\gamma > 0$, and a time $T$. Suppose that each local-dynamics function $g^{(i)}$ is Lipschitz continuous and bounded (see \Cref{sec:LD}). 
  Fix the arbitrary constants $\varepsilon \in (0,1]$, $\delta > 0$, and $\eta > 0$.
   For sufficiently large total-interaction-frequency parameter $\ltot$, each host microbiome abundance vector $\nsup{i}(t)$ satisfies
  \begin{equation}\label{eq:HFCSA bound}
      \normLinf{\nsup{i} - \widetilde{\bm{N}}}{[\eta,T]} < \delta
  \end{equation}
with probability larger than $1 - \varepsilon$, where
  \begin{align}\label{eq:HFCSA}
         \frac{\der \widetilde{\bm{N}}}{\der t} &= \frac{1}{|H|} \sum_{j = 1}^{|H|} g^{(j)}\left(\widetilde{\bm{N}}\right) \,, \\
         \widetilde{\bm{N}}(0) &= \overline{\bm{N}}(0) \,. \nonumber
\end{align}
\end{theorem}

In this section, we present key steps of the proof of \cref{thm:HFCSA}. We give a complete proof in 
\Cref{sec:HFCSA Proof}. 

When two hosts \hsupa{i} and \hsupa{j} interact at time $t_I$, they exchange portions of their microbiome (see \cref{eq:exch}). This exchange causes no change in the sum
\begin{align} \label{eq:Ex no d}
 	\nsup{i}\left(t_I^+\right) + \nsup{j}\left(t_I^+\right) &= (1 - \gamma) \nsup{i}\left(t_I^-\right) + \gamma \nsup{j}\left(t_I^-\right) + (1 - \gamma) \nsup{j}\left(t_I^-\right) + \gamma \nsup{i}\left(t_I^-\right) \\
    		&= \nsup{i}\left(t_I^-\right) + \nsup{j}\left(t_I^-\right) \,.   \nonumber
\end{align}
Therefore, interactions do not directly change the mean microbiome abundance vector $\overline{\bm{N}}(t)$. Instead, the local dynamics drive the evolution
\begin{equation} \label{eq:dermean}
    \frac{\der \overline{\bm{N}}}{\der t} = \frac{1}{|H|} \sum_{j = 1}^{|H|} g^{(j)}\left(\nsup{j}\right) \,.
\end{equation}

For sufficiently large total-interaction-frequency parameter $\ltot$, interactions occur on a much faster time scale than the local dynamics. 
Because of this separation of time scales, all microbiome abundance vectors $\nsup{i}(t)$ converge rapidly, which entails that
\begin{equation}
    \normLinf{\nsup{i} - \overline{\bm{N}}}{[\eta,T]} < \xi
\end{equation}
for some bound $\xi$ with high probability. For fixed probability, a larger $\ltot$ yields a smaller bound $\xi$. For a sufficiently small bound $\xi$, each abundance vector $\nsup{i}(t)$ is close enough to $\overline{\bm{N}}(t)$ so that $\overline{\bm{N}}(t)$ is well-approximated by {the approximate microbiome abundance vector $\widetilde{\bm{N}}(t)$ (see \cref{eq:HFCSA}). For times $t \in [\eta,T]$, each microbiome abundance vector $\nsup{i}(t)$ is very close to $\overline{\bm{N}}(t)$ and $\overline{\bm{N}}(t)$ is very close to $\widetilde{\bm{N}}(t)$. Therefore, on the interval $[\eta,T]$, each abundance vector 
$\nsup{i}(t)$ is well-approximated by $\widetilde{\bm{N}}(t)$.

As an example of the HFCSA, consider a two-host system in which each host has local dynamics \cref{eq:toy}. Let $\nsup{1}(0) = (2,2)$ and $\nsup{2}(0) = (12,12)$. In \cref{fig:HFSCA ex}, we show how the approximation improves as we increase the total-interaction-frequency parameter $\ltot = \lambda_{12}$ for fixed interaction strength $\gamma = 0.02$.

\begin{figure}[htbp]
  \centering
  \includegraphics[scale=0.45]{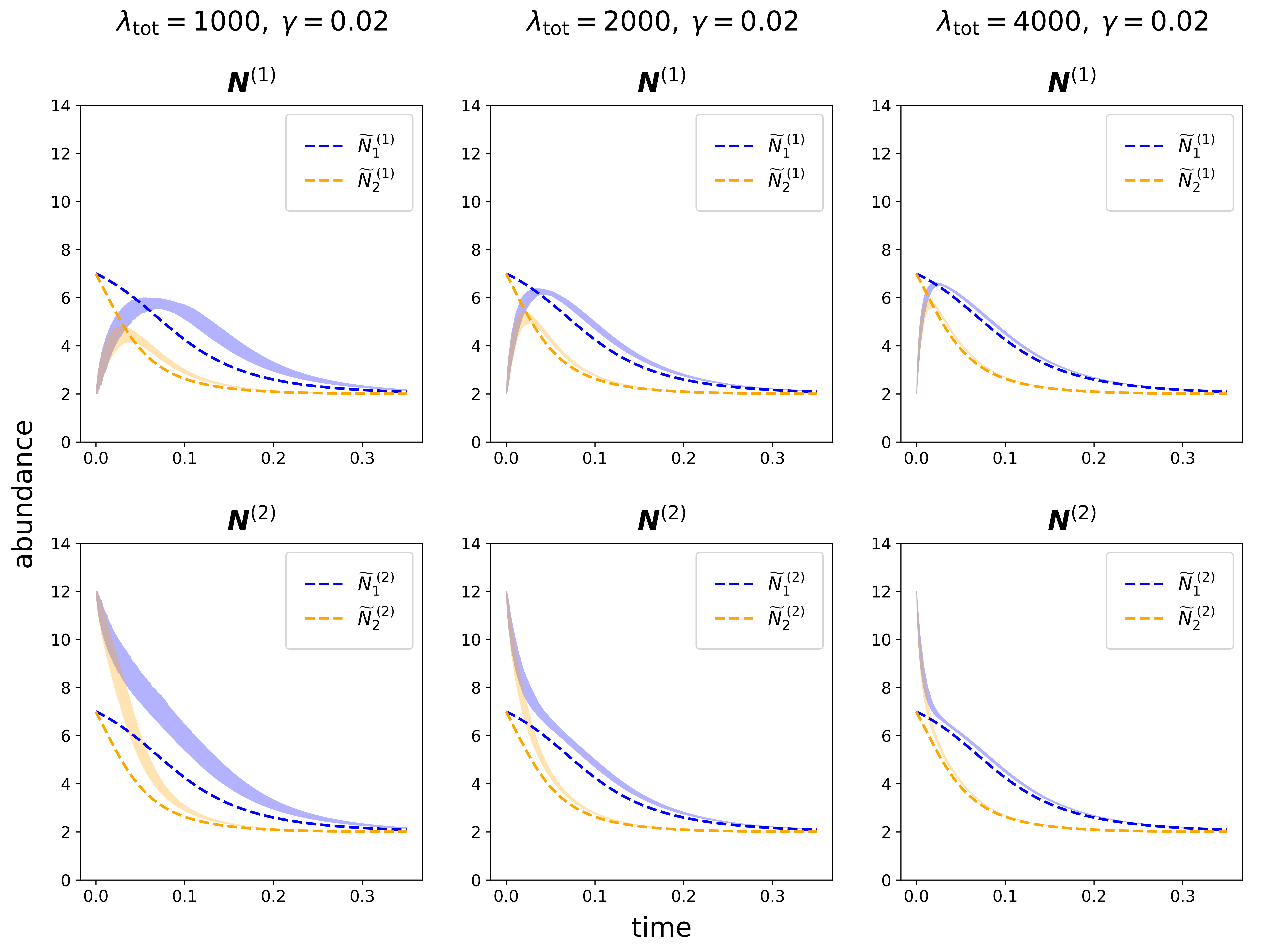}
  \caption{Numerical experiments for a two-host system in which each host has local dynamics \cref{eq:toy}. The three columns show experiments for different values of the total-interaction-frequency parameter $\lambda_{\mathrm{tot}}$ for fixed interaction strength $\gamma = 0.02$. We show the microbiome abundances for \hsupa{1} in the first row and the microbiome abundances for \hsupa{2} in the second row. We run 500 simulations for each set of parameters. The highlighted region shows the range between the $5^\mathrm{{th}}$ and $95^\mathrm{{th}}$ percentiles of the simulated host abundances. The dashed curves show the HFCSA for these experiments.}
  \label{fig:HFSCA ex}
\end{figure}



\section{Numerical Experiments {with a 10-Host System}} \label{sec:Num}

In this section, we present simulations for a system of $10$ hosts with local dynamics \cref{eq:toy}. We showed the interaction network for this system in \cref{fig:Interaction Network}. This network has $25$ edges, and we suppose that all relative interaction-frequency parameters $\lambda_{ij}$ are equal. For each edge $\left(\hsup{i},\hsup{j}\right) \in E$, the corresponding relative interaction-frequency parameter is $l_{ij} = {1}/{25}$. We explore the accuracy of our three approximations for a range of values of the total-interaction-frequency parameter $\ltot$ and the interaction strength $\gamma$. 
{We compare each of these approximations to numerical simulations of \crefcomma{eq:exch,eq:toy}. For each numerical simulation, we use a Gillespie algorithm~\cite{GA,Gill2007} to determine the interaction times for each host. We then use a Runge--Kutta 4-5 (RK45) ODE solver (via the function {\tt solve\_ivp} in {\sc SciPy} \cite{rk45}) to simulate the local dynamics between interactions.}


\subsection{Pair Approximation {of} the LFA} \label{sec:Pair App}

In the LFA, the approximate basin probability tensor $\widetilde\Psi(t)$ (see \cref{eq:LFA}) has dimension $m_1 \times m_2 \times \cdots \times m_{|H|}$. In many circumstances, this tensor is too large to analyze directly. Therefore, we use a pair approximation \cite{porter2016,Newman} of $\widetilde\Psi(t)$ for our calculations of the LFA in this {section}.

We approximate $\widetilde\Psi(t)$ by tracking the individual probabilities
\begin{equation}
    \psisupsubapp{i}{a}(t) = \text{ probability that host } \nsup{i}(t) \text{ is in basin } a \\
\end{equation}
and the dyadic (i.e., pair) probabilities
\begin{equation}
    \psisupsubapp{ij}{ab}(t) = \text{ probability that host } \nsup{i}(t) \text{ is in basin } a \text{ and host } \nsup{j}(t) \text{ is in basin } b \,.
\end{equation}
To obtain a pair approximation, we also need to consider the triadic (i.e., triplet) probabilities
\begin{align}
    \psisupsubapp{ijk}{abc}(t) &= \text{ probability that host } \nsup{i}(t) \text{ is in basin } a \text{, host } \nsup{j}(t) \text{ is in basin } b \text{,} \\ 
    							&\qquad\text{ and host } \nsup{k}(t) \text{ is in basin } c \,. \nonumber
\end{align}

In the LFA, if an interaction between hosts \hsupa{i} and \hsupa{j} causes their microbiome abundance vectors to move from basin $d$ to basin $a$ and from basin $e$ to basin $b$, respectively, then $\Phi^{(ij)}_{abde} = 1$. Otherwise, $\Phi^{(ij)}_{abde} = 0$. The probability that a given interaction is between \hsupa{i} and \hsupa{j} is $l_{ij}$. Therefore, the change in $\psisupapp{i}_a(t)$ due to an interaction at time $t_I$ is
\begin{equation}
    \psisupsubapp{i}{a}(t_I^+) - \psisupsubapp{i}{a}(t_I^-) = \sum_{j,b,e} \sum_{d\neq a} l_{ij} \Phi^{(ij)}_{abde} \psisupsubapp{ij}{de}(t_I^-) -\sum_{j,b,e} \sum_{d\neq a} l_{ij} \Phi^{(ij)}_{deab} \psisupsubapp{ij}{ab}(t_I^-) \,.
\end{equation}

Using frequency-scaled time, the expected number of interactions in a system during an interval $[t^*,t^* + dt^*]$ is $dt^*$. Therefore,
\begin{equation} \label{eq:single der}
    \frac{\der}{\der t^*}\psisupsubapp{i}{a} = \sum_{j,b,e} \sum_{d\neq a} l_{ij} \left[\Phi^{(ij)}_{abde} \psisupsubapp{ij}{de} - \Phi^{(ij)}_{dbae} \psisupsubapp{ij}{ae}\right] \,.
\end{equation}
The dyadic probabilities satisfy
\begin{align} \label{eq:pair der 1}
    \frac{\der}{\der t^*}\psisupsubapp{ij}{ab} &=   \sum_{e\neq b,d\neq a} l_{ij}\left[ \Phi^{(ij)}_{abde} \psisupsubapp{ij}{de} - \Phi^{(ij)}_{deab} \psisupsubapp{ij}{ab}\right] \nonumber\\ 
    &\qquad + \sum_{k,c,f} \sum_{d\neq a} l_{ik} \left[\Phi^{(ik)}_{acdf} \psisupsubapp{ijk}{dbf} - \Phi^{(ik)}_{dcaf} \psisupsubapp{ijk}{abf}\right] \\
    &\qquad + \sum_{k,c,f} \sum_{e\neq b} l_{jk} \left[\Phi^{(jk)}_{bcef} \psisupsubapp{ijk}{aef} - \Phi^{(jk)}_{ecbf} \psisupsubapp{ijk}{abf}\right] \nonumber \,.
\end{align}
The right-hand sides in \cref{eq:single der} and \cref{eq:pair der 1} are exact expressions. We form an approximation by replacing the triadic probabilities in \cref{eq:pair der 1} with combinations of dyadic probabilities. In the derivatives of $\psisupapp{ij}(t^*)$, when considering the impact of interactions between hosts \hsupa{i} and \hsupa{k}, we use the approximation
\begin{equation} \label{eq:this-approx}
    \psisupsubapp{ijk}{abc}(t^*) \approx \frac{\psisupsubapp{ij}{ab}(t^*) \psisupsubapp{ik}{ac}(t^*)}{\psisupsubapp{i}{a}(t^*)} \,,
\end{equation}
which assumes that $\psisupsubapp{j}{b}$ and $\psisupsubapp{k}{c}$ are independent. Inserting the approximation \cref{eq:this-approx}
into \cref{eq:pair der 1} gives
\begin{align} \label{eq:pair der 2}
    \frac{\der}{\der t^*}\psisupsubapp{ij}{ab} \approx &  \sum_{e\neq b,d\neq a} l_{ij} \left[\Phi^{(ij)}_{abde} \psisupsubapp{ij}{de} - \Phi^{(ij)}_{deab} \psisupsubapp{ij}{ab}\right] \\ 
    &\qquad + \sum_{k,c,f} \sum_{d\neq a} l_{ik} \left[\Phi^{(ik)}_{acdf} \frac{\psisupsubapp{ij}{db} \psisupsubapp{ik}{df}}{\psisupsubapp{i}{d}} - \Phi^{(ik)}_{dcaf} \frac{\psisupsubapp{ij}{ab} \psisupsubapp{ik}{af}}{\psisupsubapp{i}{a}}\right] \nonumber\\
    &\qquad + \sum_{k,c,f} \sum_{e\neq b} l_{jk} \left[\Phi^{(jk)}_{bcef} \frac{\psisupsubapp{ij}{ae} \psisupsubapp{jk}{ef}}{\psisupsubapp{j}{e}} - \Phi^{(jk)}_{ecbf} \frac{\psisupsubapp{ij}{ab} \psisupsubapp{jk}{bf}}{\psisupsubapp{j}{b}}\right]b\nonumber \,.
\end{align}

Equations \cref{eq:single der,eq:pair der 2} constitute a pair approximation of the evolution of $\widetilde\Psi(t^*)$. In our simulations (see \Cref{sec:LFA Num}), we compare the individual probabilities $\psisupsubapp{i}{a}(t^*)$ from our pair approximation \crefcomma{eq:single der,eq:pair der 2} with the fraction $\psi^{(i),\,\text{sim}}_a(t^*)$ of simulations in which $\nsup{i}(t^*)$ is in basin of attraction $a$.


\subsection{Simulations for the Low-Frequency Approximation} 
\label{sec:LFA Num}

In this subsection, we show numerical results for the LFA \cref{eq:LFA}. We use the pair approximation \crefcomma{eq:single der,eq:pair der 2} to determine the individual probabilities $\psisupsubapp{i}a(t^*))$. As we described in \Cref{sec:Pair App}, an individual probability $\psisupsubapp{i}a(t^*)$ is an approximation of the probability from the LFA that $\nsup{i}(t^*)$ is in basin of attraction $a$. We compare these individual probabilities to the 
{empirical probabilities} $\psi^{(i),\,\text{sim}}_a(t^*)$, which are the fractions of simulations of \crefcomma{eq:exch,eq:toy} in which $\nsup{i}(t^*)$ is in basin of attraction $a$. We perform these approximations and simulations for 59 linearly spaced interaction strengths $\gamma \in [0,0.5]$ and 13 logarithmically spaced values of the total-interaction-frequency parameter $\ltot \in [2.5\times10^{-2},2.5\times10^4]$. The boundary set is $\mathcal{B} = \{0.1,0.4\}$ for this system, so the LFA is not valid when $\gamma = 0.1$ or $\gamma = 0.4${, so we exclude these values of $\gamma$ from our simulations. Accordingly, we perform simulations for values of $\gamma \in [0,0.5]$ that are multiples of $\frac{0.5}{60}$ (except for $\gamma = 0.1$ and $\gamma = 0.4$).} For each pair of $\gamma$ and $\ltot$ values, we select a random four-dimensional (4D) vector $\psisupapp{i}(0)$ from the Dirichlet distribution\footnote{The Dirichlet distribution is a multivariate generalization of the Beta distribution. The distribution $\text{Dir}(1,1,1,1)$ gives a {prior that is symmetric with respect to {the} initial basins of attraction.}}
$\text{Dir}(1,1,1,1)$ \cite{MacKay05} for each host. The entries of each of these 4D vectors sum to $1$. For each simulation, we set each $\nsup{i}(0)$ to be the stable equilibrium point in one of the basins of attraction. For each basin of attraction $a$, the initial microbiome abundance vector $\nsup{i}(0)$ is in that basin of attraction with probability $\psisupsubapp{i}{a}(0)$.

We perform 1000 simulations of \crefcomma{eq:exch,eq:toy} for each pair of $\gamma$ and $\ltot$ values, and we calculate the fraction $\psi^{(i),\,\text{sim}}_a(t^*)$ of the simulations in which $\nsup{i}(t^*)$ is in basin of attraction $a$. We compare $\psi^{(i),\,\text{sim}}_a(t^*)$ to $\psisupapp{i}(t^*)$, which we calculate using the pair approximation \crefcomma{eq:single der,eq:pair der 2}. We calculate $\boldsymbol{\psi}^{(i),\,\text{sim}}(t^*)$ and $\psisupapp{i}(t^*)$ for 1001 evenly spaced frequency-scaled times $t_k^* = {k}/{500}$ in the interval $[0,2]$. In \cref{fig:LFA Error}, we plot the error 
\begin{equation} \label{eq:LFA error}
    \text{Error} = \frac{2}{1001} \sum_{k = 1}^{1001} \sqrt{\sum_{i = 1}^{10} \sum_{a = 1}^4 \left(\psi^{(i),\,\text{sim}}_a(t_k^*) - \psisupsubapp{i}{a}(t_k^*)\right)^2 }
\end{equation}
for each pair of $\gamma$ and $\ltot$ values. The error \cref{eq:LFA error}  is a discrete approximation of the $2$-norm \normLtwoa{\boldsymbol{\psi}^{(i),\,\text{sim}} - \psisupapp{i}}{[0,T^*]}.

\begin{figure}[tbp]
  \centering
  \includegraphics[scale=0.6]{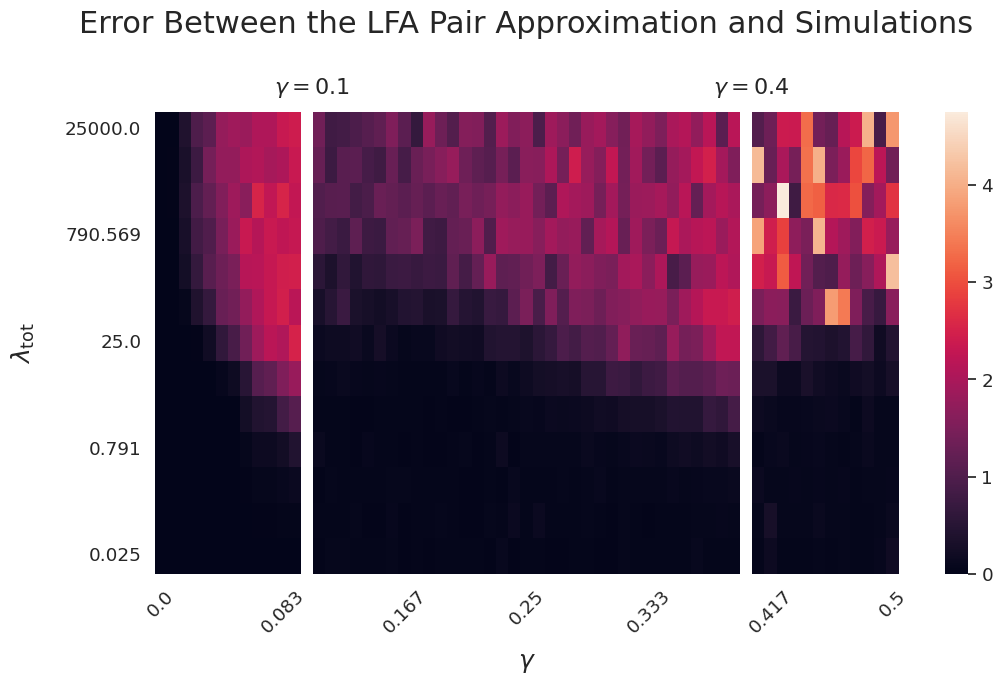}
  \caption{The error \cref{eq:LFA error} between the {pair approximation \crefcomma{eq:single der,eq:pair der 2} 
  of the} LFA {(see \cref{eq:LFA})} versus means of 1000 simulations of \crefcomma{eq:exch,eq:toy} for each pair of interaction strength $\gamma$ and total-interaction-frequency parameter $\ltot$. 
  We plot $\gamma$ on a linear scale and $\ltot$ on a logarithmic scale. We do not plot errors for $\gamma = 0.1$ and $\gamma = 0.4$ because the LFA is not valid for these values.}
  \label{fig:LFA Error}
\end{figure}

{Our approximation} is most accurate when the total-interaction-frequency parameter $\ltot$ is small. Its accuracy is also best when the interaction strength $\gamma$ is not near $0.1$ or $0.4$. In our experiments, there are 
two ways in which the the dynamics of the pair approximation of the LFA differs from the dynamics of direct numerical simulations of the dynamical system \crefcomma{eq:exch,eq:toy}. In both situations, we believe that these discrepancies result from the LFA, rather than from the subsequent pair approximation. The first type of error arises when repeated interactions occur in sufficiently quick succession to yield a transition that the LFA misses. For example, when $\gamma \approx 0.342$, the LFA does not predict that the $\nsup{i}(t)$ can move from basin $2$ to basin $1$. For sufficiently large values of $\ltot$, repeated interactions in quick succession are common (and not merely possible), which causes the LFA to overestimate the probability that a host is in basin $2$ and underestimate the probability that a host is in basin $1$. In \cref{fig:LFA Type 1}, we illustrate this type of error for $\gamma \approx 0.342$ and several values of $\ltot$.
The second type of error arises when repeated interactions in sufficiently quick succession cause the LFA to overestimate the impact of the second and subsequent interactions. As an example, when the interaction strength is
$\gamma \approx 0.433$, the LFA predicts that $\nsup{j}(t)$ will be in basin $1$ after an interaction whenever \hsupa{j} interacts with \hsupa{i} and 
$\nsup{i}(t)$ is in basin $1$ before the interaction. This prediction arises because the LFA assumes that $\nsup{i} = (2,2)$ before this interaction. However, if \hsupa{i} recently interacted with a different host, then $\nsup{i}(t)$ 
can be in basin $1$ but {sufficiently far from} the equilibrium point $(2, 2)$ {such that $\nsup{j}(t)$ remains in basin $2$ after \hsupa{j} interacts with \hsupa{i}.}
Consequently, the LFA overestimates the probability that a host is in basin $1$. In \cref{fig:LFA Type 2}, we illustrate this type of error for $\gamma \approx 0.433$ and several values of $\ltot$.

\begin{figure}[tbp]
  \centering
  \includegraphics[scale=0.4]{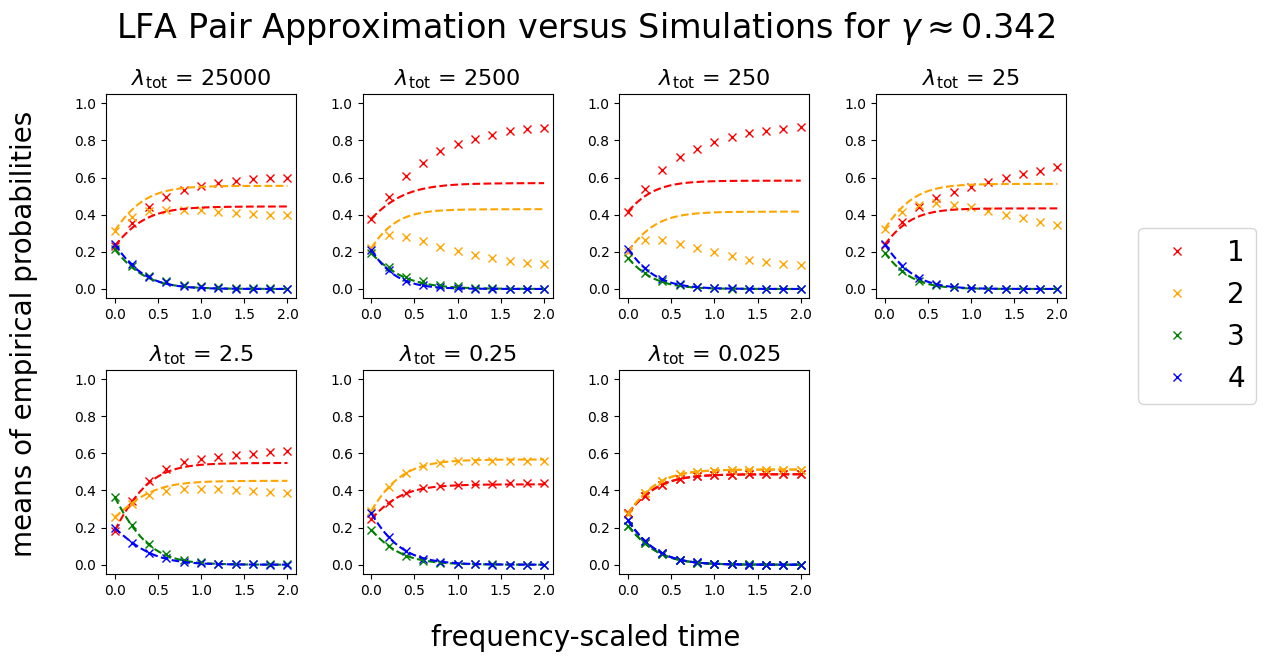}
  \caption{The means of the {empirical} probabilities $\psi^{(i),\,\mathrm{sim}}_a(t^*)$ of
  all hosts for interaction strength $\gamma \approx 0.342$ and several values of the total-interaction-frequency parameter $\ltot$. The dashed curves indicate the means of the probabilities $\psisupapp{i}_a(t^*)$ 
  of all hosts from the pair approximation \crefcomma{eq:single der,eq:pair der 2} 
  of the LFA (see \cref{eq:LFA}).}
  \label{fig:LFA Type 1}
\end{figure}

\begin{figure}[htbp]
  \centering
  \includegraphics[scale=0.4]{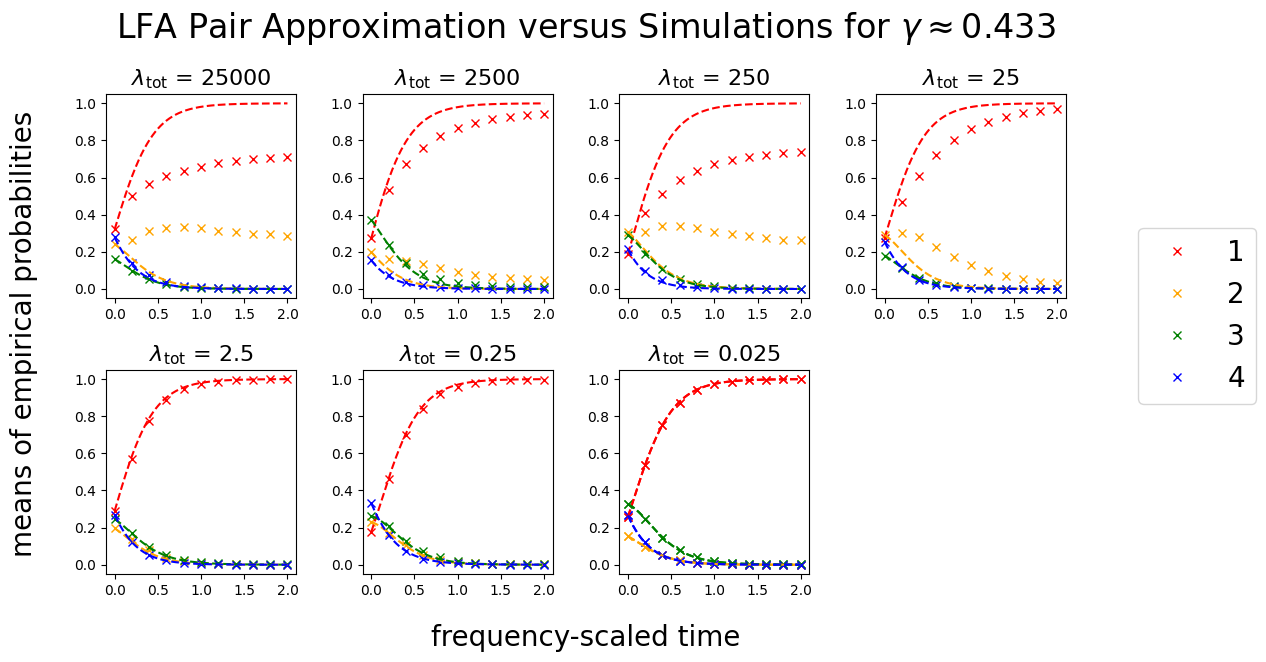}
  \caption{The means of the {empirical} probabilities $\psi^{(i),\,\mathrm{sim}}_a(t^*)$ of
  all hosts for interaction strength $\gamma \approx 0.433$ and several values of the total-interaction-frequency parameter $\ltot$. The dashed curves indicate the means of the probabilities $\psisupapp{i}_a(t^*)$ 
  of all hosts from the pair approximation \crefcomma{eq:single der,eq:pair der 2} 
  of the LFA (see \cref{eq:LFA}).}
  \label{fig:LFA Type 2}
\end{figure}


\subsection{Simulations for the High-Frequency Approximations}

In this subsection, we show numerical results for the HFLSA \cref{eq:HFLSA} and the HFCSA \cref{eq:HFCSA}. We compare the approximate microbiome abundance vectors from these two approximations to simulations of \crefcomma{eq:exch,eq:toy}.

To evaluate the accuracy of the approximate microbiome abundance vectors $\nsupapp{i}(t)$ that we obtain from the HFLSA \cref{eq:HFLSA}, we compare them to the microbiome abundance vectors $\nsup{i}(t)$ from \crefcomma{eq:exch,eq:toy}. We perform these simulations for 13 logarithmically spaced values of the total-interaction-frequency parameter $\ltot \in [25,2500]$ and 75 linearly spaced values of $\ltot \gamma \in [0.04,0.3]$. We use the product $\ltot \gamma$ as a parameter instead of the interaction strength $\gamma$ on its own to illustrate the improvement of the HFLSA as we increase $\ltot$ for fixed $\ltot \gamma$. For each pair of $\ltot$ and $\ltot \gamma$ values, we perform 1000 simulations over the time interval $[0,1]$ with initial conditions
\begin{align} \label{eq:init cond}
    \nsup{1}(0) &= (12,12)\,, \ \nsup{2}(0) = (2,2)\,, \ \nsup{3}(0) = (12,2) \,, \\
    \nsup{4}(0) &= (2,2)\,, \ \nsup{5}(0) = (12,12) \,, \ \nsup{6}(0) = (12,12)\,,  \notag \\
    \nsup{7}(0) &= (2,2) \,, \ \nsup{8}(0) = (12,2) \,, \ \nsup{9}(0) = (2,12) \,, \ \nsup{10}(0) = (2,12) \,. \nonumber
\end{align}
The microbiome abundance vector $\bm{N}^{(i),l}(t)$ is the $l$th simulated microbiome abundance vector for host \hsupa{i}. 

For each of our simulations, we calculate the microbiome abundance vectors $\bm{N}^{(i),l}(t)$ for 101 evenly spaced times $t_k = {k} / {100}$. {We compare these simulations to the approximate microbiome abundance vector $\nsupapp{i}(t)$ from the HFLSA \cref{eq:HFLSA}.} In \cref{fig:HFLSA Error}, we plot
\begin{equation} \label{eq:HFA error}
    \text{Error} = \frac{1}{101 \times 1000} \sum_{l = 1}^{1000} \sum_{k = 1}^{101} \sqrt{\sum_{i = 1}^{10} \sum_{a = 1}^2 \left(\bm{N}^{(i),l}_a(t_k) - \nsupapp{i}_a(t_k)\right)^2 }
\end{equation}
for each pair of $\ltot$ and $\ltot \gamma$ values. The error \cref{eq:HFA error} is a discrete approximation of the mean of \normLtwoa{\bm{N}^{(i),l}_a(t_k) - \nsupapp{i}_a(t_k)}{[0,1]} of
all simulations. For any fixed value of $\ltot \gamma$, the HFLSA is more accurate for larger total-interaction-frequency parameter $\ltot$. However, for a fixed value of $\ltot$, the error depends significantly on the value of $\ltot \gamma$. Each value of $\ltot \gamma$ yields a set $\left\{\nsupapp{i}(1)\right\}$ of final approximate microbiome abundance vectors. For all but a finite set of values of $\ltot \gamma$, each final approximate microbiome abundance vector $\nsupapp{i}(1)$ changes continuously with $\ltot \gamma$. However, there are a finite number of $\ltot \gamma$ values for which some final approximate microbiome abundance vector $\nsupapp{i}(1)$ has a 
jump. For the initial set of microbiome abundance vectors \cref{eq:init cond}, these jumps occur at
\begin{equation} \label{eq:phase tran}
    \ltot \gamma \in \{0.0866,0.1069,0.1160,0.1161,0.1416,1.6432,1.7174,1.7425,1.8187,1.8363\} \, .
\end{equation} 
The regions in which the HFLSA performs worst in our simulations are near $\ltot \gamma \approx 0.1$ and $\ltot \gamma \approx 1.8$, which are very close to several of the values in \cref{eq:phase tran}.

\begin{figure}[tbp]
  \centering
  \includegraphics[scale=0.6]{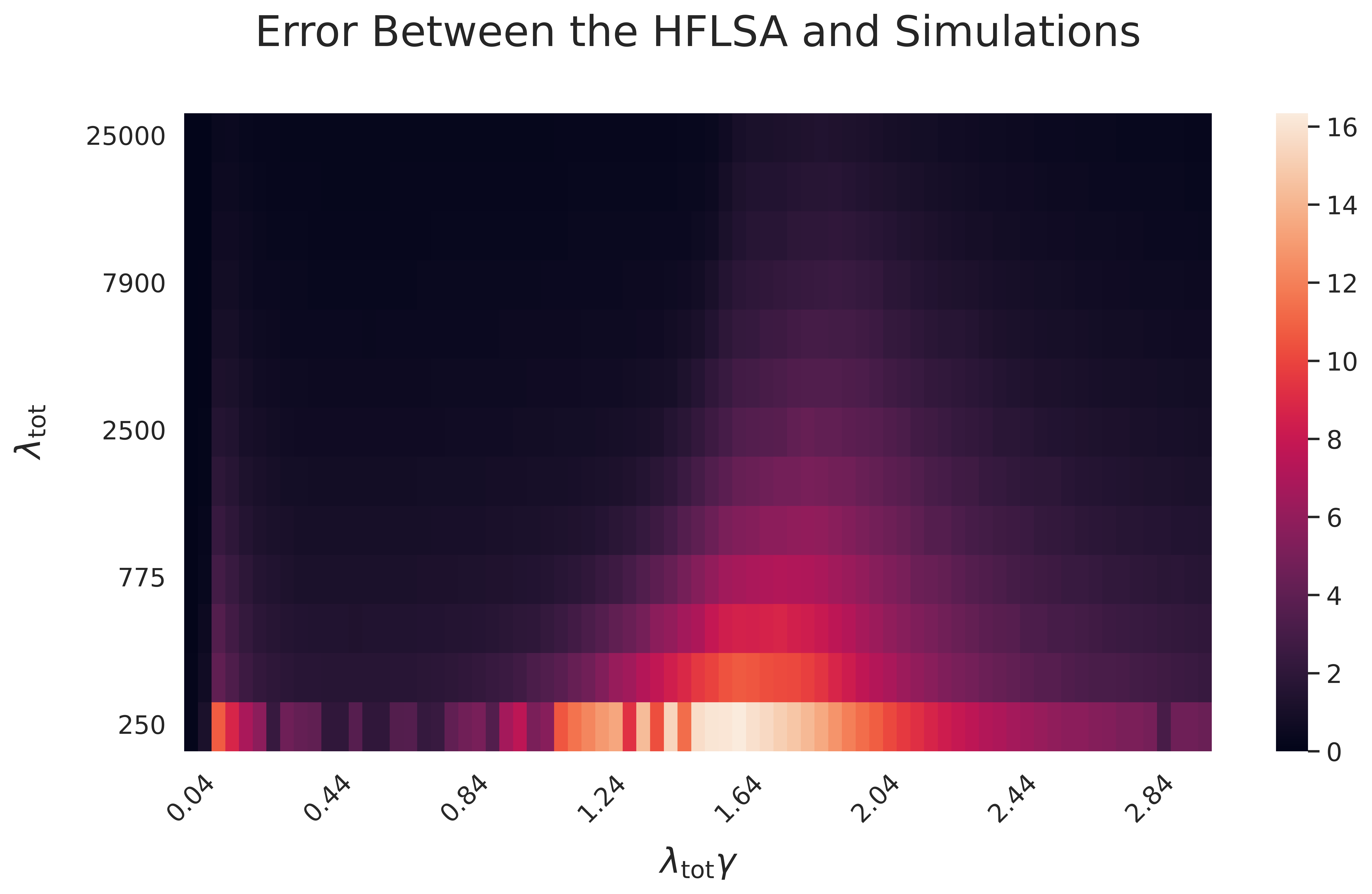}
  \caption{The mean error \cref{eq:HFA error} between the approximate microbiome abundance vectors $\left\{\nsupapp{i}(t)\right\}$ from the HFLSA (see \cref{eq:HFLSA}) and the microbiome abundance vectors 
  $\left\{\nsup{i}(t)\right\}$ for 1000 simulations of \crefcomma{eq:exch,eq:toy}. We plot $\ltot$ on a logarithmic scale and plot $\ltot \gamma$ on a linear scale.} 
  \label{fig:HFLSA Error}
\end{figure}

For the HFCSA \cref{eq:HFCSA}, we perform simulations for 61 linearly spaced values of the interaction strength $\gamma \in [0,0.5]$ and 13 logarithmically spaced values of the total-interaction-frequency parameter 
$\ltot \in [2.5\times10^{-1},2.5\times10^4]$. We use the same initial conditions \cref{eq:init cond} as in our HFLSA simulations. We again {perform 1000 simulations and evaluate each simulated microbiome abundance vector $\bm{N}^{(i),l}(t)$ for 101 evenly spaced times $t_k = {k} / {100}$} on the time interval $[0, 1]$. In \cref{fig:HFCSA Error}, we show the error \cref{eq:HFA error} for each pair of $\gamma$ and $\ltot$ values. The approximation is accurate for sufficiently large $\ltot$. In general, a larger $\gamma$ increases the rate at which each microbiome abundance vector $\nsup{i}(t)$ converges to the mean microbiome abundance vector $\overline{\bm{N}}(t)$. Consequently, the HFCSA yields a better approximation of $\nsup{i}(t)$ for larger values of $\gamma$.

\begin{figure}[htbp]
  \centering
  \includegraphics[scale=0.6]{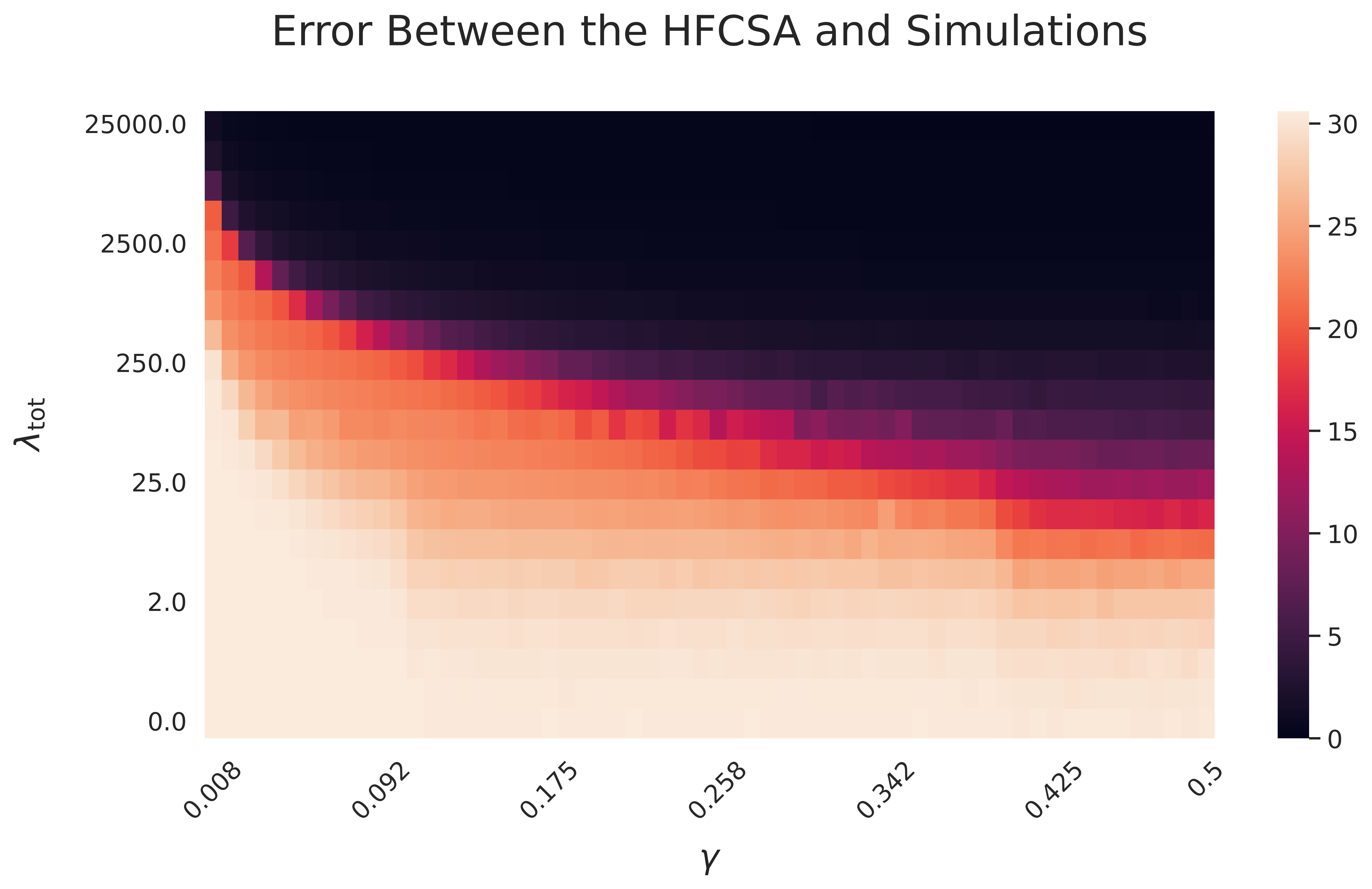}
  \caption{The mean error \cref{eq:HFA error} between the approximate microbiome abundance vectors $\left\{\nsupapp{i}(t)\right\}$ from the HFCSA (see \cref{eq:HFCSA}) and the microbiome abundance vectors $\left\{\nsup{i}(t)\right\}$ for 1000 simulations of \crefcomma{eq:exch,eq:toy}. We plot $\ltot$ on a logarithmic scale and plot $\gamma$ on a linear scale.}
  \label{fig:HFCSA Error}
\end{figure}



\section{Conclusions and Discussion} \label{sec:Conc}


\subsection{Summary}

We developed a novel framework to model the microbiome dynamics of living hosts that incorporates both local dynamics within {a locale} and exchanges of microbiomes between {locales}. Our framework extends existing metacommunity theory by accounting for the discrete nature of host interactions. Unlike classical mass-effects models, our framework incorporates two distinct parameters that control interaction frequencies and interaction strength. Using both analytical approximations and numerical computations, we demonstrated that both {interaction frequencies and interaction strengths} are necessary to determine microbiome dynamics. 

We developed approximations in three parameter regions, and we proved their accuracy in those regions. In \cref{tab:summary}, we summarize these approximations.
Our low-frequency approximation (LFA) gives a good approximation of microbiome dynamics when local dynamics are much faster than host interactions. Our high-frequency, low-strength approximation (HFLSA) encodes the dynamics of a system when interactions are frequent but weak, resulting in a model with the same form as the mass-effects model \cref{eq:meta}. Finally, our high-frequency, constant-strength approximation (HFCSA) accurately predicts the rapid convergence of all hosts' microbiome dynamics {to synchronized dynamics} when interactions are frequent and have constant interaction strength $\gamma$. 
We validated our three approximations through numerical experiments on an illustrative model of microbiome dynamics for a range of parameter values.

\begin{table}[ht!]
\centering
\begin{tabular}{|p{3cm}|p{3.5cm}|p{7.5cm}|}
\hline
\textbf{Approximation} & \textbf{Parameter region} & \textbf{Description} \\
\hline

Low-frequency approximation (LFA) &
Total-interaction-frequency parameter  $\ltot \to 0$ for fixed interaction strength $\gamma$ &
Hosts' microbiomes evolve independently for long times
with occasional microbiome-exchange events. The local microbiome dynamics are near 
attracting equilibrium points
when each interaction occurs. \\
\hline

High-frequency, low-strength \newline approximation (HFLSA) &
Total-interaction-frequency parameter $\ltot \to \infty$ and interaction strength $\gamma \to 0$ for fixed $\ltot \gamma$ &
Frequent but small microbiome exchanges that effectively produce continuous {dispersal between interacting hosts}. This effectively yields a mass-effects model of the form \cref{eq:meta}.
 \\
\hline

High-frequency, constant-strength approximation (HFCSA) &
Total-interaction-frequency parameter  $\ltot \to \infty$ for fixed interaction strength $\gamma$ &
Hosts interact very frequently and exchange substantial proportions of their microbiome in each interaction. This leads to {the} rapid {convergence of the
microbiome state vectors of the
hosts {to synchronized dynamics}.} \\ 
\hline
\end{tabular}
\vspace{1 em}
\label{tab:summary}
\caption{Summary of the three approximations that we examined in our modeling framework.}
\end{table}

{Using our modeling framework, we examined the probability that all microbiome abundance vectors converge by some fixed time.}
This probability depends both on the interaction-frequency parameters and on the interaction strength. Using the LFA, we showed for sufficiently small interaction-frequency parameters that the convergence probability depends on whether the interaction strength $\gamma$ is large enough that a single interaction between two hosts places their microbiome abundance vectors in the same basin of attraction. By contrast, using the HFLSA, we showed for sufficiently large interaction-frequency parameters that the convergence probability depends on the product $\ltot \gamma$ of the total-interaction-frequency parameter $\ltot$ and the interaction strength $\gamma$. For intermediate values of the interaction-frequency parameters, we used numerical simulations 
to examine convergence probabilities.


\subsection{Outlook} \label{sec:FW}

Our modeling framework provides a foundation for many promising research directions in microbiome dynamics. 
As an illustration, in the present paper, we examined idealized examples
of microbiome local dynamics and host interaction networks.
Realistic models of microbiome systems are more complicated and higher-dimensional than our examples, and they also have 
more intricate basin-of-attraction structures \cite{MathModelMicro}. Additionally, the values of the interaction strength 
in our examples are larger than those in most real-world interactions.
(For example, a 10-second intimate kiss 
yields a microbiome-exchange proportion between about 0.035 and 0.05 \cite{kiss}.) In our examples, we used an idealized model and large interaction strengths to ensure that our explanations and plots are clear.
Although we did not examine realistic microbiome models, one can use our framework
to study the effects of host interactions in many ecological models of local dynamics.
{In particular, our modeling framework provides a natural setting to investigate how the discrete nature of host interactions can lead to qualitatively different microbiome dynamics than 
{continuous host interactions in} the mass-effects paradigm.}
One can also use our framework to study the impact of the structure of interaction networks on microbiome dynamics.

There are many possible extensions of our modeling framework. For example, we considered a homogeneous interaction strength $\gamma$ for simplicity. However, one can allow each pair of hosts to have different interaction strengths $\gamma_{ij}$. Additionally, hosts can exchange different microbe species with different exchange strengths, so one can also consider interaction strengths $\gamma_{ijk}$ that encode the exchange of different species $k$ when hosts \hsupa{i} and \hsupa{j} interact. This seems helpful when using consumer--resource models for the local dynamics in a locale. In this case, it also may be desirable to separately encode the strengths of microbiome exchange and resource exchange. Another way to extend our framework is to relax our assumption of instantaneous microbiome exchange by instead employing rapid but continuous functions.
Additionally, in our framework, the LFA assumes that the attractors of each local-dynamics function $\gi$ consist of a finite set of stable equilibrium points. We believe that it is possible to extend the LFA to systems in which the functions $\gi$ have more complicated attractors, such as limit cycles and chaotic attractors. We also believe that it is possible to generalize the HFLSA and HFCSA to systems in which the times between consecutive interactions between 
the same adjacent hosts follow a distribution other than an exponential distribution.



\appendix



\section{Proof of the Low-Frequency Approximation (LFA) Theorem} \label{sec:LFA Proof}
In this section, we prove the LFA Theorem (see \cref{thm:LFA}).

\begin{theorem}[Low-Frequency-Approximation Theorem] 
  Suppose that the attractors of each host's local dynamics consist of a finite set of stable equilibrium points at which the local-dynamics function $\gi$ is inward pointing, and let each $\gi$ be continuous and bounded (see \Cref{sec:LD}).
  Fix $\gamma \not\in \mathcal{B}$, all $l_{ij}$, and a frequency-scaled time $T^*$. As the total-interaction-frequency parameter $\lambda_{\mathrm{tot}} \to 0$, the basin probability tensor $\Psi(t^*)$ converges uniformly to $\widetilde{\Psi}(t^*)$ on $[0,T^*]$, where
  \begin{align}
      \frac{\der}{\der t^*}\widetilde\Psi_{b_1,\ldots,b_{|H|}} (t^*) &= \Phi_{b_1,\ldots,b_{|H|},a_1,\ldots,a_{|H|}} \widetilde\Psi_{a_1,\ldots, a_{|H|}}(t^*) - \widetilde\Psi_{b_1,\ldots,b_{|H|}}(t^*) \,, \\
      \widetilde{\Psi}(0) &= \Psi(0) \,. \nonumber
  \end{align}
\end{theorem}

\begin{proof}
Each local-dynamics function $g^{(i)}$ is bounded {(see \Cref{sec:LD}),
so there} exists a constant $M$ such that each entry of $\nsup{i}(t)$ is nonnegative and
\begin{equation}
 	\norminf{\nsup{i}(t)} \leq M  
\end{equation}
for each microbiome abundance vector $\nsup{i}(t)$ and all times $t \geq 0$.

Consider an arbitrary but fixed $\varepsilon > 0$. We will show that 
\begin{equation}
      \norminf{\Psi(t^*) - \widetilde{\Psi}(t^*)} < \varepsilon 
\end{equation}
for sufficiently small total-interaction-frequency parameter $\ltot$ and all frequency-scaled times $t^* \in [0,T^*]$.

For each $i$, let $\mathcal{A}_i$ be the set of stable equilibrium points of the local dynamics of host \hsupa{i}. Suppose that adjacent hosts \hsupa{i} and \hsupa{j} interact at time $t_I$ and that $\nsup{i} \left(t_I^-\right)$ and $\nsup{j} \left(t_I^-\right)$ are at stable equilibrium points $\bm{a}^{(i)} \in \mathcal{A}_i$ and $\bm{a}^{(j)} \in \mathcal{A}_j$, respectively. For $\bm{x}, \bm{y} \in [0,M]^n$, let
\begin{equation}
    \mathcal{X}\left(\bm{x},\bm{y}\right) = (1 - \gamma) \bm{x} + \gamma \bm{y} \,.
\end{equation}
After the interaction, $\nsup{i} \left(t_I^+\right) = \mathcal{X}\left(\bm{a}^{(i)}, \bm{a}^{(j)}\right)$ and $\nsup{j} \left(t_I^+\right) = \mathcal{X}\left(\bm{a}^{(j)}, \bm{a}^{(i)}\right)$. Because $\gamma \not\in \mathcal{B}$, it follows that $\mathcal{X}\left(\bm{a}^{(i)}, \bm{a}^{(j)}\right)$ and $\mathcal{X}\left(\bm{a}^{(j)}, \bm{a}^{(i)}\right)$ are each in a basin of attraction, which we denote by $b_i$ and $b_j$, respectively, of their associated stable equilibrium points $\bm{b}^{(i)} \in \mathcal{A}_i$ and $\bm{b}^{(j)} \in \mathcal{A}_j$.

Let
\begin{align}
    B(\bm{x},\delta) &= \left\{ \bm{y} \mid \normtwo{\bm{y} - \bm{x}} < \delta \,\text{ and }\, \bm{y} \in [0,M]^n \right\} \,, \\
    \overline{B}(\bm{x},\delta) &= \left\{ \bm{y} \mid \normtwo{\bm{y} - \bm{x}}\leq\delta \,\text{ and }\, \bm{y} \in [0,M]^n \right\} \,.
\end{align}
The basins of attraction of stable equilibrium points are open sets, so there exists \linebreak $\delta(\bm{a}^{(i)},\bm{a}^{(j)}) = \delta(\bm{a}^{(j)},\bm{a}^{(i)})$ such that
\begin{align}
    &B\left(\mathcal{X}\left(\bm{a}^{(i)}, \bm{a}^{(j)}\right), \delta(\bm{a}^{(i)},\bm{a}^{(j)})\right) \subseteq b_i \,, \\
    &B\left(\mathcal{X}\left(\bm{a}^{(j)}, \bm{a}^{(i)}\right), \delta(\bm{a}^{(i)},\bm{a}^{(j)})\right) \subseteq b_j \nonumber \,.
\end{align}
All $\mathcal{A}_i$ are finite, so there are minima
\begin{equation}
    \delta_{ij} = \min_{\bm{a}^{(i)} \in \mathcal{A}_i} \min_{\bm{a}^{(j)} \in \mathcal{A}_j}  \delta(\bm{a}^{(i)},\bm{a}^{(j)}) \,.
\end{equation}

Let the flow $\bm{X}^{(i)}(t,\bm{x)}$ be the solution of
\begin{align}
    \frac{\partial \bm{X}^{(i)}}{\partial t}(t,\bm{x}) &= \gi\left(\bm{X}^{(i)}(t,\bm{x})\right) \,, \\
    \bm{X}^{(i)}(0,\bm{x}) &= \bm{x} \,. \nonumber
\end{align}

For each local-dynamics function $\gi$, each $\bm{a}^{(i)} \in \mathcal{A}_i$ is a stable equilibrium point at which $\gi$ is inward pointing. Therefore, there exists $\delta\left(\bm{a}^{(i)}\right)$ such that $\gi(\bm{y}) \cdot \left(\bm{a}^{(i)} - \bm{y}\right) > 0$ for $\bm{y} \in \overline{B}\left(\bm{a}^{(i)},\delta\left(\bm{a}^{(i)}\right)\right)$.
If $\bm{x}$ is in the basin of attraction of $\bm{a}^{(i)}$, then $\bm{X}(t_a,\bm{x}) \in B\left(\bm{a}^{(i)},\delta\left(\bm{a}^{(i)}\right)\right)$ for some time $t_a$. We then have
\begin{align} \label{eq:trajder}
	\left[\frac{\partial}{\partial t} \normtwo{\bm{X}^{(i)} - \bm{a}^{(i)}}^2\right](t_a,\bm{x}) &= \left[2 \frac{\partial \bm{X}^{(i)}}{\partial t} \cdot \left(\bm{X}^{(i)} - \bm{a}^{(i)}\right)\right](t_a,\bm{x}) \,, \\
		&= -2\gi\left(\bm{X}^{(i)}(t_a,\bm{x})\right) \cdot \left(\bm{a}^{(i)} - \left(\bm{X}^{(i)}(t_a,\bm{x})\right)\right) < 0 \,. \nonumber
\end{align}
Consequently, $\normtwo{\bm{X}^{(i)}(t,\bm{x}) - \bm{a}^{(i)}}$ is monotonically decreasing for $t \geq t_a$.
There are only finitely many hosts, so there exists
\begin{align}
    \delta = \frac{1}{2} \min \left\{\Delta_1 \cup \Delta_2\right\} \, , 
\end{align}
with
\begin{align}    
    \Delta_1 &= \bigcup_i \bigcup_{\bm{a}^{(i)} \in \mathcal{A}_i} \left\{\delta\left(\bm{a}^{(i)}\right)\right\} \,,  \\
    \Delta_2 &= \bigcup_{\{i,j \ \mid \ \left(\hsup{i},\hsup{j}\right) \in E\} } \left\{\delta_{ij} \right\} \,. \notag
\end{align}

For any $\bm{x}, \bm{y} \in [0,M]^n$, let
\begin{equation}
    \mathcal{F}\left(\bm{x}, \bm{y}\right) = \overline{B}\left(\mathcal{X}\left(\bm{x}, \bm{y}\right),\delta \ \right) \,. 
\end{equation}
For each pair of adjacent hosts \hsupa{i} and \hsupa{j} and each $\bm{a}^{(i)} \in \mathcal{A}_i$ and $\bm{a}^{(j)} \in \mathcal{A}_j$, we have
\begin{align}
    \mathcal{F}\left(\bm{a}^{(i)}, \bm{a}^{(j)}\right) &= \overline{B}\left(\mathcal{X}\left(\bm{a}^{(i)}, \bm{a}^{(j)}\right), \delta \ \right)\subset B\left(\mathcal{X}\left(\bm{a}^{(i)}, \bm{a}^{(j)}\right), \delta(\bm{a}^{(i)},\bm{a}^{(j)})\right) \subseteq b_i \,, \\
    \mathcal{F}\left(\bm{a}^{(j)}, \bm{a}^{(i)}\right) &= \overline{B}\left(\mathcal{X}\left(\bm{a}^{(j)}, \bm{a}^{(i)}\right), \delta \ \right) \subset B\left(\mathcal{X}\left(\bm{a}^{(j)}, \bm{a}^{(i)}\right),\delta(\bm{a}^{(i)},\bm{a}^{(j)})\right) \subseteq b_j \nonumber \ 
\end{align}
for some basins of attraction $b_i$ and $b_j$ of the local dynamics of hosts \hsupa{i} and \hsupa{j}, respectively.

Recall that $\mathcal{U}^{(i)} \in [0,M]^n$ is the set of points that are not in the basin of attraction of some stable equilibrium point of the local dynamics of host \hsupa{i}, and let
\begin{equation}
    \mathcal{U} = \bigcup_i \ \mathcal{U}^{(i)} \,.
\end{equation}
Each set $\mathcal{U}^{(i)}$ has measure $0$, so $\mathcal{U}$ also has measure $0$. Let $\bm{x} \in [0,M]^n \setminus \mathcal{U}$. For each $i$, the vector $\bm{x}$ is in the basin of attraction of some $\bm{a}^{(i)} \in \mathcal{A}_i$. Therefore, there exists some time $t_a$ such that $\normtwo{\bm{X}^{(i)}(t_a,\bm{x}) - \bm{a}^{(i)}} = \delta$. We refer to such a time as a \emph{crossing time}. Moreover, because $2\delta \leq \delta \left(\bm{a}^{(i)}\right)$, it follows from \cref{eq:trajder} that $\normtwo{\bm{X}^{(i)}(t,\bm{x}) - \bm{a}^{(i)}}$ is monotonically decreasing in $t$ on some interval $(t_a - \eta, t_a + \eta)$ for all $t \geq t_a$. Therefore, the crossing time $t_a$ is the unique time that satisfies the equality $\normtwo{\bm{X}^{(i)}(t_a,\bm{x}) - \bm{a}^{(i)}} = \delta$. Because the crossing time is unique, we can define a function $\mathcal{T}^{(i)}(\bm{x})$ such that $\normtwo{\bm{X}^{(i)}\left(\mathcal{T}^{(i)}(\bm{x}),\bm{x}\right) - \bm{a}^{(i)}} = \delta$. This function $\mathcal{T}^{(i)}(\bm{x})$ gives the unique crossing time for a flow that starts at $\bm{x}$.

The local-dynamics function $\gi$ is continuous, so
\begin{equation}
    \mathcal{S}(t,\bm{x}) = \normtwo{\bm{X}^{(i)}(t,\bm{x}) - \bm{a}^{(i)}} - \delta
\end{equation}
is continuously differentiable with respect to both $t$ and $\bm{x}$. Evaluating $\mathcal{S}(t,\bm{x})$ at $t = \mathcal{T}^{(i)}(\bm{x})$ yields $\mathcal{S}\left(\mathcal{T}^{(i)}(\bm{x}),\bm{x}\right) = 0$. The function $\mathcal{S}(t,\bm{x})$ is monotonically decreasing in $t$ on some interval $(\mathcal{T}^{(i)}(\bm{x}) - \eta, \mathcal{T}^{(i)}(\bm{x}) + \eta)$. Therefore, by the Implicit Function Theorem, $\mathcal{T}^{(i)}$ is a continuous function on each basin of attraction of {the local dynamics of host \hsupa{i}.} 

Because $\mathcal{T}^{(i)}(\bm{x})$ is continuous and each set $\mathcal{F}(\bm{a}^{(i)},\bm{a}^{(j)})$ is compact, there exist times
\begin{align}
	\tau_{ij} = \max_{\bm{a}^{(i)} \in \mathcal{A}_i, \bm{a}^{(j)} \in \mathcal{A}_j} \left\{\mathcal{T}^{(i)}(\bm{x}) \  \text{ for }\, \bm{x} \in \mathcal{F}\left(\bm{a}^{(i)},\bm{a}^{(j)}\right)\right\} \,, 
\end{align}
with
\begin{align}	
	\tau = \max_{\{i,j \ \mid \ \left(\hsup{i},\hsup{j}\right) \in E \}} \tau_{ij} \,.
\end{align}

Suppose that an interaction occurs between hosts \hsupa{i} and \hsupa{j} at time $t_{I,1}$. Let 
\begin{align}
    \nsup{i}\left(t_{I,1}^-\right) &\in B\left(\bm{a}^{(i)},\delta\right) \,, \\
    \nsup{j}\left(t_{I,1}^-\right) &\in B\left(\bm{a}^{(j)},\delta\right) \nonumber
\end{align}
for some basins of attraction $\bm{a}^{(i)} \in \mathcal{A}_i$ and $\bm{a}^{(j)} \in \mathcal{A}_j$. We then have 
\begin{align}
    \nsup{i}\left(t_{I,1}^+\right) &\in \mathcal{F}\left(\bm{a}^{(i)},\bm{a}^{(j)}\right) \,, \\
    \nsup{j}\left(t_{I,1}^+\right) &\in \mathcal{F}\left(\bm{a}^{(j)},\bm{a}^{(i)}\right) \,. \nonumber
\end{align}
If the next interaction occurs at time $t_{I,2} > t_{I,1} + \tau$, then 
\begin{align}
    \nsup{i}\left(t_{I,2}^-\right) &\in B\left(\bm{b}^{(i)},\delta\right) \,, \\
    \nsup{j}\left(t_{I,2}^-\right) &\in B\left(\bm{b}^{(j)},\delta\right) \nonumber
\end{align}
for some basins of attraction $\bm{b}^{(i)} \in \mathcal{A}_i$ and $\bm{b}^{(j)} \in \mathcal{A}_j$. If no two interactions occur within time $\tau$ of each other, then 
\begin{equation}
    \nsup{i}\left(t_{I}^-\right) \in B\left(\bm{a}^{(i)},\delta\right)
\end{equation}
for all {locales} $i$, all interaction times $t_{I}$, and some basin of attraction $\bm{a}^{(i)} \in \mathcal{A}_i$. The effect of each interaction on the basin probability tensor $\Psi$ is then described exactly by the interaction operator
\begin{equation}
	\left(\phi\left(\Psi\left(t_I^-\right)\right)\right)_{b_1,\ldots,b_{|H|}} = \Phi_{b_1,\ldots,b_{|H|},a_1,\ldots,a_{|H|}} \Psi_{a_1,\ldots, a_{|H|}}\left(t_I^-\right) \,.
\end{equation}

Let $\mathcal{I}$ be the set of all possible sets $\Omega = \{t_l^*\}_{l=1}^L$ of frequency-scaled interaction times in the interval $[0,T^*]$. We select a set $\Omega$ of interaction times using a Poisson process on $[0,T^*]$ with rate parameter $1$. 
Let $q : I \to R^+$ be the probability density function for these interactions. Therefore,
\begin{equation}
    \Pr\left(\Omega \in \mathcal{J}\right) = \int_\mathcal{J} q\left(\Omega'\right) \der\Omega' 
\end{equation}
for any $\mathcal{J} \subseteq \mathcal{I}$. We define a counting function
\begin{equation}
    \nu(\Omega,t^*) = \bigg|\left\{t_l^* \in \Omega \mid t_l^* \leq t^*\right\} \bigg| 
\end{equation}
that tracks the number of interactions that occur in the interval $[0,t^*]$ for the set $\Omega$. The approximate basin probability tensor {$\widetilde{\Psi}(t^*)$ from the LFA 
(see \cref{eq:LFA})} conditioned on $\Omega$ is
\begin{equation}
    \widetilde\Psi^{(\Omega)}(t^*) = \phi^{\nu(\Omega,t^*)}\left(\Psi(0)\right) \,.
\end{equation}
Therefore,
\begin{equation}
    \widetilde\Psi(t^*) = \int_\mathcal{I} q(\Omega') \widetilde{\Psi}^{(\Omega')}(t^*) \, \der\Omega' \,.
\end{equation}

Let $\Psi^{(\Omega)}(t^*)$ be the basin probability tensor conditioned on $\Omega$. We then have
\begin{equation}
    \Psi(t^*) = \int_\mathcal{I} q(\Omega') \Psi^{(\Omega')}(t^*) \, \der\Omega' \,.
\end{equation}
If $\Omega = \{t_l^*\}_{l = 1}^L$ satisfies
\begin{equation} \label{eq:spacing}
    \min_{l\in\{2,\ldots,L\}} \{t_l^* - t_{l - 1}^*\} > \tau^* \,,
\end{equation}
then
\begin{equation}
    \Psi^{(\Omega)}(t^*) = \Phi^{\nu(\Omega,t^*)} \left(\Psi(0)\right) \,.
\end{equation}

Let $\mathcal{I}_1 \subset \mathcal{I}$ be the set of interaction sets $\Omega = \{t_l^*\}_{l = 1}^L$ for which \cref{eq:spacing} holds, and let $\mathcal{I}_2 = \mathcal{I} \setminus \mathcal{I}_1 \subset \mathcal{I}$. Choose $dt^*$ such that $\tau^* < dt^* < 2\tau^*$ and ${T^*} / {dt^*}$ is an integer. For sufficiently small $\tau^*$, this is always possible. If two interaction times $t_l^*$ and $t_{l - 1}^*$ occur within time $\tau^*$ of each other, then both interactions occur in an interval $\left[k \, dt^*,(k - 1) dt^*\right]$ and/or an interval $\left[\left(k + \frac{1}{2}\right) dt^*,\left(k - \frac{1}{2}\right) dt^*\right]$ for an integer $k$. (These two types of intervals overlap.) All of these intervals have width $dt^*$. 
Therefore, the probability that at least two interactions occur in one {specified} interval is
\begin{align}
    \Pr\bigg(\nu\left(\Omega, dt^*\right) \geq 2\bigg) &= 1 - e^{dt^*} - dt^* e^{dt^*} \\ 
    	&= 1 - e^{dt^*} (1 - dt^*) \nonumber \\
    	&< 1 - (1 + dt^*) (1 - dt^*) \nonumber\\
    	&= (dt^*)^2 \, . \nonumber
\end{align}
There are $2 \frac{T^*}{dt^*} - 1$ such intervals. Therefore, the probability that at least two interactions occur in at least one of these intervals is less than $\left(2 \frac{T^*}{dt^*} - 1\right) \left(dt^*\right)^2$. This probability equals the probability that a set $\Omega$ of interactions satisfies $\Omega \in \mathcal{I}_2$. Consequently,
\begin{align}
    \Pr\left(\Omega \in \mathcal{I}_2\right) & < \left(2 \frac{T^*}{dt^*} - 1\right) (dt^*)^2 \\
    	&< 2 \frac{T^*}{dt^*} (dt^*)^2 \nonumber \\ 
	&= 2T^* dt^* \nonumber \\
    	&< 4 T^* \tau^* \nonumber  \\
    	&= 4 T^* \ltot \tau \,. \nonumber
\end{align}
{The difference between the basin probability tensor and the approximate basin probability tensor is}
\begin{align}
    \Psi(t^*) - \widetilde\Psi(t^*) &= \int_\mathcal{I} q(\Omega') \Psi^{(\Omega')}(t^*) \, \der\Omega' - \int_\mathcal{I} q(\Omega') \widetilde{\Psi}^{(\Omega')}(t^*) \, \der\Omega' \\
    	&= \int_\mathcal{I} q(\Omega') \left[\Psi^{(\Omega')}(t^*) - \widetilde\Psi^{(\Omega')}(t^*)\right] \der\Omega' \nonumber \\
    	&= \int_{\mathcal{I}_1} q(\Omega') \left[\Psi^{(\Omega')}(t^*) - \widetilde\Psi^{(\Omega')}(t^*)\right] \der\Omega' + \int_{\mathcal{I}_2} q(\Omega') \left[\Psi^{(\Omega')}(t^*) - \widetilde\Psi^{(\Omega')}(t^*)\right] \der\Omega' \nonumber \\
    	&= \int_{\mathcal{I}_2} q(\Omega') \left[\Psi^{(\Omega')}(t^*) - \widetilde\Psi^{(\Omega')}(t^*)\right] \der\Omega' \,, \nonumber
\end{align}
where the integrand in the integral over $\mathcal{I}_1$ vanishes because no two interactions occur within time $\tau^*$ of each other. Therefore,
\begin{equation}
    \norminf{\Psi(t^*) - \widetilde\Psi(t^*)} \leq \int_{\mathcal{I}_2} q(\Omega') \norminf{\Psi^{(\Omega')}(t^*) - \widetilde\Psi^{(\Omega')}(t^*)} \der\Omega' \,. 
\end{equation}
Each entry of $\Psi(t^*)$ and $\widetilde\Psi(t^*)$ is a probability and hence is in the interval $[0,1]$, so we know that $\norminf{\Psi(t^*) - \widetilde\Psi(t^*)} \leq 1$. 
{We let $\ltot < \frac{\varepsilon}{4T^* \tau}$, and we then have} 
\begin{align}
    \norminf{\Psi(t^*) - \widetilde\Psi(t^*)} &\leq \int_{\mathcal{I}_2} q(\Omega') \der\Omega' \\
    &= \Pr\left(\Omega \in \mathcal{I}_2\right) \nonumber \\
    &< 4 T^* \ltot \tau  \nonumber \\
    &< \varepsilon \,. \nonumber
\end{align}
This bound holds for all frequency-scaled times $t^* \in [0,T^*]$. Because $\varepsilon > 0$ is arbitrary, the basin probability tensor $\Psi(t)$ converges uniformly to $\widetilde{\Psi}(t^*)$ on $[0,T^*]$ as 
the total-interaction-frequency parameter $\lambda_{\mathrm{tot}} \to 0$.
\end{proof}



\section{Proof of the High-Frequency Low-Strength Approximation (HFLSA) Theorem} \label{sec:HFLSA Proof}

In this section, we prove the HFLSA Theorem
(see \cref{thm:HFLSA}).

\begingroup
\begin{theorem}[High-Frequency, Low-Strength Approximation Theorem]
  Fix the relative 
  \newline
  interaction-frequency parameters $l_{ij}$, the product $\ltot \gamma$, and a time $T$. Let each local-dynamics function $g^{(i)}$ be continuously differentiable and bounded (see \Cref{sec:LD}),
  and fix arbitrary $\varepsilon\in(0,1]$ and $\delta > 0$. For sufficiently large $\ltot$, each host's microbiome abundance vector $\nsup{i}(t)$ satisfies
  \begin{equation} 
      \normLinf{\nsup{i} - \nsupapp{i}}{[0,T]} < \delta
  \end{equation}
  with probability larger than $1 - \varepsilon$, where
\begin{align} 
        \frac{\der \nsupapp{i}}{\der t} &= g^{(i)}(\nsupapp{i}) + \sum_j \lambda_{ij} \gamma \left(\nsupapp{j} - \nsupapp{i}\right) \,, \\
        \nsupapp{i}(0) &= \nsup{i}(0) \,. \nonumber
\end{align}
\end{theorem}
\endgroup

\begin{proof}
Each local-dynamics function $g^{(i)}$ is bounded (see \Cref{sec:LD}),
so there exists a constant $M$ such that 
\begin{equation}
 	\norminf{\nsup{i}(t)} \leq M   
\end{equation}
and each entry of $\nsup{i}(t)$ is nonnegative for each microbiome abundance vector $\nsup{i}(t)$ and all times $t \geq 0$. Each approximate microbiome abundance vector $\nsupapp{i}(t)$ also satisfies
\begin{equation}
 \norminf{\nsupapp{i}(t)} \leq M   
\end{equation}
because
\begin{align}
    \left(\sum_j \lambda_{ij} \gamma \left(\nsupapp{j} - \nsupapp{i}\right)\right)_x &\geq 0 \quad \text{if }\,\,  \nsupapp{i}_x = 0 \,, \\
    \left(\sum_j \lambda_{ij} \gamma \left(\nsupapp{j} - \nsupapp{i}\right)\right)_x &\leq 0 \quad \text{if } \,\, \nsupapp{i}_x = M  \notag
\end{align}
for each entry $x$ of $\nsupapp{i}(t)$.

We also assume that each local-dynamics function $g^{(i)}$ is continuously differentiable and hence continuous. Each host's microbiome abundance vector $\nsup{i}(t)$ is in the region $[0,M]^{n}$, which is compact, so there exist constants $G$ and $F$ such that all $\nsup{i}(t)$ satisfy the bounds
\begin{align} \label{eq:LD Bound 1}
        \norminf{\frac{\der \nsup{i}}{\der t}} &= \norminf{g^{(i)}\left(\nsup{i}\right)} \leq G \,, \\
        \norminf{\frac{\der^2 \nsup{i}}{\der t ^2}} &= \norminf{Dg^{(i)}\left(\nsup{i}\right) \cdot g^{(i)}\left(\nsup{i}\right)}  \leq 2F \,.  \nonumber
\end{align}
For an interaction involving host \hsupa{i} that occurs at time $t_I$, we choose $\nsup{i}\left(t_I\right) = \nsup{i}\left(t_I^+\right)$. Therefore, $\nsup{i}(t)$ is right-continuous at time $t_I$. It is usually not left-continuous at time $t_I$, so it is usually not left-differentiable at time $t_I$.\footnote{The only situation where $\nsup{i}(t)$ is left-continuous at time $t_I$ occurs when the host \hsupa{j} with which \hsupa{i} interacts has a microbiome vector $\nsup{j}(t_I^-) = \nsup{i}(t_I^-)$.} In such situations, the derivatives that we use in \cref{eq:LD Bound 1} are right derivatives. By Taylor's theorem,
\begin{align}
        \norminf{\nsup{i}(t + dt) - \nsup{i}(t)} &\leq G \, dt \,, \\
         \norminf{\nsup{i}(t + dt) - \nsup{i}(t) - \gi \left(N^i(t)\right) dt} &\leq F \, dt^2 \notag
\end{align}
for any interval $(t,t + dt]$ in which there are no interactions. Each local-dynamics function $\gi$ is also Lipschitz continuous. Therefore, there exists a constant $C$ such that 
\begin{equation}
    \norminf{g^{(i)}\left(\bm{x}\right) - g^{(i)}\left(\bm{y}\right)} \leq C \norminf{\bm{x} - \bm{y}}
\end{equation}
for each local-dynamics function $\gi$ and all $\bm{x},\bm{y} \in [0,M]^{n}$.

Let $\widetilde{G} = G + M \ltot \gamma$, which is a constant because $\ltot \gamma$ is fixed. For all host microbiome abundance vectors $\nsupapp{i}(t)$, we have
\begin{align}
    \norminf{\frac{\der \nsupapp{i}}{\der t}} &= \norminf{g^{(i)}\left(\nsupapp{i}\right) + \sum_j \lambda_{ij} \gamma \left(\nsupapp{j} - \nsupapp{i}\right)} \\
    	&\leq G + \sum_j \lambda_{ij} \gamma \norminf{\nsupapp{j} - \nsupapp{i}} \notag \\
	&\leq G + M \ltot \gamma \notag \\ 
	&= \widetilde G \,. \nonumber
\end{align}
Let $\widetilde{F} = F + G \ltot \gamma + M \ltot^2 \gamma^2$. For all $\nsupapp{i}(t)$, we have
\begin{align}
         \norminf{\frac{\der^2 \nsupapp{i}}{\der t ^2}} 
         &= \norminf{
         	\begin{aligned}
         	   &Dg^{(i)}\left(\nsupapp{i}\right) \cdot g^{(i)}\left(\nsupapp{i}\right) \\
         	   &\qquad + \sum_j \lambda_{ij} \gamma \left[ g^{(j)}\left(\nsupapp{j}\right) + \sum_k \lambda_{jk} \gamma \left(\nsupapp{k}-\nsupapp{j}\right) \right] \\
         	   &\qquad + \sum_j \lambda_{ij} \gamma \left[ g^{(i)}\left(\nsupapp{i}\right) + \sum_k \lambda_{ik} \gamma \left(\nsupapp{k}-\nsupapp{i}\right) \right]
        		\end{aligned}} \\
        &\leq 2F + G \gamma \sum_j \lambda_{ij} + M \gamma^2 \sum_j \sum_k \lambda_{ij} \lambda_{jk} + G \gamma \sum_j \lambda_{ij} + M \gamma^2 \sum_j \sum_k \lambda_{ij} \lambda_{ik} \nonumber\\
        &\leq 2F + 2G \ltot \gamma + 2M \ltot^2 \gamma^2 = 2\widetilde F \,. \nonumber
\end{align}
 By Taylor's theorem, 
\begin{align}
        \norminf{\ \nsupapp{i}(t + dt) - \nsupapp{i}(t) \ } &\leq \widetilde G \, dt \,, \\
        \norminf{ \ \nsupapp{i}(t + dt) - \nsupapp{i}(t) - \left[\gi \left(\nsupapp{i}(t)\right) + \sum_j \lambda_{ij} \gamma \left(\nsupapp{j}(t) - \nsupapp{i}(t)\right) \right] dt } &\leq \widetilde F \, dt^2 \notag
\end{align}
for any times $t$ and $t + dt$.

We phrased 
\cref{thm:HFLSA} in terms of finding a sufficiently large total-interaction-frequency parameter $\ltot$ so that
\begin{equation}
      \normLinf{\nsup{i} - \nsupapp{i}}{[0,T]} < \delta
  \end{equation}
with probability larger than $1 - \varepsilon$. The product $\ltot \gamma$ is fixed, so finding a sufficiently large $\ltot$ is equivalent to finding a sufficiently small $\gamma$. We define the error term
  \begin{equation}
    \Ei(t) = \nsup{i}(t) - \nsupapp{i}(t) \, .
\end{equation}
We will show that one can bound each absolute error $\norminf{E^{(i)}(t)}$ with probability larger than $1 - \varepsilon$ by a term that involves $\gamma$. For sufficiently small $\gamma$, we will show that
\begin{equation}
    \norminf{\Ei(t)} \leq \delta 
\end{equation}
for all times $t \in [0,T]$ with probability larger than $1 - \varepsilon$.

Fix a $dt$ such that $dt^4 \leq \gamma \leq 4 \, dt^4 < 1$ and {${T}/{dt}$ is an integer}. 
This is always possible for sufficiently small $\gamma$. Let $t_k = k \, dt$. There are only finitely many times $t_k$ in the interval $[0,T]$. The probability that an interaction occurs precisely at any of these times is $0$. Therefore, for the remainder of this proof, we only consider interactions that occur at times $t \neq t_k$ for any $k$. Under this assumption,
\begin{equation}
    \nsup{i}(t_k^-) = \nsup{i}(t_k)
\end{equation}
for all $t_k \in [0,T]$.

We now consider how a microbiome abundance vector $\nsup{i}(t)$ changes during a time interval $[t_k,t_{k + 1}]$. Let $L_k$ be the number of interactions that occur in $(t_k,t_{k + 1})$. 
(This interval is open because no interactions occur at any $t_k$.) We denote the associated ordered set of interactions by $\{t_{k,l}\}_{l = 1}^{L_k}$. Additionally, we let $t_{k,0} = t_k$ and $t_{k,L_k + 1} = t_{k + 1}$, 
and we define $dt_{k,l} = t_{k,l} - t_{k,l - 1}$. For $l \in \{1,\ldots,L_k + 1\}$, let
\begin{equation}
    \Ai_{k,l} = \nsup{i}\left(t_{k,l}^-\right) - \nsup{i}(t_{k}) \,.
\end{equation}
 For $l \in \{1,\ldots,L_k\}$, let
\begin{equation}
    \Ji_{k,l} = \nsup{i}\left(t_{k,l}^+\right) - \nsup{i}\left(t_{k,l}^-\right) \,.
\end{equation}
The difference $\Ji_{k,l}$ is the change in $\nsup{i}(t)$ after an interaction at time $t_{k,l}$. If this interaction does not involve host \hsupa{i}, then $\Ji_{k,l} = \bm{0}$. Otherwise, for an interaction between hosts \hsupa{i} 
and \hsupa{j}, we have
\begin{equation}
    \Ji_{k,l} = \gamma \left(\nsup{j}\left(t_{k,l}^-\right) - \nsup{i}\left(t_{k,l}^-\right)\right) \, .
\end{equation}
In either case, $\norminf{\Ji_{k,l}} \leq M \gamma$. For $l \geq 2$, we decompose the difference $\Ai_{k,l}$ by writing
\begin{align}
        \Ai_{k,l} &= \nsup{i}\left(t_{k,l}^-\right) - \nsup{i}(t_{k,l - 1}^+) + \nsup{i}(t_{k,l - 1}^+) - \nsup{i}(t_{k,l - 1}^-) \\
        &\qquad + \nsup{i}(t_{k,l - 1}^-) - \nsup{i}(t_{k}) \nonumber \\
        &= \nsup{i}\left(t_{k,l}^-\right) - \nsup{i}(t_{k,l - 1}^+) + \Ji_{k,l - 1} + \Ai_{k,l - 1}  \,. \nonumber
    \end{align}
Therefore,
\begin{align}
    \norminf{\Ai_{k,l}} & \leq \norminf{\nsup{i}\left(t_{k,l}^-\right) - \nsup{i}(t_{k,l - 1}^+)} + \norminf{\Ji_{k,l - 1}} + \norminf{\Ai_{k,l - 1}}  \\
        &  \leq \norminf{\Ai_{k,l - 1}} + G \, dt_{k,l} + M \gamma \,. \nonumber
\end{align}
For $l = 1$, we have
\begin{equation}
    \norminf{\Ai_{k,1}} \leq \norminf{\nsup{i}(t_{k,1}^-) - \nsup{i}(t_{k})} \leq  G \, dt_{k,1}\,.
\end{equation}
Therefore,
\begin{equation}
    \norminf{\Ai_{k,l}} \leq \sum_{l'=1}^l G \, dt_{k,l'} + (l - 1) M \gamma \leq \sum_{l'=1}^{L_k+1} G \, dt_{k,l'} + L_k M \gamma = G \, dt + L_k M \gamma \,.
\end{equation}

The error between the actual microbiome abundance vector $\nsup{i}(t_k)$ and the approximate microbiome abundance vector $\nsupapp{i}(t_k)$ is
\begin{equation}
        \Ei_k = \nsup{i}(t_k) - \nsupapp{i}(t_k) \,.
\end{equation}
Consider the difference
\begin{align} \label{eq:error diff}
    \Ei_{k + 1} - \Ei_{k}  &= \left[\nsup{i}(t_{k + 1}) - \nsupapp{i}(t_{k + 1})\right] - \left[\nsup{i}(t_k) - \nsupapp{i}(t_k) \right] \\
    		&= \left[\nsup{i}(t_{k + 1}) - \nsup{i}(t_{k})\right] - \left[\nsupapp{i}(t_{k + 1}) - \nsupapp{i}(t_{k})\right] \,. \nonumber
\end{align}
We thus have 
\smallskip
\begin{equation} \label{eq:insert1}
    \nsupapp{i}(t_{k + 1}) - \nsupapp{i}(t_{k}) = \left[\gi \left(\nsupapp{i}(t_k)\right) + \sum_j \lambda_{ij} \gamma \left(\nsupapp{j}(t_k) - \nsupapp{i}(t_{k})\right) \right] dt + \etai{\,approx}{k} \,,
\end{equation}
where
\begin{equation}
    \norminf{\etai{\,approx}{k}} \leq \widetilde F \, dt^2 \, .
\end{equation}

The change of the microbiome abundance vector $\nsup{i}(t)$ during the interval $[t_k,t_{k + 1}]$ is
\begin{align} \label{eq:insert2}
        \nsup{i}(t_{k + 1}) - \nsup{i}(t_{k}) &= \sum_{l = 1}^{L_k + 1} \left[ \nsup{i}\left(t_{k,l}^-\right) - \nsup{i}\left(t_{k,l - 1}^-\right) \right]  \\
        		&= \sum_{l = 1}^{L_k + 1} \!\left[ \nsup{i}\left(t_{k,l}^-\right) - \nsup{i}\left(t_{k,l - 1}^+\right)\right] + \!\sum_{l = 2}^{L_k + 1} \!\left[ \nsup{i}\left(t_{k,l - 1}^+\right) - \nsup{i}\left(t_{k,l - 1}^-\right) \right]  \nonumber \\
        		&= \sum_{l = 1}^{L_k + 1} \left[ \gi \left(\nsup{i}\left(t_k\right) \right) \, dt_{k,l} + \etai{\,local}{k,l} \right] + \sum_{l = 1}^{L_k} \Ji_{k,l}  \nonumber\\
        		&= \gi\left(\nsup{i}\left(t_k\right)\right) dt + \sum_{l = 1}^{L_k + 1} \left[\etai{\,local}{k,l}\right] + \sum_{l = 1}^{L_k} \Ji_{k,l} \,, \nonumber
 \end{align}
where
\begin{align}
    \etai{\,local}{k,l} &= \left[\nsup{i}\left(t_{k,l}^-\right) - \nsup{i}\left(t_{k,l-1}^+\right) - \gi\left(\nsup{i}\left(t_{k,l}^-\right) \right) \, dt_{k,l} \right] \\
        &\qquad + \left[ \gi\left(\nsup{i}\left(t_{k,l}^-\right) \right) \, dt_{k,l} - \gi\left(\nsup{i}\left(t_k\right) \right) \, dt_{k,l} \right] \nonumber \,.
\end{align}
Therefore,
\begin{align} \label{eq:etalockl}
    \norminf{\etai{\,local}{k,l}} & \leq F \, dt_{k,l}^2 + C \norminf{\nsup{i}\left(t_{k,l}^-\right) - \nsup{i}\left(t_{k}\right)} \, dt_{k,l}  \\ 
    &= F \, dt_{k,l}^2 + C \norminf{\Ai_{k,l}} \, dt_{k,l} \nonumber \\
    & \leq F \, dt_{k,l}^2 + C G \, dt \, dt_{k,l} + C L_k M \gamma \, dt_{k,l} \nonumber \\
    & \leq \left(F \, dt + C G \, dt + C L_k M \gamma \right) dt_{k,l} \,. \nonumber
\end{align}

We now consider 
\begin{equation}
    \sum_{l=1}^{L_k} \Ji_{k,l} \,,
\end{equation}
which is the sum of the changes in the microbiome abundance vector $\nsup{i}(t)$ due to the interactions of host \hsupa{i} with other hosts. Let
\medskip
\begin{equation}
        \Jiapp_{k,l} = \begin{cases}
            \bm{0}\,, & \!\!\!\text{the interaction at time } t_{k,l} \text{ does not involve \hsupa{i}} \\
            \gamma\left(\nsup{j}(t_{k}) - \nsup{i}(t_{k})\right)\,, & \!\!\!\text{the interaction at time } t_{k,l} \text{ is between \hsupa{i} and \hsupa{j}} \, .
        \end{cases}
\end{equation}
We have
\begin{equation} \label{eq:sum ex}
    \sum_{l = 1}^{L_k} \Ji_{k,l} = \sum_{l=1}^{L_k} \left(\Jiapp_{k,l} + \etai{\,exchange}{k,l}\right) \,,
\end{equation}
where
\begin{equation}
    \etai{\,exchange}{k,l} = \Ji_{k,l} - \Jiapp_{k,l} \,.
\end{equation}
If the interaction at time $t_{k,l}$ does not involve host \hsupa{i}, then $\etai{\,exchange}{k,l} = \bm{0}$. 
For an interaction at time $t_{k,l}$ between hosts \hsupa{i} and \hsupa{j}, we have
\begin{align}
        \etai{\,exchange}{k,l} &= \gamma \left(\nsup{j}(t_{k})-\nsup{i}(t_{k})\right) - \gamma\left(\nsup{j}\left(t_{k,l}^-\right) -\nsup{i}\left(t_{k,l}^-\right)\right)  \\
        &= \gamma \left[\left(\nsup{j}(t_{k}) - \nsup{j}\left(t_{k,l}^-\right)\right) - \left(\nsup{i}(t_{k}) - \nsup{i}\left(t_{k,l}^-\right)\right)\right] \,. \nonumber
    \end{align}
In either case,
\begin{align} \label{eq:etaexchkl}
        \norminf{\etai{\,exchange}{k,l}} &\leq  \gamma \left(\norminf{\nsup{j}(t_{k}) - \nsup{j}\left(t_{k,l}^-\right)} + \norminf{\nsup{i}(t_{k}) - \nsup{i}\left(t_{k,l}^-\right)}\right) \\
        &= \gamma \left(\norminf{A^{(j)}_{k,l}} + \norminf{\Ai_{k,l}}\right) \nonumber \\
        &\leq 2 G \gamma \, dt + 2L_k M \gamma^2 \,. \nonumber
    \end{align}

We seek to bound
\begin{equation} \label{eq:sumnorm}
    \norminf{\sum_{l = 1}^{L_k} \Jiapp_{k,l} - \E\left[\sum_{l = 1}^{L_k} \Jiapp_{k,l}\right]} \,.
\end{equation}
There are two sources of stochasticity in the sum
\begin{equation} \label{eq:sumjapp}
    \sum_{l = 1}^{L_k} \Jiapp_{k,l} \,.
\end{equation}
The number $L_k$ of interactions follows a Poisson distribution with mean $\ltot dt$, and the term $\Jiapp_{k,l}$ depends on which pair of hosts interacts at time $t_{k,l}$. The sum \cref{eq:sumnorm} is stochastic, so we are only able to bound it with high probability.

The quantity $\Jiapp_{k,l}$ is vector-valued, and we denote an entry $x$ of it by $\left(\Jiapp_{k,l}\right)_x$. The sum \cref{eq:sumjapp} is also vector-valued, and we denote an entry $x$ of it by $\left(\sum_{l = 1}^{L} \Jiapp_{k,l}\right)_x$. We now calculate the expectation and the variance of each entry of \cref{eq:sumjapp}. The interaction at time $t_{k,l}$ is between hosts \hsupa{i} and \hsupa{j} with probability ${\lambda_{ij}}/{\ltot}$. Therefore,
\begin{align}
    \E \left[L_k\right] &= \ltot  dt \,, \\
    \Var \left[L_k\right] &= \ltot  dt \,, \nonumber \\
    \E \left[ \left(\Jiapp_{k,l}\right)_x \right] &= \sum_j \frac{\lambda_{ij}}{\ltot}\gamma\left(\nsup{j}(t_{k}) - \nsup{i}(t_{k})\right)_x \,, \nonumber \\
    \Var \left[ \left(\Jiapp_{k,l}\right)_x \right] &= \E \left[ \left(\Jiapp_{k,l}\right)_x^2\right] - \left(\E \left[ \left(\Jiapp_{k,l}\right)_x \right]\right)^2 \,. \nonumber
\end{align}

The expectation of each entry of the sum \cref{eq:sumjapp} is
\begin{align}
    \E \left[\left(\sum_{l = 1}^{L_k} \Jiapp_{k,l}\right)_x\right] &= \E\left[\;\E \left[\left(\sum_{l = 1}^{L_k} \Jiapp_{k,l}\right)_x \;\middle|\; L_k\right]\;\right]  \\
    	&= \E \left[L_k  \sum_j \frac{\lambda_{ij}}{\ltot} \gamma \left(\nsup{j}(t_{k}) - \nsup{i}(t_{k})\right)_x \right] \nonumber\\
    	&= \ltot  dt  \sum_j \frac{\lambda_{ij}}{\ltot} \gamma \left(\nsup{j}(t_{k}) - \nsup{i}(t_{k})\right)_x \nonumber\\
    	&= \sum_j \lambda_{ij} \gamma \left(\nsup{j}(t_{k}) - \nsup{i}(t_{k})\right)_x dt \,. \nonumber
\end{align}
Therefore,
\begin{equation}
    \E \left[\sum_{l=1}^{L_k} \Jiapp_{k,l}\right] = \sum_j \lambda_{ij} \gamma \left(\nsup{j}(t_{k}) - \nsup{i}(t_{k})\right)dt \,.
\end{equation}

Applying the law of total variance, the variance of each entry of the sum \cref{eq:sumjapp} is
\begin{align}
        \Var\left[\left(\sum_{l = 1}^{L_k} \Jiapp_{k,l}\right)_x\right] &= \E\left[\; \Var \left[\left(\sum_{l=1}^{L_k} \Jiapp_{k,l}\right)_x\;\middle|\; L_k \right]\; \right] + \Var\left[ \; \E \left[\left(\sum_{l = 1}^{L_k} \Jiapp_{k,l}\right)_x\;\middle|\; L_k \right]\; \right] \\
        		&= \E\left[L_k \Var \left[\left(\Jiapp_{k,l}\right)_x \right]\; \right] + \Var\left[ \; L_k\E \left[\left(\Jiapp_{k,l}\right)_x \right]\; \right] \nonumber\\
        		&= \ltot dt \,\Var \left[\left(\Jiapp_{k,l}\right)_x \right] + \ltot  dt \left(\E \left[\left(\Jiapp_{k,l}\right)_x \right]\right)^2 \nonumber\\
        		&=  \ltot dt \,\E \left[\left(\Jiapp_{k,l}\right)_x^2\right] \nonumber\\
        		&=  \ltot  dt \sum_j \frac{\lambda_{ij}}{\ltot}\gamma^2\left(\nsup{j}(t_{k})-\nsup{i}(t_{k})\right)_x^2  \nonumber\\
        		&=   \sum_j \lambda_{ij}\gamma^2\left(\nsup{j}(t_{k}) - \nsup{i}(t_{k})\right)_x^2 dt \,. \nonumber
    \end{align}
These variances satisfy the bound
    \begin{align}
             \Var\left[\sum_{l = 1}^{L_k} \left(\Jiapp_{k,l}\right)_x\right] &\leq  \sum_j \lambda_{ij} \gamma^2 \left(\nsup{j}(t_{k}) - \nsup{i}(t_{k})\right)_x^2 dt  \\
            	&\leq  M^2 \gamma^2 \, dt  \sum_j \lambda_{ij} \leq M^2 \ltot \gamma^2 \, dt \,. \nonumber
        \end{align}

The error due to the stochasticity of the interactions is
\begin{equation} \label{eq:etakexch}
    \etai{\,exchange}{k} = \sum_{l = 1}^{L_k} \Jiapp_{k,l} - \E\left[\sum_{l = 1}^{L_k} \Jiapp_{k,l}\right] \, ,
\end{equation}
so
\smallskip
\begin{equation} \label{eq:sumjeta}
    \sum_{l = 1}^{L_k} \Jiapp_{k,l} = \etai{\,exchange}{k} + \E\left[\sum_{l = 1}^{L_k} \Jiapp_{k,l}\right] = \etai{\,exchange}{k} + \sum_j \lambda_{ij} \gamma \left(\nsup{j}(t_{k}) - \nsup{i}(t_{k})\right)dt \,. 
\end{equation}
The error \cref{eq:etakexch} satisfies
\begin{equation}
    \norminf{\etai{\,exchange}{k}} = \max_x \left\{\;\normod{\sum_{l=1}^{L_k} \left(\Jiapp_{k,l}\right)_x - \E\left[\sum_{l=1}^{L_k} \left(\Jiapp_{k,l}\right)_x\right]} \; \right\} \,.
\end{equation}
Let $\kappa = \dim\left(\nsup{i}\right)$ and $\alpha = \sqrt{\frac{3 \kappa T}{\varepsilon \, dt}}$. By Chebyshev's inequality,
\begin{align}
    \Pr\left(\;\normod{\sum_{l = 1}^{L_k} \left(\Jiapp_{k,l}\right)_x - \E\left[\sum_{l = 1}^{L_k} \left(\Jiapp_{k,l}\right)_x\right]} \geq \alpha \sqrt{\Var\left[\sum_{l = 1}^{L_k} \left(\Jiapp_{k,l}\right)_x\right]} \;\right) &\leq \frac{1}{\alpha^2} \,, \\
    \Pr\left(\;\normod{\sum_{l = 1}^{L_k} \left(\Jiapp_{k,l}\right)_x - \E\left[\sum_{l = 1}^{L_k} \left(\Jiapp_{k,l}\right)_x\right]} \geq \alpha M \gamma \sqrt{\ltot dt} \; \right) &\leq \frac{1}{\alpha^2} \, , \nonumber\\
    \Pr\left(\;\normod{\sum_{l = 1}^{L_k} \left(\Jiapp_{k,l}\right)_x - \E\left[\sum_{l = 1}^{L_k} \left(\Jiapp_{k,l}\right)_x\right]} \geq M \gamma \sqrt{\frac{3\kappa T \ltot}{\varepsilon}} \; \right) &\leq \frac{\varepsilon \, dt}{3 \kappa T} \, . \nonumber
\end{align}
Therefore, 
\begin{equation} \label{eq:err exchange}
    \Pr\left(\;\norminf{\etai{\,exchange}{k}} <  M \gamma \sqrt{\frac{3 \kappa T \ltot}{\varepsilon}} \; \right) > 1 - \frac{\varepsilon \, dt}{3 T} \, .
\end{equation}

Inserting \cref{eq:sumjeta} into \cref{eq:sum ex} yields
\begin{align} \label{eq:insert3}
        \sum_{l = 1}^{L_k} \Ji_{k,l} &= \sum_{l = 1}^{L_k} \left(\Jiapp_{k,l} + \etai{\,exchange}{k,l}\right)  \\
        			&= \sum_j \lambda_{ij} \gamma \left(\nsup{j}(t_{k}) - \nsup{i}(t_{k})\right) dt + \sum_{l = 1}^{L_k} \etai{\,exchange}{k,l} + \etai{\,exchange}{k} \,. \nonumber
\end{align}
Inserting \cref{eq:insert1}, \cref{eq:insert2}, and \cref{eq:insert3} into \cref{eq:error diff} yields
\begin{align}
    \Ei_{k + 1} - \Ei_{k} &= \left[\nsup{i}(t_{k+1}) - \nsup{i}(t_{k})\right] - \left[\nsupapp{i}(t_{k + 1}) - \nsupapp{i}(t_{k})\right] \\
    &\begin{aligned}
         &=  \gi\left(\nsup{i}\left(t_k\right) \right) dt + \sum_{l = 1}^{L_k+1} \etai{\,local}{k,l} \\
        		&\qquad + \sum_j \lambda_{ij} \gamma \left(\nsup{j}(t_{k}) - \nsup{i}(t_{k})\right) dt + \sum_{l = 1}^{L_k} \etai{\,exchange}{k,l} + \etai{\,exchange}{k} \\ 
        		&\qquad - \left[\gi\left(\nsupapp{i}(t_k)\right) + \sum_j \lambda_{ij} \gamma \left(\nsupapp{j}(t_k) - \nsupapp{i}(t_k)\right)\right] dt - \etai{\,approx}{k} 
    \end{aligned} \nonumber \\
    &\begin{aligned}
        &=  \gi\left(\nsup{i}\left(t_k\right)\right) dt - \gi\left(\nsupapp{i}(t_k)\right) dt \\
        		&\qquad + \sum_j \lambda_{ij} \gamma \left(\nsup{j}(t_{k}) - \nsup{i}(t_{k})\right) dt - \sum_j \lambda_{ij} \gamma \left(\nsupapp{j}(t_k) - \nsupapp{i}(t_k)\right)dt \\
        		&\qquad + \sum_{l = 1}^{L_k+1} \etai{\,local}{k,l} + \sum_{l=1}^{L_k} \etai{\,exchange}{k,l} + \etai{\,exchange}{k} - \etai{\,approx}{k} \,.
    \end{aligned} \nonumber
\end{align}
Therefore,
\begin{align}
        \norminf{\Ei_{k + 1}} \leq & \norminf{\Ei_{k}} + \norminf{\gi\left(\nsup{i}\left(t_k\right)\right) - \gi\left(\nsupapp{i}(t_k)\right)}  dt \\
        &+ \sum_j \lambda_{ij} \gamma \left(\norminf{\nsup{j}(t_{k}) - \nsupapp{j}(t_k)} + \norminf{\nsup{i}(t_{k}) - \nsupapp{i}(t_k)}\right) dt \nonumber\\
        &+ \sum_{l = 1}^{L_k + 1} \norminf{\etai{\,local}{k,l}} + \sum_{l = 1}^{L_k} \norminf{\etai{\,exchange}{k,l}} + \norminf{\etai{\,exchange}{k}} + \norminf{\etai{\,approx}{k}} \nonumber \,.
    \end{align}

The inequality \cref{eq:err exchange} gives a bound for $\norminf{\etai{\,exchange}{k}}$ that holds with probability larger than $1 - \frac{\varepsilon \, dt}{3 T}$. The inequalities \cref{eq:etalockl} and \cref{eq:etaexchkl} give bounds for the error terms $\norminf{\etai{\,local}{k,l}}$ and $\norminf{\etai{\,exchange}{k,l}}$, respectively. The latter two bounds depend on the number $L_k$ of interactions in the interval $(t_k, t_{k + 1})$. 
As we discussed previously, $L_k$ is stochastic and is distributed as a Poisson random variable with mean $\ltot dt$. Therefore, we give bounds for $L_k$ that hold with high probability. To obtain these bounds, we first define
\begin{equation}
    \beta = \sqrt{\frac{3 T}{\varepsilon \, dt}} \,.
\end{equation}
By Chebyshev's inequality,
\begin{align}
    \Pr\left(L_k \geq \ltot dt + \beta \sqrt{\ltot dt} \right) &\leq \frac{1}{\beta^2} \,, \\
    \Pr\left(L_k \geq \ltot  dt +  \sqrt{\frac{3 T \ltot}{\varepsilon} } \right) &\leq \frac{\varepsilon \, dt}{3 T} \,. \nonumber
\end{align}
Therefore, with probability greater than $1 - \frac{\varepsilon \, dt}{3 T}$, we have the bounds
\begin{equation}
    L_k < \ltot dt +  \sqrt{\frac{3 T \ltot}{\varepsilon} } 
\end{equation}
and
\begin{align} \label{eq:bp 2}
    L_k \gamma &< \ltot \gamma \, dt +  \sqrt{\gamma} \sqrt{\frac{3 T \ltot \gamma}{\varepsilon}} \\
    	&< \ltot \gamma \, dt +  2 dt^2 \sqrt{\frac{3 T \ltot \gamma}{\varepsilon}} \nonumber\\
    	&< \left(\ltot \gamma +  2\sqrt{\frac{3 T \ltot \gamma}{\varepsilon}}\right) \, dt \,, \nonumber
\end{align}
where we use the inequalities $\gamma < 4 \,dt^4$ and $dt < 1$. With probability {greater than} $1 - \left(\frac{\varepsilon \, dt}{3T} + \frac{\varepsilon \, dt}{3T}\right) = 1 - \frac{2\varepsilon \, dt}{3T}$, the bounds in \cref{eq:err exchange} and \cref{eq:bp 2} both hold, yielding
\begin{align} \label{eq:ek1bound}
    \norminf{\Ei_{k + 1}}&
         \leq \norminf{\Ei_{k}} + C \norminf{\Ei_{k}} dt + \left(\norminf{E^{(j)}_{k}} + \norminf{\Ei_{k}}\right)\ltot \gamma \, dt \\
        		&\qquad + \sum_{l = 1}^{L_k+1}  \left(F \, dt + C G \, dt + C L_k M \gamma \right) dt_{k,l} \nonumber\\
        		&\qquad + \sum_{l = 1}^{L_k}\left(2 G \gamma \, dt + 2L_k M \gamma^2\right) + M \gamma \sqrt{\frac{3\kappa T \ltot}{\varepsilon}} + \widetilde F \, dt^2 \nonumber\\
    &
         \leq \left(1 + (C + \ltot \gamma) \, dt\right) \norminf{\Ei_{k}}  + \ltot \gamma \, dt \norminf{E^{(j)}_{k}} \nonumber\\
        		&\qquad + (F + CG) dt^2 + C L_k M \gamma \, dt + 2 G L_k \gamma \, dt \nonumber\\
        		&\qquad + 2 L_k^2 M \gamma^2 + M \sqrt{\gamma} \sqrt{\frac{3\kappa T \ltot \gamma}{\varepsilon}} + \widetilde F \, dt^2
    \nonumber \\
    &
        \leq  \left(1 + (C + \ltot \gamma) \, dt\right) \norminf{\Ei_{k}}  + \ltot \gamma \, dt \norminf{E^{(j)}_{k}} \nonumber\\
        		&\qquad + (F + CG) dt^2 + C M \left(\ltot \gamma +  2\sqrt{\frac{3 T \ltot \gamma}{\varepsilon}}\right) \, dt^2  \nonumber \\
        		&\qquad  + 2 G \left(\ltot \gamma +  2\sqrt{\frac{3 T \ltot \gamma}{\varepsilon}}\right) \, dt^2 + 2 M \left(\ltot \gamma +  2\sqrt{\frac{3 T \ltot \gamma}{\varepsilon}}\right)^2 \, dt^2 \nonumber \nonumber \\
        		&\qquad  + 2M \sqrt{\frac{3\kappa T \ltot \gamma}{\varepsilon}} \, dt^2 + \widetilde F \, dt^2 \,. \nonumber 
\end{align}
Grouping all of the prefactors of $dt^2$ into a single constant $Z$, we simplify \cref{eq:ek1bound} and write
\begin{equation}
    \norminf{\Ei_{k + 1}}  \leq   \left(1 + (C + \ltot \gamma) \, dt\right) \norminf{\Ei_{k}}  + \ltot \gamma \, dt \norminf{E^{(j)}_{k}} + Z \, dt^2 \, .
\end{equation}

Let
\begin{equation}
    E_k = \max_i \left\{\norminf{\Ei_k}\right\} 
\end{equation}
denote the maximum error at time $t_k$. The maximum error at time $0$ is $E_0 = 0$. Let {$Y = C + 2\ltot \gamma$}. 
The maximum error at time $t_{k + 1}$ satisfies the bound
\begin{equation} \label{eq:Ek bound}
    E_{k + 1}  \leq  (1 + Y  dt) E_{k} + Z \, dt^2 \, .
\end{equation}
With probability {greater than} $1 - \frac{T}{dt} \left(\frac{2\varepsilon \, dt}{3 T}\right) = 1 - \frac{2}{3\varepsilon} > 1 - \varepsilon$, the inequality \cref{eq:Ek bound} holds for all $k \in \{1,\ldots,\frac{T}{dt} - 1\}$. 
Therefore, with probability greater than $1-\varepsilon$, the maximum error at time $t_k$ satisfies
\begin{align}
    \norminf{E_{k}} &\leq Z \, dt^2 \sum_{k'=0}^{k - 1} (1 + Y  dt)^{k'} \\
    &= Z \, dt^2 \left(\frac{(1 + Y  dt)^k - 1}{(1 + Y  dt) - 1}\right) \nonumber\\
    &\leq \frac{Z}{Y} \, dt \left(e^{Y  dt}\right)^k \nonumber \\
     &\leq \frac{Z}{Y} \, dt \left(e^{Y  dt}\right)^\frac{T}{dt}  \nonumber\\
    &\leq \frac{Ze^{Y T}}{Y} \, dt \nonumber
\end{align}
for all $k \in \{0,\ldots,\frac{T}{dt}\}$

Consider an arbitrary time $t' \in (t_k,t_{k+1})$. For some $l$, we have $t' \in [t_{k,l},t_{k,l + 1})$. It follows that
\begin{align}
     \Ei (t') &= \nsup{i}(t') - \nsupapp{i}(t') \\
           &=  \nsup{i}(t') - \nsup{i}\left(t_{k,l}^+\right) + \nsup{i}\left(t_{k,l}^+\right) - \nsup{i}\left(t_{k,l}^-\right) \notag  \\
           &\qquad + \nsup{i}\left(t_{k,l}^-\right) - \nsup{i}(t_k) + \nsup{i}(t_k) - \nsupapp{i}(t_k) \notag \\ 
           &\qquad + \nsupapp{i}(t_k) - \nsupapp{i}(t')
       \,. \nonumber
       \end{align}
Therefore, 
\begin{align}
       \norminf{\Ei (t')} &
            \leq \norminf{\nsup{i}(t') - \nsup{i}\left(t_{k,l}^+\right)} + \norminf{\nsup{i}\left(t_{k,l}^+\right) - \nsup{i}\left(t_{k,l}^-\right)} \\
           &\qquad + \norminf{\nsup{i}\left(t_{k,l}^-\right) - \nsup{i}(t_k)} + \norminf{\nsup{i}(t_k) - \nsupapp{i}(t_k)} \nonumber \\
           &\qquad + \norminf{\nsupapp{i}(t_k) - \nsupapp{i}(t')}
      \nonumber\\
       &\leq \ G \left(t' - t_{k,l}\right) + M \gamma  + \norminf{\Ai_{k,l}} + \norminf{\Ei_k} + \widetilde G (t' - t_k)
        \nonumber\\
       &\leq \ G \, dt + 4M \, dt^2  + (G \, dt + L_k M \gamma) + \frac{Z e^{Y T}}{Y} \, dt + \widetilde G \, dt
        \nonumber\\
       &\leq \ \left(2G + 4M + \frac{Z e^{Y T}}{Y} + \widetilde G\right)  dt + M \left(\ltot \gamma +  2\sqrt{\frac{3 T \ltot \gamma}{\varepsilon}}\right)  dt
       \, , \nonumber\\
       &\leq \ \left(2G + 4M + \frac{Ze^{YT}}{Y} + \widetilde G + M \ltot \gamma + 2M \sqrt{\frac{3 T \ltot \gamma}{\varepsilon}}\right)  dt 
       \, . \nonumber
\end{align}
Because $dt < \gamma^\frac{1}{4}$, we have
\begin{equation}
    \norminf{\Ei (t')} < W \gamma^\frac{1}{4} \,,
\end{equation}
where
\begin{equation}
    W = \left(2G + 4M + \frac{Z e^{YT}}{Y}+ \widetilde G + M \ltot \gamma + 2M \sqrt{\frac{3 T \ltot \gamma}{\varepsilon}}\right) \, .
\end{equation}
With the inequality
\begin{equation}
    \gamma \leq \frac{\delta^4}{2W} \, ,
\end{equation}
it follows for all {locales} $i$ and all times $t \in [0,T]$ that
\begin{equation}
    \norminf{\Ei(t)} \leq \delta 
\end{equation}
with probability larger than $1 - \varepsilon$.

\end{proof}


\section{Proof of the High-Frequency Constant-Strength Approximation (HFCSA) Theorem} \label{sec:HFCSA Proof}

In this section, we prove the HFCSA Theorem (see \cref{thm:HFCSA}).

\begin{theorem}[High-Frequency, Constant-Strength Approximation Theorem]\label{thm:HFCSA}
  Fix the relative interaction-frequency parameters $l_{ij}$, the interaction strength $\gamma > 0$, and a time $T$. Suppose that each local-dynamics function $g^{(i)}$ is Lipschitz continuous and 
  bounded (see \Cref{sec:LD}).
  Fix the arbitrary constants $\varepsilon \in (0,1]$, $\delta > 0$, and $\eta > 0$.
   For sufficiently large total-interaction-frequency parameter $\ltot$, each host microbiome abundance vector $\nsup{i}(t)$ satisfies
  \begin{equation}\label{eq:HFCSA bound}
      \normLinf{\nsup{i} - \widetilde{\bm{N}}}{[\eta,T]} < \delta
  \end{equation}
with probability larger than $1 - \varepsilon$, where
  \begin{align} 
         \frac{\der \widetilde{\bm{N}}}{\der t} &= \frac{1}{|H|} \sum_{j = 1}^{|H|} g^{(j)}\left(\widetilde{\bm{N}}\right) \,, \\
         \widetilde{\bm{N}}(0) &= \overline{\bm{N}}(0) \,. \nonumber
\end{align}
\end{theorem}

\begin{proof}
Each local-dynamics function $g^{(i)}$ is bounded (see \Cref{sec:LD}),
 so there exists a constant $M$ such that each entry of $\nsup{i}(t)$ is nonnegative and
\begin{equation}
 	\norminf{\nsup{i}(t)} \leq M   
\end{equation}
for each microbiome abundance vector $\nsup{i}(t)$ and all times $t \geq 0$.

The approximate microbiome abundance vector $\widetilde{\bm{N}}t)$ is the mean of all $\nsup{i}(t)$. 
The dynamics of $\widetilde{\bm{N}}(t)$ 
(see \cref{eq:HFCSA}) is given by the mean of the local dynamics of all hosts. Each local-dynamics function $\gi$ is bounded, so each of their means are also bounded.
 Therefore, for all times $t \geq 0$, each entry of $\widetilde{\bm{N}}(t)$ is nonnegative and
\begin{equation}
	 \norminf{\widetilde{\bm{N}}(t)} \leq M   \,.
\end{equation}

We assume that each local-dynamics function $\gi$ is Lipschitz continuous. Therefore, there exists a constant $C$ such that 
\begin{equation}\label{constant35}
    \norminf{g^{(i)}\left(\bm{x}\right) - g^{(i)}\left(\bm{y}\right)} \leq C \norminf{\bm{x} - \bm{y}}
\end{equation}
for each $\gi$ and all $\bm{x},\bm{y} \in[0,M]^{n}$. Each local-dynamics function $\gi$ is continuous (which follows from their Lipschitz continuity) and each microbiome abundance vector $\nsup{i}(t)$ is in the compact 
region $[0,M]^{n}$, so there exists a constant $G$ such that \begin{equation} \label{eq:LD Bound 3}
        \norminf{\frac{\der \nsup{i}}{\der t}} = \norminf{g^{(i)}\left(\nsup{i}\right)} \leq G 
\end{equation}
for all $\nsup{i}(t)$. For an interaction involving host \hsupa{i} that occurs at time $t_I$, we choose $\nsup{i}\left(t_I\right) = \nsup{i}\left(t_I^+\right)$. Therefore, $\nsup{i}(t)$ is right-continuous at time $t_I$. It is usually not left-continuous at time $t_I$, so it is usually not left-differentiable at time $t_I$.\footnote{The only situation where $\nsup{i}(t)$ is left-continuous at time $t_I$ occurs when the host \hsupa{j} with which \hsupa{i} interacts has a microbiome vector $\nsup{j}(t_I^-) = \nsup{i}(t_I^-)$.} In such situations, the derivative that we use in \cref{eq:LD Bound 3} is a right derivative.

Define the Dirichlet energy
\begin{equation}
    U(t) = \frac{1}{2} \sum_{i,j} \frac{\lambda_{ij}}{\ltot} \normtwo{\nsup{i}(t) - \nsup{j}(t)}^2 \,,
\end{equation}
which is nonnegative by construction. Let $\xi > 0$ be arbitrary but fixed. We will show that $U(t) \leq \xi$ with probability larger than $1 - \varepsilon$ for sufficiently large $\ltot$ and all times $t \in [\eta,T]$.

Between interactions,
\begin{equation}
    \frac{\der U}{\der t} =  \sum_{i,j} \frac{\lambda_{ij}}{\ltot} \left(\gi\left(\nsup{i}\right) -g^{(j)}\left(\nsup{j}\right)\right) \cdot \left(\nsup{i} - \nsup{j}\right) \,.
\end{equation}
Therefore,
\begin{align} \label{eq:dU bound}
    \normod{\frac{\der U }{\der t}} &\leq  \sum_{i,j} \frac{\lambda_{ij}}{\ltot} (2G) M  \\
    	&= 4 G M \nonumber \,.
\end{align}

Fix a $dt > 0$ such that ${T}/{dt}$ is an integer and
\begin{equation} \label{eq:deltabound}
    dt < \max\left\{\eta, \frac{\xi}{12 G M}, \frac{\delta}{2 C M e^{C T}} \right\} \,.
\end{equation}
Let $t_k = k \, dt$. There are only finitely many $t_k$ in the interval $[0,T]$. The probability that an interaction occurs precisely at any of these times $t_k$ is $0$. Therefore, for the remainder of this proof, we only consider interactions that occur at times 
$t \neq t_k$ for any $k$. Under this assumption,
\begin{equation}
    U(t_k^-) = U(t_k) 
\end{equation}
for each $t_k \in [0,T]$.

We now consider how the Dirichlet energy $U(t)$ changes during a time interval $[t_k,t_{k + 1}]$. Let $L_k$ be the number of interactions that occur in $(t_k,t_{k + 1})$. 
(This interval is open because no interactions occur at any $t_k$.) We denote the associated ordered set of interactions by $\{t_{k,l}\}_{l = 1}^{L_k}$. Additionally, we let $t_{k,0} = t_k$ and $t_{k,L_k + 1} = t_{k + 1}$, and we define $dt_{k,l} = t_{k,l} - t_{k,l - 1}$. For $l \in \{1,\ldots,L_k\}$, let
\begin{equation}
    W_{k,l} = U\left(t_{k,l}^+\right) - U\left(t_{k,l}^-\right) \,.
\end{equation}
The difference $W_{k,l}$ is the change in $U$ due to an interaction at time $t_{k,l}$.

We decompose the change in $U(t)$ during the interval $[t_k,t_k+1]$ by writing
\begin{align} \label{eq:Udecomp}
        U(t_{k+1}) - U(t_{k}) &= \sum_{l = 1}^{L_k + 1} \left[U\left(t_{k,l}^-\right) - U\left(t_{k,l - 1}^-\right)\right]  \\
        		&= \sum_{l = 1}^{L_k + 1} \left[U\left(t_{k,l}^-\right) - U\left(t_{k,l-1}^+\right)\right] + \sum_{l = 2}^{L_k + 1} \left[U\left(t_{k,l-1}^+\right) - U\left(t_{k,l-1}^-\right)\right]  \nonumber \\
        		&= \sum_{l = 1}^{L_k + 1} \left[U\left(t_{k,l}^-\right) - U\left(t_{k,l-1}^+\right)\right] + \sum_{l = 1}^{L_k} W_{k,l} \,.  \nonumber
    \end{align}
The magnitude of the first sum on the right-hand side of \cref{eq:Udecomp} has the upper bound
\begin{align}
    \normod{\sum_{l = 1}^{L_k+1} \left[U\left(t_{k,l}^-\right) - U\left(t_{k,l-1}^+\right)\right]} &\leq \sum_{l = 1}^{L_k+1}\normod{\left[U\left(t_{k,l}^-\right) - U\left(t_{k,l-1}^+\right)\right]} \\
    	&\leq \sum_{l=1}^{L_k + 1} 4 G M \, dt_{k,l} \nonumber \\
    	&= 4 G M \, dt \,. \nonumber
\end{align}

We now consider the sum $\sum_{l=1}^{L_k} W_{k,l}$. If the interaction at time $t_{k,l}$ is between hosts \hsupa{i} and \hsupa{j}, then
\begin{align}
    W_{k,l} &= \frac{\lambda_{ij}}{\ltot} \left[\normtwo{\nsup{i}\left(t_I^+\right) - \nsup{j}\left(t_I^+\right)}^2 - \normtwo{\nsup{i}\left(t_I^-\right) - \nsup{j}\left(t_I^-\right)}^2\right] \ \\
	    &= \frac{\lambda_{ij}}{\ltot} \left[\begin{aligned} 
	    	&\normtwo{(1 - \gamma) \nsup{i}\left(t_I^-\right) + \gamma \nsup{j}\left(t_I^-\right) - (1 - \gamma) \nsup{j}\left(t_I^-\right) - \gamma \nsup{i}\left(t_I^-\right)}^2 \ \\ 
	    &\quad- \normtwo{\nsup{i}\left(t_I^-\right) - \nsup{j}\left(t_I^-\right)}^2\end{aligned}\right] \nonumber  \\
	    &= \frac{\lambda_{ij}}{\ltot} \left[(1 - 2\gamma)^2 \normtwo{\nsup{i}\left(t_I^-\right) - \nsup{j}\left(t_I^-\right)}^2 - \normtwo{\nsup{i}\left(t_I^-\right) - \nsup{j}\left(t_I^-\right)}^2\right] \nonumber  \\
	    &= -4\gamma (1 - \gamma) \frac{\lambda_{ij}}{\ltot} \normtwo{\nsup{i}\left(t_I^-\right) - \nsup{j}\left(t_I^-\right)}^2\nonumber \\
	    &\leq 0 \,. \notag
\end{align}

We now show by contradiction that the Dirichlet energy $U(t) \leq {\xi}/{3}$ with probability at least $1 - \frac{\varepsilon \, dt}{2T}$ for sufficiently large total-interaction-frequency parameter $\ltot$ and some time $t \in [t_k,t_{k + 1}]$.

Suppose that $U(t) > {\xi}/{3}$ for all times $t \in [t_k,t_{k + 1}]$. For all times $t \in [t_k,t_{k + 1}]$, there are then some {locales} $i$ and $j$ such that
\begin{equation}
    \frac{\lambda_{ij}}{\ltot} \normtwo{\nsup{i}\left(t\right) - \nsup{j}\left(t\right)}^2 > \frac{\xi}{3 |H|^2} \,.
\end{equation}
Therefore, at each time $t_{k,l}^-$ immediately before an interaction, there at least one pair $i$ and $j$ such that
\begin{equation}
    \frac{\lambda_{ij}}{\ltot} \normtwo{\nsup{i}\left(t_{k,l}^-\right) - \nsup{j}\left(t_{k,l}^-\right)}^2 > \frac{\xi}{3 |H|^2} \,.
\end{equation}

Let
\begin{equation}
    l_\text{min} = \min_{i,j} \left\{ l_{ij} \mid l_{ij} > 0 \right\} \,.
\end{equation}
An interaction at time $t_{k,l}$ occurs between hosts \hsupa{i} and \hsupa{j} with probability $l_{ij} \geq l_\text{min}$. Therefore, regardless of the previous interactions, each $W_{k,l}$ satisfies
\begin{equation}
    W_{k,l} < -\frac{4 \gamma (1-\gamma) \xi}{3 |H|^2}
\end{equation}
with probability at least $l_\text{min}$.

We define {independent and identically distributed random variables $\mathcal{W}_l$} such that
\begin{align}
    \Pr\left({\mathcal{W}_l} = 0\right) &= 1 - l_\text{min} \,, \\
    \Pr\left({\mathcal{W}_l} = -\frac{4 \gamma (1 - \gamma) \xi}{3 |H|^2}\right) &= l_\text{min} \nonumber \,.
\end{align}
For all $w \leq 0$, we have
\begin{equation} \label{eq:sumWcomp}
    \Pr\left(\sum_{l = 1}^{L_k} W_{k,l} < w\right) \geq \Pr\left(\sum_{l = 1}^{L_k} {\mathcal{W}_l} < w\right)\,.
\end{equation}

We seek to bound
\begin{equation} \label{eq:sumW}
    \sum_{l = 1}^{L_k} {\mathcal{W}_l} \,.
\end{equation} 
Each random variable ${\mathcal{W}_l}$ is stochastic, so we can only find a bound for \cref{eq:sumW} that holds with some probability. The first two moments of ${\mathcal{W}_l}$ are
\begin{align}
    \E\left[{\mathcal{W}_l}\right] &= -l_\text{min} \frac{4 \gamma (1 - \gamma) \xi}{3 |H|^2} \,, \\
     \E\left[{\mathcal{W}_l}^2\right] &= l_\text{min} \left(\frac{4 \gamma (1 - \gamma) \xi}{3 |H|^2}\right)^2 \,. \nonumber
\end{align}
The expectation of the sum \cref{eq:sumW} is
\begin{align}
    \E \left[\sum_{l=1}^{L_k} {\mathcal{W}_l}\right] &= \E\left[\;\E\left[\sum_{l = 1}^{L_k} {\mathcal{W}_l} \;\middle|\; L_k\right]\;\right]  \\
    	&= \E \left[L_k \left(-l_\text{min} \frac{4 \gamma (1 - \gamma) \xi}{3 |H|^2}\right) \right] \nonumber\\
    	&= -\ltot  dt \, l_\text{min} \frac{4 \gamma (1 - \gamma) \xi}{3 |H|^2}\,. \nonumber
    \end{align}
Applying the law of total variance, the variance of the sum \cref{eq:sumW} is
\begin{align}
        \Var\left[\sum_{l = 1}^{L_k} {\mathcal{W}_l}\right] &= \E\left[\; \Var \left[\sum_{l = 1}^{L_k} {\mathcal{W}_l}\;\middle|\; L_k \right]\; \right] + \Var\left[ \; \E \left[\sum_{l = 1}^{L_k} {\mathcal{W}_l}\;\middle|\; L_k \right]\; \right] \\
        		&= \E\left[L_k \Var \left[{\mathcal{W}_l} \right]\; \right] + \Var\left[ \; L_k \E\left[{\mathcal{W}_l} \right]\; \right] \nonumber\\
        		&= \ltot   dt \, \Var\left[{\mathcal{W}_l} \right] + \ltot  dt \left(\E\left[{\mathcal{W}_l} \right]\right)^2 \nonumber\\
        		&=  \ltot  dt \, \E\left[{\mathcal{W}_l}^2\right] \nonumber\\
        		&=  \ltot  dt \, l_\text{min} \left(\frac{4 \gamma (1 - \gamma) \xi}{3 |H|^2}\right)^2  \,. \nonumber
\end{align}

Let $\alpha = \sqrt{\frac{2 T}{\varepsilon \, dt}}$. By Chebyshev's inequality,
\begin{align} \label{eq:chebW}
    \Pr\left(\;\sum_{l = 1}^{L_k} {\mathcal{W}_l} \geq  \E\left[\sum_{l = 1}^{L_k} {\mathcal{W}_l}\right] + \alpha \sqrt{\Var\left[\sum_{l = 1}^{L_k} {\mathcal{W}_l}\right]} \;\right) &\leq \frac{1}{\alpha^2} \,, \\
    \Pr\left(\;\sum_{l = 1}^{L_k} {\mathcal{W}_l} \geq -\ltot  dt  \, l_\text{min} \frac{4 \gamma (1 - \gamma) \xi}{3 |H|^2} + \sqrt{\frac{2 T \ltot l_\text{min}}{\varepsilon}} \frac{4 \gamma (1 - \gamma) \xi}{3 |H|^2} \;\right) &\leq \frac{\varepsilon \, dt}{2 T} \,, \nonumber\\
    \Pr\left(\;\sum_{l = 1}^{L_k} {\mathcal{W}_l} < -\left(\ltot  dt \, l_\text{min} - \sqrt{\frac{2 T \ltot l_\text{min}}{\varepsilon}} \right) \frac{4 \gamma (1-\gamma) \xi}{3 |H|^2} \;\right) 
    &> 1 - \frac{\varepsilon \, dt}{2 T} \,. \nonumber
\end{align}
Using the bound \cref{eq:sumWcomp} in \cref{eq:chebW} gives
\begin{align}
    \Pr\left(\;\sum_{l = 1}^{L_k} W_{k,l} < -\left(\ltot  dt  \, l_\text{min} - \sqrt{\frac{2 T \ltot l_\text{min}}{\varepsilon}} \right) \frac{4 \gamma (1 - \gamma) \xi}{3 |H|^2} \;\right) 
    &> 1- \frac{\varepsilon \, dt}{2 T} \,, \\
    \Pr\left(\;\sum_{l = 1}^{L_k} W_{k,l} < -\sqrt{\ltot} \left(\sqrt{\ltot} dt \, l_\text{min} - \sqrt{\frac{2 T l_\text{min}}{\varepsilon}} \right)\frac{4 \gamma (1-\gamma) \xi}{3 |H|^2} \;\right) 
    &> 1 - \frac{\varepsilon \, dt}{2 T} \,. \nonumber
\end{align}
By choosing sufficiently large $\ltot$, we can make $\sqrt{\ltot} \left(\sqrt{\ltot} dt \, l_\text{min} - \sqrt{\frac{2 T l_\text{min}}{\varepsilon}} \right)\frac{4 \gamma (1-\gamma) \xi}{3 |H|^2}$ arbitrarily large. In particular, we choose a sufficiently large $\ltot$ so that
\begin{equation}
    \sqrt{\ltot} \left(\sqrt{\ltot} dt  l_\text{min} - \sqrt{\frac{2 T l_\text{min}}{\varepsilon}} \right)\frac{4 \gamma (1 - \gamma) \xi}{3 |H|^2} \geq U(0) + 4 G M T + 4 G M \, dt \,.
\end{equation}
Therefore,
\begin{equation} \label{eq:sumWbound}
    \Pr\left(\;\sum_{l = 1}^{L_k} W_{k,l} < -\left[U(0) + 4 G M T+ 4 G M \, dt\right] \;\right) > 1 - \frac{\varepsilon \, dt}{2 T} \,.
\end{equation}

Local dynamics can cause the Dirichlet energy $U(t)$ to change at a rate of at most $4 G M$ per unit time (see \cref{eq:dU bound}). Additionally, because each $W_{k,l}$ is nonpositive, interactions cannot cause $U$ to increase. Therefore, for any time $t \in [0,T]$, we have $U(t) < U(0) + 4 G M T$. Combining this upper bound, the bound \cref{eq:sumWbound}, and the decomposition of $U(t_{k + 1}) - U(t_k)$ in \cref{eq:Udecomp} yields
\begin{align}
    U(t_{k + 1}) &\leq U(t_{k}) + \normod{\sum_{l = 1}^{L_k + 1} \left[U\left(t_{k,l}^-\right) - U\left(t_{k,l - 1}^+\right)\right]} + \normod{\sum_{l = 1}^{L_k} W_{k,l}} \\
    	&< U(0) + 4 G M T + 4 G M \, dt - \left[U(0) + 4 G M T + 4 G M \, dt\right] \nonumber \\
    	&< 0 
\end{align}
with probability greater than $1 - \frac{\varepsilon \, dt}{2 T}$. However, $U(t)$ is always nonnegative. Therefore, with probability greater than $1 - \frac{\varepsilon \, dt}{2 T}$, we have a contradiction and there is some time $t' \in [t_k,t_{k + 1}]$ such that $U(t') \leq {\xi}/{3}$. For some $l'$, we have $t' \in [t_{k,l'},t_{k,l'+1})$. Therefore, 
\begin{align}
    U(t_{k + 1}) &= U(t') + U(t_{k,l' + 1}) - U(t') + \sum_{l = l' + 2}^{L_k + 1} \left[U\left(t_{k,l}^-\right) - U\left(t_{k,l - 1}^+\right)\right] + \sum_{l = l' + 1}^{L_k} W_{k,l} \\
    	&\leq U(t') + \normod{U(t_{k,l' + 1}) - U(t')} + \sum_{l = l' + 2}^{L_k + 1} \normod{U\left(t_{k,l}^-\right) - U\left(t_{k,l - 1}^+\right)} + \normod{\sum_{l = l' + 1}^{L_k} W_{k,l}} \nonumber \\    
    	&\leq \frac{\xi}{3} + 4 G M \, dt_{k,l' + 1} + \sum_{l = l' + 2}^{L_k + 1} 4 G M \, dt_{k,l} \nonumber \\
    	&\leq \frac{\xi}{3} + 4 G M \, dt \,.
    \end{align}

Recall that $dt < \frac{\xi}{12 G M}$ (see \cref{eq:deltabound}). We have
\begin{align}
    U(t_{k + 1}) < \frac{\xi}{3} + \frac{\xi}{3} =  \frac{2\xi}{3} \nonumber \,.
\end{align}
With probability greater than $1 - \frac{\varepsilon \, dt}{2T}\left(\frac{T}{dt}\right) = 1 - \frac{\varepsilon}{2} > 1 - \varepsilon$, every $U(t_k)$ except $U(0)$ satisfies $U(t_k) < {2\xi}/{3}$. 
Consider an arbitrary time $t^{\text{\ding{44}}} \in [t_k,t_{k + 1}]$, where $k \geq 1$. For some $l^{\text{\ding{44}}}$, we have $t^{\text{\ding{44}}} \in [t_{k,l^{\text{\ding{44}}}},t_{k,l^{\text{\ding{44}}} + 1})$. Therefore,
\begin{equation}
    U(t^{\text{\ding{44}}}) = U(t^{\text{\ding{44}}})-U(t_{k,l^{\text{\ding{44}}}}) +  \sum_{l = 1}^{l^{\text{\ding{44}}}} \left[U\left(t_{k,l}^-\right) - U\left(t_{k,l-1}^+\right)\right] + \sum_{l = 1}^{l^{\text{\ding{44}}}} W_{k,l} + U(t_k) \,,
\end{equation}
which implies that
\begin{align} \label{eq:xiboundU}
    U(t^{\text{\ding{44}}}) & \leq \normod{U(t^{\text{\ding{44}}}) - U(t_{k,l^{\text{\ding{44}}}})} + \sum_{l = 1}^{l^{\text{\ding{44}}}} \normod{U\left(t_{k,l}^-\right) - U\left(t_{k,l - 1}^+\right)} + U(t_k) \\
    	&\leq 4 G M \, dt_{k,l^{\text{\ding{44}}} + 1} + \sum_{l = 1}^{l^{\text{\ding{44}}}} 4 G M \, dt_{k,l} + U(t_k) \nonumber\\
    	&\leq 4 G M \, dt + U(t_k) \nonumber\\
    	&< \frac{\xi}{3} + \frac{2\xi}{3} \nonumber\\
    	&= \xi \,. \nonumber
\end{align}

With probability larger than $1 - \varepsilon$, we have $U(t) < \xi$ for all times $t \in [dt,T]$. Because $\xi$ is arbitrary, by choosing a sufficiently large total-interaction-frequency parameter $\ltot$, we can make $U(t)$ arbitrarily 
small in the interval $[dt,T]$ with probability larger than $1 - \varepsilon$. Assume that $U(t) < \xi$ for all $t \in [dt,T]$. We can then bound the maximum difference between any microbiome abundance vectors $\nsup{i}(t)$ and $\nsup{j}(t)$. In particular, 
if {hosts} \hsupa{i} and \hsupa{j} are adjacent, then
\begin{align}
    \frac{\lambda_{ij}}{\ltot} \normtwo{\nsup{i}(t) - \nsup{j}(t)}^2 &< \xi \,, \\
    \normtwo{\nsup{i}(t) - \nsup{j}(t)} &< \sqrt{l_\text{max}\xi} \,, \nonumber
\end{align}
where $l_\text{max} = \max_{i,j}  l_{ij}$. A shortest path between any two hosts in the interaction network has length at most $|H| - 1$. Therefore, the 2-norm of the difference between each pair of microbiome abundance vectors $\nsup{i}(t)$ and $\nsup{j}(t)$ satisfies
\begin{equation}\label{eq:ninjbound}
    \normtwo{\nsup{i}(t) - \nsup{j}(t)} < \left(|H| - 1\right)\sqrt{l_\text{max} \, \xi} \,.
\end{equation}
The inequality \cref{eq:ninjbound} guarantees that the difference between each microbiome abundance vector $\nsup{i}(t)$ and the mean microbiome abundance vector $\overline{\bm{N}}(t)$ satisfies
\begin{align} \label{eq:ninmbound}
    \normtwo{\nsup{i}(t) - \overline{\bm{N}}(t)} &< \left(|H| - 1\right)\sqrt{l_\text{max}\,\xi} \,, \\
    \norminf{\nsup{i}(t) - \overline{\bm{N}}(t)} &< \left(|H| - 1\right)\sqrt{l_\text{max}\,\xi} \,. \notag
\end{align}

We now bound the magnitude of the difference between the mean microbiome abundance vector $\overline{\bm{N}}(t)$ and the approximate microbiome abundance vector $\widetilde{\bm{N}}(t)$ for all times $t \in [0,T]$. We refer to 
$\norminf{\overline{\bm{N}}(t) - \widetilde{\bm{N}}(t)}$ as the \emph{approximation--mean error} at time $t$. At time $0$, the approximation--mean error $\norminf{\overline{\bm{N}}(0) - \widetilde{\bm{N}}(0)} = 0$ by construction. We bound the approximation--mean error by first bounding its time derivative, which satisfies
\begin{align}
    \frac{\der}{\der t}\norminf{\overline{\bm{N}} - \widetilde{\bm{N}}} &\leq \norminf{\frac{\der}{\der t}\left(\overline{\bm{N}} - \widetilde{\bm{N}}\right)} \\
   	 &= \norminf{\frac{1}{|H|} \sum_{j = 1}^{|H|} g^{(j)}\left(\nsup{j}\right) - \frac{1}{|H|} \sum_{j = 1}^{|H|} g^{(j)}\left(\widetilde{\bm{N}}\right)} \nonumber \\
   	 &\leq \frac{1}{|H|} \sum_{j = 1}^{|H|} \norminf{g^{(j)}\left(\nsup{j}\right) - g^{(j)}\left(\widetilde{\bm{N}}\right)}\nonumber \\
   	 &\leq \frac{1}{|H|} \sum_{j = 1}^{|H|} C \norminf{\nsup{j} - \widetilde{\bm{N}}}\nonumber \\
   	 &\leq \frac{1}{|H|}\sum_{j = 1}^{|H|} C \norminf{\nsup{j} - \widetilde{\bm{N}}} \nonumber\\
   	 &\leq \frac{1}{|H|} \sum_{j = 1}^{|H|} C \left( \:\norminf{\nsup{j} - \overline{\bm{N}}} + \norminf{\overline{\bm{N}} - \widetilde{\bm{N}}}\:\right)\nonumber \\
   	 &\leq C \norminf{\overline{\bm{N}} - \widetilde{\bm{N}}} + \frac{1}{|H|} \sum_{j = 1}^{|H|} C \norminf{\nsup{j} - \overline{\bm{N}}}  \nonumber
\end{align}
for some constant $C$ (see \eqref{constant35}).

For times $t \in [dt,T]$, each $\norminf{\nsup{j}(t) - \overline{\bm{N}}(t)} < \left(|H| - 1\right) \sqrt{l_\text{max} \, \xi}$. Therefore, on this interval,
\begin{equation}
    \frac{\der}{\der t}\norminf{\overline{\bm{N}} - \widetilde{\bm{N}}} <  C \norminf{\overline{\bm{N}} - \widetilde{\bm{N}}} + \left(|H| - 1\right)\sqrt{l_\text{max} \, \xi} \,.
\end{equation}
For $t \in [0,dt]$, we obtain a weaker bound for the time derivative of the approximation--mean error:
\begin{align}
    \frac{\der}{\der t}\norminf{\overline{\bm{N}} - \widetilde{\bm{N}}} &\leq \frac{1}{|H|} \sum_{j = 1}^{|H|} C \norminf{\nsup{j} - \widetilde{\bm{N}}} \\
    &\leq \frac{1}{|H|} \sum_{j = 1}^{|H|} C M \nonumber \\
    &= C M \nonumber \, .
\end{align}
Therefore, for times $t \in [0,dt]$, the approximation--mean error satisfies
\begin{align}
     \norminf{\overline{\bm{N}}(t) - \widetilde{\bm{N}}(t)} &\leq \norminf{\overline{\bm{N}}(0) - \widetilde{\bm{N}}(0)} + C M t \\
    		&\leq C M \, dt \nonumber \\
    		&\leq \frac{\delta}{2 e^{C T}} \,. \nonumber
\end{align}
For times $t \in [dt,T]$, we construct a function $u(t)$ that gives an upper bound of the 
\linebreak
approximation--mean error. This function $u(t)$ is the solution of the dynamical system
\begin{align}
    \frac{\der u}{\der t} &= C u + \left(|H| - 1\right)\sqrt{l_\text{max} \, \xi} \,, \\
    u(dt) &= \frac{\delta}{2 e^{C T}} \,. \nonumber
\end{align}
Using an integrating factor to solve for $u(t)$ yields
\begin{align}
     u(t) = e^{C (t - dt)} \left(\frac{\delta}{2 e^{C T}}\right) + \frac{\left(|H| - 1\right) \sqrt{l_\text{max} \, \xi}}{C} \left(e^{C(t - dt)} - 1\right) \,. 
\end{align}

For $t \in [dt,T]$, the function $u(t)$ satisfies
\begin{align}
    u(t) &\leq e^{C T} \left(\frac{\delta}{2e^{C T}}\right) + \frac{\left(|H| - 1\right) \sqrt{l_\text{max} \,\xi}}{C} e^{C T} \\
   		 &= \frac{\delta}{2} + \frac{\left(|H| - 1\right) \sqrt{l_\text{max} \,\xi}}{C} e^{C T} \,. \nonumber
\end{align}
Therefore,
\begin{equation} \label{eq:nmnabound}
    \norminf{\overline{\bm{N}}(t) - \widetilde{\bm{N}}(t)} \leq \frac{\delta}{2} + \frac{\left(|H|-1\right)\sqrt{l_\text{max} \,\xi}}{C} e^{C T} \,.
\end{equation}
As we discussed previously (see \cref{eq:xiboundU}), we can make $\xi$ arbitrarily small by choosing a sufficiently large $\ltot$. Therefore, we make $\xi$ sufficiently small so that
\begin{equation} \label{eq:xibound}
    \left(|H| - 1\right) \sqrt{l_\text{max} \, \xi} < \max \left\{\frac{\delta}{4}, \frac{\delta C}{4 e^{C T}} \right\} \,.
\end{equation} 
Using the bound \cref{eq:xibound} in \cref{eq:nmnabound} and \cref{eq:ninmbound} yields
\begin{align}
    \norminf{\overline{\bm{N}}(t) - \widetilde{\bm{N}}(t)} &< \frac{\delta}{2} + \frac{\delta C}{4 e^{C T}} \left(\frac{e^{C T}}{C}\right) = \frac{3\delta}{4} \,, \\
    \norminf{\nsup{i}(t) - \overline{\bm{N}}(t)} &< \frac{\delta}{4} \notag
\end{align}
for any $t \in [dt,T]$ with probability larger than $1 - \varepsilon$. Because $dt \leq \eta$, for all times $t \in [\eta,T]$, the difference between each microbiome abundance vector $\nsup{i}(t)$ and the approximate microbiome abundance vector $\widetilde{\bm{N}}(t)$ satisfies
\begin{align}
    \norminf{\nsup{i}(t) - \widetilde{\bm{N}}(t)} &\leq \norminf{\nsup{i}(t) - \overline{\bm{N}}(t)} + \norminf{\overline{\bm{N}}(t) - \widetilde{\bm{N}}(t)} \\
  	  &< \frac{\delta}{4} + \frac{3\delta}{4} = \delta \nonumber
\end{align}
with probability larger than $1 - \varepsilon$.

\end{proof}


\section*{Acknowledgements}

We thank Kedar Karhadkar, Christopher Klausmeier, Thomas Koffel, Elena Litchman, Sarah Tymochko, and Jonas Wickman for helpful discussions. We also thank two anonymous referees for useful comments on the manuscript.





\end{document}